\documentclass[a4paper,draft]{amsart}

\usepackage{latexsym}\usepackage{ifthen}\usepackage[leqno]{amsmath}\usepackage{enumerate}\usepackage{calc}
\usepackage{hyphenat}\usepackage{amstext,amsbsy,amsopn,amsthm,amsgen,amsfonts,amscd,amsxtra,upref}
\usepackage{mathrsfs}\usepackage{euscript}\usepackage{amssymb}

\swapnumbers
\theoremstyle{plain}
\newtheorem{Thm}{Theorem}[section]
\newtheorem{Lem}[Thm]{Lemma}
\newtheorem{Cor}[Thm]{Corollary}
\newtheorem{Pro}[Thm]{Proposition}
\newtheorem{Prp}[Thm]{Properties}
\newtheorem{Sub}[Thm]{Sublemma}

\theoremstyle{definition}
\newtheorem{Def}[Thm]{Definition}
\newtheorem{Exm}[Thm]{Example}
\newtheorem{Exs}[Thm]{Examples}
\newtheorem{Prb}[Thm]{Problem}

\theoremstyle{remark}
\newtheorem{Rem}[Thm]{Remark}
\newtheorem{Rms}[Thm]{Remarks}
\newtheorem*{Com}{Commentary}

\newcommand{\myEmail}{piotr.niemiec@uj.edu.pl}
\newcommand{\myAddress}[1]{\noindent{}\ITE{\equal{#1}{}}{}{Piotr Niemiec\\{}}
   In\-sty\-tut Ma\-te\-ma\-ty\-ki\\{}Wy\-dzia\l{} Ma\-te\-ma\-ty\-ki i In\-for\-ma\-ty\-ki\\{}
   U\-ni\-wer\-sy\-tet Ja\-giel\-lo\'{n}\-ski\\{}ul. \L{}o\-ja\-sie\-wi\-cza 6\\{}
   30-348 Kra\-k\'{o}w\\{}Poland}
\newcommand{\myData}[1][Piotr Niemiec]{\author[P. Niemiec]{Piotr Niemiec}\address{\myAddress{#1}}
   \email{\myEmail}}

\newcommand{\RRR}{\mathbb{R}}

\newcommand{\ZZZ}{\mathbb{Z}}
\newcommand{\CCc}{\CMcal{C}}
\newcommand{\DDd}{\CMcal{D}}\newcommand{\FFf}{\CMcal{F}}
\newcommand{\GGg}{\CMcal{G}}
\newcommand{\LLl}{\CMcal{L}}

\newcommand{\RRr}{\CMcal{R}}

\newcommand{\WWw}{\CMcal{W}}

\newcommand{\mM}{\mathfrak{m}}

\newcommand{\SECT}[1]{\section{#1}\renewcommand{\theequation}{\arabic{section}-\arabic{equation}}
   \setcounter{equation}{0}}

\newcounter{help}
\newcommand{\ITE}[3]{\ifthenelse{#1}{#2}{#3}}\newcommand{\ITEE}[3]{\ITE{\equal{#1}{#2}}{#3}{}}

\newcommand{\card}{\operatorname{card}}

\newcommand{\const}{\operatorname{const}}

\newcommand{\dist}{\operatorname{dist}}

\newcommand{\id}{\operatorname{id}}

\newcommand{\Iso}{\operatorname{Iso}}
\newcommand{\Lip}{\operatorname{Lip}}
\newcommand{\Metr}{\operatorname{Metr}}

\newcommand{\leqsl}{\leqslant}\newcommand{\geqsl}{\geqslant}
\newcommand{\epsi}{\varepsilon}\newcommand{\varempty}{\varnothing}\newcommand{\dd}{\colon}

\newcommand{\TFCAE}{The following conditions are equivalent:}
\newcommand{\tfcae}{the following conditions are equivalent:}\newcommand{\iaoi}{if and only if}

\newcommand{\COR}[1]{Corollary~\textup{\ref{cor:#1}}}
\newcommand{\DEF}[1]{Definition~\textup{\ref{def:#1}}}

\newcommand{\LEM}[1]{Lemma~\textup{\ref{lem:#1}}}

\newcommand{\PRO}[1]{Proposition~\textup{\ref{pro:#1}}}

\newcommand{\REM}[1]{Remark~\textup{\ref{rem:#1}}}

\newcommand{\THM}[1]{Theorem~\textup{\ref{thm:#1}}}

\newenvironment{cor}[1]{\begin{Cor}\label{cor:#1}}{\end{Cor}}
\newenvironment{dfn}[1]{\begin{Def}\label{def:#1}}{\end{Def}}

\newenvironment{lem}[1]{\begin{Lem}\label{lem:#1}}{\end{Lem}}
\newenvironment{prb}[1]{\begin{Prb}\label{prb:#1}}{\end{Prb}}
\newenvironment{pro}[1]{\begin{Pro}\label{pro:#1}}{\end{Pro}}

\newenvironment{rem}[1]{\begin{Rem}\label{rem:#1}}{\end{Rem}}

\newenvironment{thm}[1]{\begin{Thm}\label{thm:#1}}{\end{Thm}}

\newcommand{\hULL}[1]{\textup{\textsf{H}}_2(#1)} 
\newcommand{\dOUBLE}[2][p]{#2 \times_{#1} \{-1,1\}} 

\newcommand{\bibITEM}[2]{\ITE{\equal{#2}{}}{\bibitem{#1} }{\bibitem[#2]{#1} }}
\newcommand{\BIB}[8]{
   \bibITEM{#1}{#8} #2, \textit{#3}, #4{} \textbf{#5} (#6), #7.}
\newcommand{\myBIB}[7][P. Niemiec]{\ITE{\equal{#7}{*}\or\equal{#7}{**}}{}{#1, \textit{#2}, }
   #3{}\ITE{\equal{#4}{}}{}{ \textbf{#4}} (#5), #6\ITE{\equal{#7}{*}}{}{.}}
\newcommand{\BIb}[6]{
   \bibITEM{#1}{#6} #2, \textit{#3}, #4, #5.}
\newcommand{\BiB}[9]{
   \bibITEM{#1}{#9} #2, \textit{#3}, #4{} \textit{#5}, #6, #7, #8.}

\newcommand{\myBAPP}[4][P. Niemiec]{
   \ITE{\equal{#4}{*}}{}{#1, \textit{#2}, }#3}
\newcommand{\oNlINE}[2]{\ITEE{#1}{.}{\\#2}\ITEE{#1}{}{{} #2}}

\newcommand{\jRN}[2][]{
   \ITEE{#2}{AbhHamburg}{\ITE{\equal{#1}{+}}
      {Abh. Math. Sem. Hamburg}{Abh. Math. Sem. Hamburg}}
   \ITEE{#2}{ActaM}{\ITE{\equal{#1}{+}}
      {Acta Mathematica}{Acta Math.}}
   \ITEE{#2}{ActaMSinES}{\ITE{\equal{#1}{+}}
      {Acta Mathematica Sinica (English Series)}{Acta Math. Sin. (Engl. Ser.)}}
   \ITEE{#2}{AdvM}{\ITE{\equal{#1}{+}}
      {Advances in Mathematics}{Adv. Math.}}
   \ITEE{#2}{ACS}{\ITE{\equal{#1}{+}}
      {Applied Categorical Structures}{Appl. Categ. Structures}}
   \ITEE{#2}{ActaSM}{\ITE{\equal{#1}{+}}
      {Acta Scientiarum Mathematicarum (Szeged)}{Acta Sci. Math. (Szeged)}}
   \ITEE{#2}{AmJM}{\ITE{\equal{#1}{+}}
      {American Journal of Mathematics}{Amer. J. Math.}}
   \ITEE{#2}{AmMMon}{\ITE{\equal{#1}{+}}
      {The American Mathematical Monthly}{Amer. Math. Monthly}}
   \ITEE{#2}{AnnSciEcNormSupT}{\ITE{\equal{#1}{+}}
      {Annales Scientifiques de l'\'{E}cole Normale Sup\'{e}rieure (3)}
      {Ann. Sci. \'{E}c. Norm. Sup\'{e}r. (3)}}
   \ITEE{#2}{AnnM}{\ITE{\equal{#1}{+}}
      {Annals of Mathematics}{Ann. Math.}}
   \ITEE{#2}{AnnProb}{\ITE{\equal{#1}{+}}
      {The Annals of Probability}{Ann. Probab.}}
   \ITEE{#2}{AnnPALog}{\ITE{\equal{#1}{+}}
      {Annals of Pure and Applied Logic}{Ann. Pure Appl. Logic}}
   \ITEE{#2}{APM}{\ITE{\equal{#1}{+}}
      {Annales Polonici Mathematici}{Ann. Polon. Math.}}
   \ITEE{#2}{ArchM}{\ITE{\equal{#1}{+}}
      {Archiv der Mathematik}{Arch. Math.}}
   \ITEE{#2}{AttiAccLincRendNat}{\ITE{\equal{#1}{+}}
      {Atti della Accademia Nazionale dei Lincei. Rendiconti. Classe di Scienze Fisiche, 
      Matematiche e Naturali}{Atti Accad. Naz. Lincei Rend. Cl. Sci. Fis. Mat. Nat.}}
   \ITEE{#2}{BAMS}{\ITE{\equal{#1}{+}}
      {Bulletin of the American Mathematical Society}{Bull. Amer. Math. Soc.}}
   \ITEE{#2}{BAustrMS}{\ITE{\equal{#1}{+}}
      {Bulletin of the Australian Mathematical Society}{Bull. Austral. Math. Soc.}}
   \ITEE{#2}{BLondMS}{\ITE{\equal{#1}{+}}
      {Bulletin of the London Mathematical Sociecy}{Bull. Lond. Math. Soc.}}
   \ITEE{#2}{BAPolSSSM}{\ITE{\equal{#1}{+}}
      {Bulletin de l'Acad\'{e}mie Polonaise des Sciences. S\'{e}rie des Sciences 
      Math\'{e}matiques}{Bull. Acad. Pol. Sci. S\'{e}r. Sci. Math.}}
   \ITEE{#2}{BullSM}{\ITE{\equal{#1}{+}}
      {Bulletin des Sciences Math\'{e}matiques}{Bull. Sci. Math.}}
   \ITEE{#2}{BPAS}{\ITE{\equal{#1}{+}}
      {Bulletin of the Polish Academy of Sciences: Mathematics}{Bull. Pol. Acad. Sci. Math.}}
   \ITEE{#2}{CanadJM}{\ITE{\equal{#1}{+}}
      {Canadian Journal Mathematics}{Canad. J. Math.}}
   \ITEE{#2}{CollectM}{\ITE{\equal{#1}{+}}
      {Collectanea Mathematica}{Collect. Math.}}
   \ITEE{#2}{CMUC}{\ITE{\equal{#1}{+}}
      {Commentationes Mathematicae Universitatis Carolinae}{Comment. Math. Univ. Carolin.}}
   \ITEE{#2}{CRParis}{\ITE{\equal{#1}{+}}
      {C. R. Paris}{C. R. Paris}}
   \ITEE{#2}{CRASParis}{\ITE{\equal{#1}{+}}
      {Comptes Rendus de l'Acad\'{e}mie des Sciences. Paris}{C. R. Acad. Sci. Paris}}
   \ITEE{#2}{CEurJM}{\ITE{\equal{#1}{+}}
      {Central European Journal of Mathematics}{Cent. Eur. J. Math.}}
   \ITEE{#2}{CMHelv}{\ITE{\equal{#1}{+}}
      {Commentarii Mathematici Helvetici}{Comment. Math. Helv.}}
   \ITEE{#2}{CollM}{\ITE{\equal{#1}{+}}
      {Colloquium Mathematicum}{Coll. Math.}}
   \ITEE{#2}{CollMSJBoly}{\ITE{\equal{#1}{+}}
      {Colloquia Math. Soc. Janos Bolyai}{Colloq. Math. Soc. Janos Bolyai}}
   \ITEE{#2}{ComposM}{\ITE{\equal{#1}{+}}
      {Compositio Mathematica}{Compos. Math.}}
   \ITEE{#2}{CzMJ}{\ITE{\equal{#1}{+}}
      {Czechoslovak Mathematical Journal}{Czech. Math. J.}}
   \ITEE{#2}{DissM}{\ITE{\equal{#1}{+}}
      {Dissertationes Mathematicae (Roz\-pra\-wy Ma\-te\-ma\-tycz\-ne)}
      {Dissertationes Math. (Roz\-pra\-wy Mat.)}}
   \ITEE{#2}{DANSSSR}{\ITE{\equal{#1}{+}}
      {Doklady Akademii Nauk SSSR}{Dokl. Akad. Nauk SSSR}}
   \ITEE{#2}{DMJ}{\ITE{\equal{#1}{+}}
      {Duke Mathematical Journal}{Duke Math. J.}}
   \ITEE{#2}{ELA}{\ITE{\equal{#1}{+}}
      {The Electronic Journal of Linear Algebra}{Electron. J. Linear Algebra}}
   \ITEE{#2}{ExtrM}{\ITE{\equal{#1}{+}}
      {Extracta Mathematicae}{Extracta Math.}}
   \ITEE{#2}{FM}{\ITE{\equal{#1}{+}}
      {Fundamenta Mathematicae}{Fund. Math.}}
   \ITEE{#2}{FAA}{\ITE{\equal{#1}{+}}
      {Functional Analysis and its Applications}{Funct. Anal. Appl.}}
   \ITEE{#2}{FunkAnalPril}{\ITE{\equal{#1}{+}}
      {Funktsional'ny\u{\i} Analiz i Ego Prilozheniya}{Funkts. Anal. Prilozh.}}
   \ITEE{#2}{GTopA}{\ITE{\equal{#1}{+}}
      {General Topology and its Applications}{General Topol. Appl.}}
   \ITEE{#2}{HJM}{\ITE{\equal{#1}{+}}
      {Houston Journal of Mathematics}{Houston J. Math.}}
   \ITEE{#2}{IllinoisJM}{\ITE{\equal{#1}{+}}
      {Illinois Journal of Mathematics}{Illinois J. Math.}}
   \ITEE{#2}{IndagMP}{\ITE{\equal{#1}{+}}
      {Indagationes Mathematicae (Proceedings)}{Indagationes Math. Proc.}}
   \ITEE{#2}{IndianaUMJ}{\ITE{\equal{#1}{+}}
      {Indiana University Mathematical Journal}{Indiana Univ. Math. J.}}
   \ITEE{#2}{InHauEtSPM}{\ITE{\equal{#1}{+}}
      {Inst. Hautes \'{E}tudes Sci. Publ. Math.}{Inst. Hautes \'{E}tudes Sci. Publ. Math.}}
   \ITEE{#2}{IEOT}{\ITE{\equal{#1}{+}}
      {Integral Equations and Operator Theory}{Integral Equations Operator Theory}}
   \ITEE{#2}{InterJM}{\ITE{\equal{#1}{+}}
      {International Journal of Mathematics}{Internat. J. Math.}}
   \ITEE{#2}{IsraelJM}{\ITE{\equal{#1}{+}}
      {Israel Journal of Mathematics}{Israel J. Math.}}
   \ITEE{#2}{JAT}{\ITE{\equal{#1}{+}}
      {Journal of Approximation Theory}{J. Approx. Theory}}
   \ITEE{#2}{JAusMSA}{\ITE{\equal{#1}{+}}
      {Journal of the Australian Mathematical Society. Series A}{J. Aust. Math. Soc. Ser. A}}
   \ITEE{#2}{JCA}{\ITE{\equal{#1}{+}}
      {Journal of Convex Analysis}{J. Convex Anal.}}
   \ITEE{#2}{JChinUST}{\ITE{\equal{#1}{+}}
      {J. China Univ. Sci. Tech.}{J. China Univ. Sci. Tech.}}
   \ITEE{#2}{JFA}{\ITE{\equal{#1}{+}}
      {Journal of Functional Analysis}{J. Funct. Anal.}}
   \ITEE{#2}{JKoreanMS}{\ITE{\equal{#1}{+}}
      {Journal of the Korean Mathematical Society}{J. Korean Math. Soc.}}
   \ITEE{#2}{JLieTh}{\ITE{\equal{#1}{+}}
      {Journal of Lie Theory}{J. Lie Theory}}
   \ITEE{#2}{JMAnApp}{\ITE{\equal{#1}{+}}
      {J. Math. Anal. Appl.}{J. Math. Anal. Appl.}}
   \ITEE{#2}{JLondMS}{\ITE{\equal{#1}{+}}
      {Journal of the London Mathematical Society}{J. London Math. Soc.}}
   \ITEE{#2}{JMPuApNS}{\ITE{\equal{#1}{+}}
      {J. Math. Pures Appl., N. S.}{J. Math. Pures Appl., N. S.}}
   \ITEE{#2}{JOT}{\ITE{\equal{#1}{+}}
      {Journal of Operator Theory}{J. Operator Theory}}
   \ITEE{#2}{JReinAngM}{\ITE{\equal{#1}{+}}
      {Journal f\"{u}r die reine und angewandte Mathematik}{J. Reine Angew. Math.}}
   \ITEE{#2}{KodaiMSemRep}{\ITE{\equal{#1}{+}}
      {Kodai Math. Sem. Rep.}{Kodai Math. Sem. Rep.}}
   \ITEE{#2}{LAA}{\ITE{\equal{#1}{+}}
      {Linear Algebra and its Applications}{Linear Algebra Appl.}}
   \ITEE{#2}{LMLA}{\ITE{\equal{#1}{+}}
      {Linear and Multilinear Algebra}{Linear Multilinear Algebra}}
   \ITEE{#2}{LNM}{\ITE{\equal{#1}{+}}
      {Lecture Notes in Mathematics}{Lecture Notes in Math.}}
   \ITEE{#2}{MathJap}{\ITE{\equal{#1}{+}}
      {Math. Japon.}{Math. Japon.}}
   \ITEE{#2}{MLQ}{\ITE{\equal{#1}{+}}
      {Mathematical Logic Quarterly}{Math. Log. Q.}}
   \ITEE{#2}{MNotes}{\ITE{\equal{#1}{+}}
      {Math. Notes}{Math. Notes}}
   \ITEE{#2}{MProcCambPhS}{\ITE{\equal{#1}{+}}
      {Mathematical Proceedings of the Cambridge Philosophical Society}
      {Math. Proc. Cambridge Phil. Soc.}}
   \ITEE{#2}{MMag}{\ITE{\equal{#1}{+}}
      {Mathematics Magazine}{Math. Mag.}}
   \ITEE{#2}{MSb}{\ITE{\equal{#1}{+}}
      {Matematicheski\u{\i} Sbornik}{Mat. Sb.}}
   \ITEE{#2}{MStud}{\ITE{\equal{#1}{+}}
      {Matematychni Studi\"{\i}}{Mat. Stud.}}
   \ITEE{#2}{MScand}{\ITE{\equal{#1}{+}}
      {Mathematica Scandinavica}{Math. Scand.}}
   \ITEE{#2}{MAnn}{\ITE{\equal{#1}{+}}
      {Mathematische Annalen}{Math. Ann.}}
   \ITEE{#2}{MAMS}{\ITE{\equal{#1}{+}}
      {Memoirs of the American Mathematical Society}{Mem. Amer. Math. Soc.}}
   \ITEE{#2}{MichMJ}{\ITE{\equal{#1}{+}}
      {Michigan Mathematical Journal}{Mich. Math. J.}}
   \ITEE{#2}{MonatM}{\ITE{\equal{#1}{+}}
      {Monatshefte f\"{u}r Mathematik}{Mh. Math.}}
   \ITEE{#2}{MZ}{\ITE{\equal{#1}{+}}
      {Math. Z.}{Math. Z.}}
   \ITEE{#2}{MZamet}{\ITE{\equal{#1}{+}}
      {Mat. Zametki}{Mat. Zametki}}
   \ITEE{#2}{NonlinA}{\ITE{\equal{#1}{+}}
      {Nonlinear Analysis: Theory, Methods \& Applications}{Nonlinear Anal.}}
   \ITEE{#2}{NAMS}{\ITE{\equal{#1}{+}}
      {Notices of the American Mathematical Society}{Notices Amer. Math. Soc.}}
   \ITEE{#2}{OpusM}{\ITE{\equal{#1}{+}}
      {Opuscula Mathematica}{Opuscula Math.}}
   \ITEE{#2}{PacJM}{\ITE{\equal{#1}{+}}
      {Pacific Journal of Mathematics}{Pacific J. Math.}}
   \ITEE{#2}{PeriodMHung}{\ITE{\equal{#1}{+}}
      {Periodica Mathematica Hungarica}{Period. Math. Hungarica}}
   \ITEE{#2}{PAMS}{\ITE{\equal{#1}{+}}
      {Proceedings of the American Mathematical Society}{Proc. Amer. Math. Soc.}}
   \ITEE{#2}{ProcCambPhS}{\ITE{\equal{#1}{+}}
      {Proceedings of the Cambridge Philosophical Society}{Proc. Cambridge Phil. Soc.}}
   \ITEE{#2}{ProcImpAcadTokyo}{\ITE{\equal{#1}{+}}
      {Proc. Imp. Acad. Tokyo}{Proc. Imp. Acad. Tokyo}}
   \ITEE{#2}{ProcKonink}{\ITE{\equal{#1}{+}}
      {Proceedings of the Koninklijke Nederlandse Akademie van Wetenschappen}
      {Nederl. Akad. Wetensch. Proc. Ser. A}}
   \ITEE{#2}{PLondMS}{\ITE{\equal{#1}{+}}
      {Proceedings of the London Mathematical Society}{Proc. London Math. Soc.}}
   \ITEE{#2}{PNAS}{\ITE{\equal{#1}{+}}
      {Proceedings of the National Academy of Sciences of the United States of America}
      {Proc. Natl. Acad. Sci. USA}}
   \ITEE{#2}{PublRIMSKyoto}{\ITE{\equal{#1}{+}}
      {Publ. Res. Inst. Math. Sci. Kyoto Univ.}{Publ. Res. Inst. Math. Sci.}}
   \ITEE{#2}{PublSUAA}{\ITE{\equal{#1}{+}}
      {Publ. Sci. Univ. Alger. S\'{e}r. A}{Publ. Sci. Univ. Alger. S\'{e}r. A}}
   \ITEE{#2}{PWN}{\ITE{\equal{#1}{+}}
      {PWN -- Polish Scientific Publishers, Warszawa}
      {PWN -- Polish Scientific Publishers, Warszawa}}
   \ITEE{#2}{RCMP}{\ITE{\equal{#1}{+}}
      {Rendiconti del Circolo Matematico di Palermo}{Rend. Circ. Mat. Palermo}}
   \ITEE{#2}{RussMS}{\ITE{\equal{#1}{+}}
      {Russian Mathematical Surveys}{Russian Math. Surveys}}
   \ITEE{#2}{SbM}{\ITE{\equal{#1}{+}}
      {Sbornik: Mathematics}{Sb. Math.}}
   \ITEE{#2}{SciRepTokyoA}{\ITE{\equal{#1}{+}}
      {Science Reports of Tokyo Kyoiku Daigaku, Section A}
      {Sci. Rep. Tokyo Kyoiku Daigaku Sect. A}}
   \ITEE{#2}{SeminProbStras}{\ITE{\equal{#1}{+}}
      {S\'{e}minaire de probabilit\'{e}s de Strasbourg}{S\'{e}min. Prob. Strasbourg}}
   \ITEE{#2}{SIAMJMAA}{\ITE{\equal{#1}{+}}
      {SIAM Journal on Matrix Analysis and Applications}{SIAM J. Matrix Anal. Appl.}}
   \ITEE{#2}{SibirMZ}{\ITE{\equal{#1}{+}}
      {Sibirski\v{\i} Mat. \v{Z}hurnal}{Sibirsk. Mat. \v{Z}.}}
   \ITEE{#2}{SM}{\ITE{\equal{#1}{+}}
      {Studia Mathematica}{Studia Math.}}
   \ITEE{#2}{TAMS}{\ITE{\equal{#1}{+}}
      {Transactions of the American Mathematical Society}{Trans. Amer. Math. Soc.}}
   \ITEE{#2}{TohokuMJ}{\ITE{\equal{#1}{+}}
      {T\^{o}hoku Mathematical Journal}{T\^{o}hoku Math. J.}}
   \ITEE{#2}{TomskUnivRev}{\ITE{\equal{#1}{+}}
      {Tomsk Universitet Review}{Tomsk. Univ. Rev.}}
   \ITEE{#2}{TopA}{\ITE{\equal{#1}{+}}
      {Topology and its Applications}{Topology Appl.}}
   \ITEE{#2}{TMNA}{\ITE{\equal{#1}{+}}
      {Topological Methods in Nonlinear Analysis}{Topol. Methods Nonlinear Anal.}}
   \ITEE{#2}{TsukubaJM}{\ITE{\equal{#1}{+}}
      {Tsukuba Journal of Mathematics}{Tsukuba J. Math.}}
   \ITEE{#2}{UspekhiMN}{\ITE{\equal{#1}{+}}
      {Uspekhi Matem. Nauk}{Uspekhi Mat. Nauk}}
   }

\newcommand{\paplist}[3][]{
   \ITEE{#3}{HAbels,AManoussos,GNoskov2011}{
      \BIB{#2}{H. Abels, A. Manoussos, G. Noskov}
         {Proper actions and proper invariant metrics}
         {\jRN{JLondMS} (2)}{83}{2011}{619--636}{#1}}
   \ITEE{#3}{NIAkhiezer,IMGlazman1993}{
      \BIb{#2}{N.I. Akhiezer and I.M. Glazman}
         {Theory of Linear Operators in Hilbert Space}
         {Dover Publications, Inc., New York}{1993}{#1}}
   \ITEE{#3}{ASAmitsur,JLevitzki1950}{
      \BIB{#2}{A.S. Amitsur and J. Levitzki}
         {Minimal identities for algebras}
         {\jRN{PAMS}}{1}{1950}{449--463}{#1}}
   \ITEE{#3}{RDAnderson1966}{
      \BIB{#2}{R.D. Anderson}
         {Hilbert space is homeomorphic to the countable infinite product of lines}
         {\jRN{BAMS}}{72}{1966}{515--519}{#1}}
   \ITEE{#3}{RDAnderson1967}{
      \BIB{#2}{R.D. Anderson}
         {On topological infinite deficiency}
         {\jRN{MichMJ}}{14}{1967}{365--383}{#1}}
   \ITEE{#3}{RDAnderson,JMcCharen1970}{
      \BIB{#2}{R.D. Anderson and J. McCharen}
         {On extending homeomorphisms to Fr\'{e}chet manifolds}
         {\jRN{PAMS}}{25}{1970}{283--289}{#1}}
   \ITEE{#3}{RDAnderson,DWCurtis,JVanMill1982}{
      \BIB{#2}{R.D. Anderson, D.W. Curtis, J. van Mill}
         {A fake topological Hilbert space}
         {\jRN{TAMS}}{272}{1982}{311--321}{#1}}
   \ITEE{#3}{RJArchbold1995}{
      \BIB{#2}{R.J. Archbold}
         {On residually finite\hyp{}dimensional $\CCc^*$-algebras}
         {\jRN{PAMS}}{123}{1995}{2935--2937}{#1}}
   \ITEE{#3}{RArens,JEells1956}{
      \BIB{#2}{R. Arens and J. Eells}
         {On embedding uniform and topological spaces}
         {\jRN{PacJM}}{6}{1956}{397--403}{#1}}
   \ITEE{#3}{AVArhangelskii2002}{
      \BIB{#2}{A.V. Arhangel'skii}
         {The Hewitt\hyp{}Nachbin completion in topological algebra. Some effects of homogeneity}
         {\jRN{ACS}}{10}{2002}{267--278}{#1}}
   \ITEE{#3}{AVArhangelskii,MGTkachenko2008}{
      \BIb{#2}{A.V. Arhangel'skii and M.G. Tkachenko}
         {Topological Groups and Related Structures}
         {Atlantis Press, Paris; World Scientific, Hackensack, NJ}{2008}{#1}}
   \ITEE{#3}{NAronszajn,PPanitchpakdi1956}{
      \BIB{#2}{N. Aronszajn and P. Panitchpakdi}
         {Extension of uniformly continuous transformations and hyperconvex metric spaces}
         {\jRN{PacJM}}{6}{1956}{405--439}{#1}}
   \ITEE{#3}{KJBabenko1948}{
      \BIB{#2}{K.J. Babenko}
         {On conjugate functions}
         {\jRN{DANSSSR}}{62}{1948}{157--160}{#1}}
   \ITEE{#3}{TBanakh1995}{
      \BIB{#2}{T.O. Banakh}
         {Topology of spaces of probability measures, I}
         {\jRN{MStud}}{5}{1995}{65--87 (Russian)}{#1}}
   \ITEE{#3}{TBanakh1995a}{
      \BIB{#2}{T.O. Banakh}
         {Topology of spaces of probability measures, II}
         {\jRN{MStud}}{5}{1995}{88--106 (Russian)}{#1}}
   \ITEE{#3}{TBanakh1998}{
      \BIB{#2}{T. Banakh}
         {Characterization of spaces admitting a homotopy dense embedding into a Hilbert manifold}
         {\jRN{TopA}}{86}{1998}{123--131}{#1}}
   \ITEE{#3}{TBanakh,CzBessaga2000}{
      \BIB{#2}{T. Banakh and Cz. Bessaga}
         {On linear operators extending [pseudo]metrics}
         {\jRN{BPAS}}{48}{2000}{35--49}{#1}}
   \ITEE{#3}{TBanakh,TNRadul1997}{
      \BIB{#2}{T.O. Banakh and T.N. Radul}
         {Topology of spaces of probability measures}
         {\jRN{SbM}}{188}{1997}{973--995}{#1}}
   \ITEE{#3}{TBanakh,TRadul,MZarichnyi1996}{
      \BIb{#2}{T. Banakh, T. Radul, M. Zarichnyi}
         {Absorbing sets in infinite\hyp{}dimensional manifolds}
         {VNTL Publishers, Lviv}{1996}{#1}}
   \ITEE{#3}{TBanakh,IZarichnyy2008}{
      \BIB{#2}{T. Banakh and I. Zarichnyy}
         {Topological groups and convex sets homeomorphic to non\hyp{}separable Hilbert spaces}
         {\jRN{CEurJM}}{6}{2008}{77--86}{#1}}
   \ITEE{#3}{HBecker,ASKechris1996}{
      \BIb{#2}{H. Becker and A.S. Kechris}{The Descriptive Set Theory of Polish Group Actions 
         \textup{(London Math. Soc. Lecture Note Series, vol. 232)}}
         {University Press, Cambridge}{1996}{#1}}
   \ITEE{#3}{GBeer1993}{
      \BIb{#2}{G. Beer}
         {Topologies on Closed and Closed Convex Sets \textup{(Mathematics and Its Applications)}}
         {Kluwer Academic Publishers, Dordrecht}{1993}{#1}}
   \ITEE{#3}{NEBenamara,NNikolski1999}{
      \BIB{#2}{N.E. Benamara and N. Nikolski}
         {Resolvent tests for similarity to a normal operator}
         {\jRN{PLondMS}}{78}{1999}{585--626}{#1}}
   \ITEE{#3}{YBenyamini,JLindenstrauss2000}{
      \BIb{#2}{Y. Benyamini and J. Lindenstrauss}
         {Geometric nonlinear functional analysis I}
         {AMS Colloquium Publications 48}{2000}{#1}}
   \ITEE{#3}{SKBerberian1974}{
      \BIb{#2}{S.K. Berberian}
         {Lectures in Functional Analysis and Operator Theory}
         {Graduate Texts in Mathematics 15, Springer\hyp{}Verlag, New York}{1974}{#1}}
   \ITEE{#3}{SNBernstein1954}{
      \BIb{#2}{S.N. Bernstein}
         {Collected Works II}
         {Akad. Nauk SSSR, Moscow}{1954 (Russian)}{#1}}
   \ITEE{#3}{CzBessaga,APelczynski1972}{
      \BIB{#2}{Cz. Bessaga and A. Pe\l{}czy\'{n}ski}
         {On spaces of measurable functions}
         {\jRN{SM}}{44}{1972}{597--615}{#1}}
   \ITEE{#3}{CzBessaga,APelczynski1975}{
      \BIb{#2}{Cz. Bessaga and A. Pe\l{}czy\'{n}ski}
         {Selected topics in infinite\hyp{}dimensional topology}
         {\jRN{PWN}}{1975}{#1}}
   \ITEE{#3}{MBestvina,JMogilski1986}{
      \BIB{#2}{M. Bestvina and J. Mogilski}
         {Characterizing certain incomplete infinite\hyp{}dimensional absolute retracts}
         {\jRN{MichMJ}}{33}{1986}{291--313}{#1}}
   \ITEE{#3}{MBestvina,PBowers,JMogilsky,JWalsh1986}{
      \BIB{#2}{M. Bestvina, P. Bowers, J. Mogilsky, J. Walsh}
         {Characterization of Hilbert space manifolds revisited}
         {\jRN{TopA}}{24}{1986}{53--69}{#1}}
   \ITEE{#3}{RBhatia1997}{
      \BIb{#2}{R. Bhatia}
         {Matrix Analysis}
         {Springer, New York}{1997}{#1}}
   \ITEE{#3}{GBirkhoff1936}{
      \BIB{#2}{G. Birkhoff}
         {A note on topological groups}
         {\jRN{ComposM}}{3}{1936}{427--430}{#1}}
   \ITEE{#3}{MSBirman,MZSolomjak1987}{
      \BIb{#2}{M.S. Birman and M.Z. Solomjak}
         {Spectral Theory of Self\hyp{}Adjoint Operators in Hilbert Space}
         {D. Reidel Publishing Co., Dordrecht}{1987}{#1}}
   \ITEE{#3}{EBishop1961}{
      \BIB{#2}{E. Bishop}
         {A generalization of the Stone\hyp{}Weierstrass theorem}
         {\jRN{PacJM}}{11}{1961}{777--783}{#1}}
   \ITEE{#3}{BBlackadar2006}{\BIb{#2}{B. Blackadar}{Operator Algebras. 
         Theory of $\CCc^*$\hyp{}algebras and von Neumann algebras \textup{(Encyclopaedia 
         of Mathematical Sciences, vol. 122: Operator Algebras and Non\hyp{}Commutative Geometry 
         III)}}{Springer\hyp{}Verlag, Berlin\hyp{}Heidelberg}{2006}{#1}}
   \ITEE{#3}{JBlass,WHolsztynski1972}{
      \BIB{#2}{J. Blass and W. Holszty\'{n}ski}
         {Cubical polyhedra and homotopy III}
         {\jRN{AttiAccLincRendNat}}{53}{1972}{275--279}{#1}}
   \ITEE{#3}{FFBonsall,NJDuncan1973}{
      \BIb{#2}{F.F. Bonsall and N.J. Duncan}
         {Complete Normed Algebras}
         {Springer Verlag, Berlin}{1973}{#1}}
   \ITEE{#3}{NBourbaki2002}{
      \BIb{#2}{N. Bourbaki}
         {Lie Groups and Lie Algebras, Chapters 4--6}
         {Springer, New York}{2002}{#1}}
   \ITEE{#3}{ABouziad1996}{
      \BIB{#2}{A. Bouziad}
         {Every \v{C}ech-analytic Baire semitopological group is a topological group}
         {\jRN{PAMS}}{124}{1996}{953--959}{#1}}
   \ITEE{#3}{PLBowers1989}{
      \BIB{#2}{P.L. Bowers}
         {Limitation topologies on function spaces}
         {\jRN{TAMS}}{314}{1989}{421--431}{#1}}
   \ITEE{#3}{JBraconnier1948}{
      \BIB{#2}{J. Braconnier}
         {Sur les groupes topologiques localement compacts}
         {\jRN{JMPuApNS}}{27}{1948}{1--85}{#1}}
   \ITEE{#3}{NBrand1982}{
      \BIB{#2}{N. Brand}
         {Another note on the continuity of the inverse}
         {\jRN{ArchM}}{39}{1982}{241--245}{#1}}
   \ITEE{#3}{ABrown1953}{
      \BIB{#2}{A. Brown}
         {On a class of operators}
         {\jRN{PAMS}}{4}{1953}{723--728}{#1}}
   \ITEE{#3}{ABrown,CKFong,DWHadwin1978}{
      \BIB{#2}{A. Brown, C.-K. Fong, D.W. Hadwin}
         {Parts of operators on Hilbert space}
         {\jRN{IllinoisJM}}{22}{1978}{306--314}{#1}}
   \ITEE{#3}{RBrown,SMorris1977}{
      \BIB{#2}{R. Brown and S. Morris}
         {Embeddings in contractible or compact objects}
         {\jRN{CollM}}{38}{1977}{213--222}{#1}}
   \ITEE{#3}{AMBruckner,JBBruckner,BSThomson1997}{
      \BIb{#2}{A.M. Bruckner, J.B. Bruckner, B.S. Thomson}
         {Real Analysis}
         {Prentice\hyp{}Hall, New Jersey}{1997}{#1}}
   \ITEE{#3}{PJCameron,AMVershik2006}{
      \BIB{#2}{P.J. Cameron and A.M. Vershik}
         {Some isometry groups of Urysohn space}
         {\jRN{AnnPALog}}{143}{2006}{70--78}{#1}}
   \ITEE{#3}{CCastaing1966}{
      \BIB{#2}{C. Castaing}{Quelques 
         probl\`{e}mes de mesurabilit\'{e} li\'{e}es \`{a} la th\'{e}orie de la commande}
         {\jRN{CRParis}}{262}{1966}{409--411}{#1}}
   \ITEE{#3}{JAVanCasteren1980}{
      \BIB{#2}{J.A. van Casteren}
         {A problem of Sz.\hyp{}Nagy}
         {\jRN{ActaSM}}{42}{1980}{189--194}{#1}}
   \ITEE{#3}{JAVanCasteren1983}{
      \BIB{#2}{J.A. van Casteren}
         {Operators similar to unitary or selfadjoint ones}
         {\jRN{PacJM}}{104}{1983}{241--255}{#1}}
   \ITEE{#3}{XCatepillan,MPtak,WSzymanski1994}{
      \BIB{#2}{X. Catepill\'{a}n, M. Ptak, W. Szyma\'{n}ski}{Multiple 
         canonical decompositions of families of operators and a model of quasinormal families}
         {\jRN{PAMS}}{121}{1994}{1165--1172}{#1}}
   \ITEE{#3}{RCauty1994}{
      \BIB{#2}{R. Cauty}
         {Un espace m\'{e}trique lin\'{e}aire qui n'est pas un r\'{e}tracte absolu}
         {\jRN{FM}}{146}{1994}{85--99, (French)}{#1}}
   \ITEE{#3}{TAChapman1971}{
      \BIB{#2}{T.A. Chapman}
         {Deficiency in infinite\hyp{}dimensional manifolds}
         {\jRN{GTopA}}{1}{1971}{263--272}{#1}}
   \ITEE{#3}{TAChapman1976}{
      \BIb{#2}{T.A. Chapman}
         {Lectures on Hilbert cube manifolds}
         {C.B.M.S. Regional Conference Series in Math. No 28, Amer. Math. Soc.}{1976}{#1}}
   \ITEE{#3}{WMChing1974}{
      \BIB{#2}{W.-M. Ching}
         {Topologies on the quasi-spectrum of a $\CCc^*$\hyp{}algebra}
         {\jRN{PAMS}}{46}{1974}{273--276}{#1}}
   \ITEE{#3}{RBChuaqui1977}{
      \BIB{#2}{R.B. Chuaqui}
         {Measures invariant under a group of transformations}
         {\jRN{PacJM}}{68}{1977}{313--329}{#1}}
   \ITEE{#3}{JBConway1985}{
      \BIb{#2}{J.B. Conway}
         {A Course in Functional Analysis}
         {Springer\hyp{}Verlag, New York}{1985}{#1}}
   \ITEE{#3}{JBConway2000}{
      \BIb{#2}{J.B. Conway}
         {A Course in Operator Theory}
         {(Graduate Studies in Mathematics, vol. 21) Amer. Math. Soc., Providence}{2000}{#1}}
   \ITEE{#3}{GCorach,AMaestripieri,MMbekhta2009}{
      \BIB{#2}{G. Corach, A. Maestripieri, M. Mbekhta}
         {Metric and homogeneous structure of closed range operators}
         {\jRN{JOT}}{61}{2009}{171--190}{#1}}
   \ITEE{#3}{MJCowen,RGDouglas1978}{
      \BIB{#2}{M.J. Cowen and R.G. Douglas}
         {Complex geometry and operator theory}
         {\jRN{ActaM}}{141}{1978}{187--261}{#1}}
   \ITEE{#3}{DWCurtis1985}{
      \BIB{#2}{D.W. Curtis}
         {Boundary sets in the Hilbert cube}
         {\jRN{TopA}}{20}{1985}{201--221}{#1}}
   \ITEE{#3}{DVanDantzig,BLVanDerWaerden1928}{
      \BIB{#2}{D. van Dantzig and B.L. van der Waerden}
         {\"{U}ber metrisch homogene R\"{a}ume}
         {\jRN{AbhHamburg}}{6}{1928}{367--376}{#1}}
   \ITEE{#3}{MMDay1958}{
      \BIb{#2}{M.M. Day}
         {Normed Linear Spaces}
         {Springer Verlag, Berlin}{1958}{#1}}
   \ITEE{#3}{CDellacherie1967}{
      \BIB{#2}{C. Dellacherie}
         {Un compl\'{e}ment au th\'{e}or\`{e}me de Weierstrass\hyp{}Stone}
         {\jRN{SeminProbStras}}{1}{1967}{52--53}{#1}}
   \ITEE{#3}{JJDijkstra1987}{
      \BIB{#2}{J.J. Dijkstra}
         {Strong negligibility of $\sigma$\hyp{}compacta does not characterize Hilbert space}
         {\jRN{PacJM}}{127}{1987}{19--30}{#1}}
   \ITEE{#3}{JJDijkstra1990}{
      \BIB{#2}{J.J. Dijkstra}
         {Characterizing Hilbert space topology in terms of strong negligibility}
         {\jRN{ComposM}}{75}{1990}{299--306}{#1}}
   \ITEE{#3}{JDixmier,CFoias1972}{
      \BiB{#2}{J. Dixmier and C. Foias}
         {Sur le spectre ponctuel d'un op\'{e}rateur \textup{(French)}}{in:}
         {Hilbert space operators and operator algebras (Proc. Internat. Conf., Tihany, 1970)}
         {\jRN{CollMSJBoly}, No. 5, North-Holland, Amsterdam}{1972}{127--133}{#1}}
   \ITEE{#3}{TDobrowolski,WMarciszewski2002}{
      \BIB{#2}{T. Dobrowolski and W. Marciszewski}
         {Failure of the Factor Theorem for Borel pre\hyp{}Hilbert spaces}
         {\jRN{FM}}{175}{2002}{53--68}{#1}}
   \ITEE{#3}{TDobrowolski,JMogilski1990}{
      \BiB{#2}{T. Dobrowolski and J. Mogilski}{Problems on Topological Classification 
         of Incomplete Metric Spaces}{Chapter 25 in:}{Open Problems in Topology}{J. van Mill 
         and G.M. Reed (eds.), North\hyp{}Holland Amsterdam}{1990}{411--429}{#1}}
   \ITEE{#3}{TDobrowolski,HTorunczyk1981}{
      \BIB{#2}{T. Dobrowolski and H. Toru\'{n}czyk}
         {Separable complete ANR's admitting a group structure are Hilbert manifolds}
         {\jRN{TopA}}{12}{1981}{229--235}{#1}}
   \ITEE{#3}{RGDouglas1966}{
      \BIB{#2}{R.G. Douglas}
         {On majorization, factorization and range inclusion of operators in Hilbert space}
         {\jRN{PAMS}}{17}{1966}{413--416}{#1}}
   \ITEE{#3}{CHDowker1947}{
      \BIB{#2}{C.H. Dowker}
         {Mapping theorems for non\hyp{}compact spaces}
         {\jRN{AmJM}}{69}{1947}{200--242}{#1}}
   \ITEE{#3}{CHDowker1952}{
      \BIB{#2}{C.H. Dowker}
         {Topology of metric complexes}
         {\jRN{AmJM}}{74}{1952}{555--577}{#1}}
   \ITEE{#3}{JDugundji1951}{
      \BIB{#2}{J. Dugundji}
         {An extension of Tietze's theorem}
         {\jRN{PacJM}}{1}{1951}{353--367}{#1}}
   \ITEE{#3}{JDugundji1958}{
      \BIB{#2}{J. Dugundji}
         {Absolute neighborhood retracts and local connectedness for arbitrary metric spaces}
         {\jRN{ComposM}}{13}{1958}{229--246}{#1}}
   \ITEE{#3}{JDugundji1965}{
      \BIB{#2}{J. Dugundji}
         {Locally equiconnected spaces and absolute neighborhood retracts}
         {\jRN{FM}}{57}{1965}{187--193}{#1}}
   \ITEE{#3}{NDunford,JTSchwartz1958}{
      \BIb{#2}{N. Dunford and J.T. Schwartz}
         {Linear Operators, part I}
         {Interscience Publishers, New York}{1958}{#1}}
   \ITEE{#3}{NDunford,JTSchwartz1963}{
      \BIb{#2}{N. Dunford and J.T. Schwartz}
         {Linear Operators, part II}
         {Interscience Publishers, New York}{1963}{#1}}
   \ITEE{#3}{NDunford,JTSchwartz1971}{
      \BIb{#2}{N. Dunford and J.T. Schwartz}
         {Linear Operators, part III}
         {Wiley\hyp{}Interscience, New York}{1971}{#1}}
   \ITEE{#3}{MLEaton,MDPerlman1977}{
      \BIB{#2}{M.L. Eaton and M.D. Perlman}
         {Reflection groups, generalized Schur functions and the geometry of majorization}
         {\jRN{AnnProb}}{5}{1977}{829--860}{#1}}
   \ITEE{#3}{JEells,NHKuiper1969}{
      \BIB{#2}{J. Eells and N.H. Kuiper}
         {Homotopy negligible subsets in infinite\hyp{}dimensional manifolds}
         {\jRN{ComposM}}{21}{1969}{151--161}{#1}}
   \ITEE{#3}{EGEffros1965}{
      \BIB{#2}{E.G. Effros}
         {The Borel space of von Neumann algebras on a separable Hilbert space}
         {\jRN{PacJM}}{15}{1965}{1153--1164}{#1}}
   \ITEE{#3}{EGEffros1966}{
      \BIB{#2}{E.G. Effros}
         {Global structure in von Neumann algebras}
         {\jRN{TAMS}}{121}{1966}{434--454}{#1}}
   \ITEE{#3}{REllis1957}{
      \BIB{#2}{R. Ellis}
         {A note on the continuity of the inverse}
         {\jRN{PAMS}}{8}{1957}{372--373}{#1}}
   \ITEE{#3}{REngelking1977}{
      \BIb{#2}{R. Engelking}
         {General Topology}
         {\jRN{PWN}}{1977}{#1}}
   \ITEE{#3}{REngelking1978}{
      \BIb{#2}{R. Engelking}
         {Dimension Theory}
         {\jRN{PWN}}{1978}{#1}}
   \ITEE{#3}{REngelking1989}{
      \BIb{#2}{R. Engelking}{General Topology. 
         Revised and completed edition \textup{(Sigma series in pure mathematics, vol. 6)}}
         {Heldermann Verlag, Berlin}{1989}{#1}}
   \ITEE{#3}{PErdos,RDMauldin1976}{
      \BIB{#2}{P. Erd\"{o}s and R.D. Mauldin}
         {The nonexistence of certain invariant measures}
         {\jRN{PAMS}}{59}{1976}{321--322}{#1}}
   \ITEE{#3}{JErnest1976}{
      \BIB{#2}{J. Ernest}
         {Charting the operator terrain}
         {\jRN{MAMS}}{171}{1976}{207 pp}{#1}}
   \ITEE{#3}{REspinola,MAKhamsi2001}{
      \BiB{#2}{R. Espinola and M.A. Khamsi}{Introduction 
         to hyperconvex spaces}{Chapter XIII in:}{Handbook of Metric Fixed Point Theory}
         {W.A. Kirk and B. Sims (editors), Kluwer Academic Publishers}{2001}{391--435}{#1}}
   \ITEE{#3}{JMFell1960a}{
      \BIB{#2}{J.M. Fell}
         {$\CCc^*$-algebras with smooth dual}
         {\jRN{IllinoisJM}}{4}{1960}{221--230}{#1}}
   \ITEE{#3}{JMFell1960b}{
      \BIB{#2}{J.M. Fell}
         {The dual spaces of $\CCc^*$-algebras}
         {\jRN{TAMS}}{94}{1960}{365--403}{#1}}
   \ITEE{#3}{LAFialkow1975}{
      \BIB{#2}{L.A. Fialkow}
         {The similarity orbit of a normal operator}
         {\jRN{TAMS}}{210}{1975}{129--137}{#1}}
   \ITEE{#3}{PAFillmore,JPWilliams1971}{
      \BIB{#2}{P.A. Fillmore and J.P. Williams}
         {On operator ranges}
         {\jRN{AdvM}}{7}{1971}{254--281}{#1}}
   \ITEE{#3}{RHFox1943}{
      \BIB{#2}{R.H. Fox}
         {On fiber spaces, II}
         {\jRN{BAMS}}{49}{1943}{733--735}{#1}}
   \ITEE{#3}{RFraisse1954}{
      \BIB{#2}{R. Fra\"{\i}ss\'{e}}
         {Sur quelques classifications des syst\`{e}mes de relations}
         {\jRN{PublSUAA}}{1}{1954}{35--182}{#1}}
   \ITEE{#3}{NAFriedman1970}{
      \BIb{#2}{N.A. Friedman}
         {Introduction to ergodic theory}
         {Van Nostrand Reinhold Company}{1970}{#1}}
   \ITEE{#3}{MFujii,MKajiwara,YKato,FKubo1976}{
      \BIB{#2}{M. Fujii, M. Kajiwara, Y. Kato, F. Kubo}
         {Decompositions of operators in Hilbert spaces}
         {\jRN{MathJap}}{21}{1976}{117--120}{#1}}
   \ITEE{#3}{SGao,ASKechris2003}{
      \BIB{#2}{S. Gao and A.S. Kechris}
         {On the classification of Polish metric spaces up to isometry}
         {\jRN{MAMS}}{161}{2003}{viii+78}{#1}}
   \ITEE{#3}{MIGarrido,FMontalvo1991}{
      \BIB{#2}{M.I. Garrido and F. Montalvo}
         {On some generalizations of the Kakutani\hyp{}Stone and Stone\hyp{}Weierstrass theorems}
         {\jRN{ExtrM}}{6}{1991}{156--159}{#1}}
   \ITEE{#3}{LGe,JShen2002}{
      \BIB{#2}{L. Ge and J. Shen}
         {Generator problem for certain property T factors}
         {\jRN{PNAS}}{99}{2002}{565--567}{#1}}
   \ITEE{#3}{IMGelfand,MANaimark1943}{
      \BIB{#2}{I.M. Gelfand and M.A. Naimark}
         {On the embedding of normed rings into the ring of operators in Hilbert space}
         {\jRN{MSb}}{12}{1943}{197--213}{#1}}
   \ITEE{#3}{RGellar,LPage1974}{
      \BIB{#2}{R. Gellar and L. Page}
         {Limits of unitarily equivalent normal operators}
         {\jRN{DMJ}}{41}{1974}{319--322}{#1}}
   \ITEE{#3}{FGesztesy,MMalamud,MMitrea,SNaboko2009}{
      \BIB{#2}{F. Gesztesy, M. Malamud, M. Mitrea, S. Naboko}{Generalized 
         polar decompositions for closed operators in Hilbert spaces and some applications}
         {\jRN{IEOT}}{64}{2009}{83--113}{#1}}
   \ITEE{#3}{LGillman,MJerison1960}{
      \BIb{#2}{L. Gillman and M. Jerison}
         {Rings of continuous functions}
         {New York}{1960}{#1}}
   \ITEE{#3}{JGlimm1960}{
      \BIB{#2}{J. Glimm}
         {A Stone\hyp{}Weierstrass theorem for $\CCc^*$\hyp{}algebras}
         {\jRN{AnnM}}{72}{1960}{216--244}{#1}}
   \ITEE{#3}{JGlimm1961}{
      \BIB{#2}{J. Glimm}
         {Type I $\CCc^*$-algebras}
         {\jRN{AnnM}}{73}{1961}{572--612}{#1}}
   \ITEE{#3}{GGodefroy,NJKalton2003}{
      \BIB{#2}{G. Godefroy and N.J. Kalton}
         {Lipschitz\hyp{}free Banach spaces}
         {\jRN{SM}}{159}{2003}{121--141}{#1}}
   \ITEE{#3}{ICGohberg,MGKrein1967}{
      \BIB{#2}{I.C. Gohberg and M.G. Krein}
         {On a description of contraction operators similar to unitary ones}
         {\jRN{FunkAnalPril}}{1}{1967}{38--60}{#1}}
   \ITEE{#3}{KRGoodearl,PMenal1990}{
      \BIB{#2}{K.R. Goodearl and P. Menal}
         {Free and residually finite\hyp{}dimensional $\CCc^*$-algebras}
         {\jRN{JFA}}{90}{1990}{391--410}{#1}}
   \ITEE{#3}{ELGriffinJr1953}{
      \BIB{#2}{E.L. Griffin Jr.}
         {Some contributions to the theory of rings of operators}
         {\jRN{TAMS}}{75}{1953}{471--504}{#1}}
   \ITEE{#3}{ELGriffinJr1955}{
      \BIB{#2}{E.L. Griffin Jr.}
         {Some contributions to the theory of rings of operators II}
         {\jRN{TAMS}}{79}{1955}{389--400}{#1}}
   \ITEE{#3}{MGromov1981}{
      \BIB{#2}{M. Gromov}
         {Groups of polynomial growth and expanding maps}
         {\jRN{InHauEtSPM}}{53}{1981}{53--73}{#1}}
   \ITEE{#3}{MGromov1999}{
      \BIb{#2}{M. Gromov}
         {Metric Structures for Riemannian and Non\hyp{}Riemannian Spaces}
         {Progress in Math. \textbf{152}, Birkh\"{a}user}{1999}{#1}}
   \ITEE{#3}{JDeGroot1956}{
      \BIB{#2}{J. de Groot}
         {Non\hyp{}archimedean metrics in topology}
         {\jRN{PAMS}}{7}{1956}{948--953}{#1}}
   \ITEE{#3}{LCGrove,CTBenson1985}{
      \BIb{#2}{L.C. Grove and C.T. Benson}
         {Finite Reflection Group}
         {2nd ed., Springer\hyp{}Verlag}{1985}{#1}}
   \ITEE{#3}{JBGuerrero,ARodriguez-Palacios2002}{
      \BIB{#2}{J.B. Guerrero and A. Rodr\'{\i}guez\hyp{}Palacios}
         {Transitivity of the Norm on Banach Spaces}
         {\jRN{ExtrM}}{17}{2002}{1--58}{#1}}
   \ITEE{#3}{VIGurarii1966}{
      \BIB{#2}{V.I. Gurari\v{\i}}{Spaces of universal placement, isotropic spaces and a problem 
         of Mazur on rotations of Banach spaces \textup{(Russian)}}
         {\jRN{SibirMZ}}{7}{1966}{1002--1013}{#1}}
   \ITEE{#3}{DWHadwin1974}{
      \BIB{#2}{D.W. Hadwin}
         {Closures of unitary equivalence classes}
         {\jRN{NAMS}}{21}{1974}{\#74T-B55}{#1}}
   \ITEE{#3}{DWHadwin1976}{
      \BIB{#2}{D.W. Hadwin}
         {An operator\hyp{}valued spectrum}
         {\jRN{NAMS}}{23}{1976}{A-163}{#1}}
   \ITEE{#3}{DWHadwin1977}{
      \BIB{#2}{D.W. Hadwin}
         {An operator\hyp{}valued spectrum}
         {\jRN{IndianaUMJ}}{26}{1977}{329--340}{#1}}
   \ITEE{#3}{DWHadwin1981}{
      \BIB{#2}{D.W. Hadwin}
         {Nonseparable approximate equivalence}
         {\jRN{TAMS}}{266}{1981}{203--231}{#1}}
   \ITEE{#3}{HHahn1932}{
      \BIb{#2}{H. Hahn}
         {Reelle Funktionen I}
         {Leipzig}{1932}{#1}}
   \ITEE{#3}{PRHalmos1950}{
      \BIb{#2}{P.R. Halmos}
         {Measure theory}
         {Van Nostrand, New York}{1950}{#1}}
   \ITEE{#3}{PRHalmos1951}{
      \BIb{#2}{P.R. Halmos}
         {Introduction to Hilbert Space and the Theory of Spectral Multiplicity}
         {Chelsea Publishing Company, New York}{1951}{#1}}
   \ITEE{#3}{PRHalmos1956}{
      \BIb{#2}{P.R. Halmos}
         {Lectures on Ergodic Theory}
         {Publ. Math. Soc. Japan, Tokyo}{1956}{#1}}
   \ITEE{#3}{PRHalmos1974}{
      \BIb{#2}{P.R. Halmos}
         {Measure theory}
         {Springer\hyp{}Verlag, New York}{1974}{#1}}
   \ITEE{#3}{PRHalmos1982}{
      \BIb{#2}{P.R. Halmos}
         {A Hilbert Space Problem Book}
         {Springer\hyp{}Verlag New York Inc.}{1982}{#1}}
  \ITEE{#3}{PRHalmos,JEMcLaughlin1963}{
      \BIB{#2}{P.R. Halmos and J.E. McLaughlin}
         {Partial isometries}
         {\jRN{PacJM}}{13}{1963}{585--596}{#1}}
   \ITEE{#3}{RWHansell1972}{
      \BIB{#2}{R.W. Hansell}
         {On the nonseparable theory of Borel and Souslin sets}
         {\jRN{BAMS}}{78}{1972}{236--241}{#1}}
   \ITEE{#3}{SHartman,JMycielski1957}{
      \BIB{#2}{S. Hartman and J. Mycielski}
         {On the imbedding of topological groups into connected topological groups}
         {\jRN{CollM}}{5}{1957}{167--169}{#1}}
   \ITEE{#3}{FHausdorff1930}{
      \BIB{#2}{F. Hausdorff}
         {Erweiterung einer Hom\"{o}omorphie}
         {\jRN{FM}}{16}{1930}{353--360}{#1}}
   \ITEE{#3}{FHausdorff1934}{
      \BIB{#2}{F. Hausdorff}
         {\"{U}ber innere Abbildungen}
         {\jRN{FM}}{23}{1934}{279--291}{#1}}
   \ITEE{#3}{FHausdorff1938}{
      \BIB{#2}{F. Hausdorff}
         {Erweiterung einer stetigen Abbildung}
         {\jRN{FM}}{30}{1938}{40--47}{#1}}
   \ITEE{#3}{DWHenderson1971}{
      \BIB{#2}{D.W. Henderson}
         {Corrections and extensions of two papers about infinite\hyp{}dimensional manifolds}
         {\jRN{GTopA}}{1}{1971}{321--327}{#1}}
   \ITEE{#3}{DWHenderson1975}{
      \BIB{#2}{D.W. Henderson}
         {$Z$\hyp{}sets in ANR's}
         {\jRN{TAMS}}{213}{1975}{205--216}{#1}}
   \ITEE{#3}{DWHenderson,RMSchori1970}{
      \BIB{#2}{D.W. Henderson and R.M. Schori}
         {Topological classification of infinite\hyp{}dimensional manifolds by homotopy type}
         {\jRN{BAMS}}{76}{1970}{121--124}{#1}}
   \ITEE{#3}{DWHenderson,JEWest1970}{
      \BIB{#2}{D.W. Henderson and J.E. West}
         {Triangulated infinite\hyp{}dimensional manifolds}
         {\jRN{BAMS}}{76}{1970}{655--660}{#1}}
   \ITEE{#3}{DAHerrero1976}{
      \BIB{#2}{D.A. Herrero}
         {Closure of similarity orbits of Hilbert space operators, II: normal operators}
         {\jRN{JLondMS}}{13}{1976}{299--316}{#1}}
   \ITEE{#3}{EHewitt,KARoss1979}{
      \BIb{#2}{E. Hewitt and K.A. Ross}{Abstract Harmonic 
         Analysis I \textup{(A Series of Comprehensive Studies in Mathematics, Vol. 115)}}
         {Springer\hyp{}Verlag, New York}{1979}{#1}}
   \ITEE{#3}{EHewitt,KARoss1997}{
      \BIb{#2}{E. Hewitt and K.A. Ross}{Abstract Harmonic 
         Analysis II \textup{(A Series of Comprehensive Studies in Mathematics, Vol. 152}}
         {Springer\hyp{}Verlag, Berlin}{1997}{#1}}
   \ITEE{#3}{BHoffmann1979}{
      \BIB{#2}{B. Hoffmann}
         {A compact contractible topological group is trivial}
         {\jRN{ArchM}}{32}{1979}{585--587}{#1}}
   \ITEE{#3}{DHofmann2002}{
      \BIB{#2}{D. Hofmann}
         {On a generalization of the Stone\hyp{}Weierstrass theorem}
         {\jRN{ACS}}{10}{2002}{569--592}{#1}}
   \ITEE{#3}{GHognas,AMukherjea1995}{
      \BIb{#2}{G. H\"ogn\"as and A. Mukherjea}{Probability 
         Measures on Semigroups. Convolution Products, Random Walks, and Random Matrices}
         {Plenum Press, New York}{1995}{#1}}
   \ITEE{#3}{MRHolmes1992}{
      \BIB{#2}{M.R. Holmes}{The universal 
         separable metric space of Urysohn and isometric embeddings thereof in Banach spaces}
         {\jRN{FM}}{140}{1992}{199--223}{#1}}
   \ITEE{#3}{MRHolmes2008}{
      \BIB{#2}{M.R. Holmes}
         {The Urysohn space embeds in Banach spaces in just one way}
         {\jRN{TopA}}{155}{2008}{1479--1482}{#1}}
   \ITEE{#3}{RRHolmes,TYTam1999}{
      \BIB{#2}{R.R. Holmes and T.Y. Tam}
         {Distance to the convex hull of an orbit under the action of a compact group}
         {\jRN{JAusMSA}}{66}{1999}{331--357}{#1}}
   \ITEE{#3}{RHorn,RMathias1990}{
      \BIB{#2}{R. Horn and R. Mathias}
         {Cauchy\hyp{}Schwartz inequalities associated with positive semidefinite matrices}
         {\jRN{LAA}}{142}{1990}{63--82}{#1}}
   \ITEE{#3}{GEHuhunaisvili1955}{
      \BIB{#2}{G.E. Huhunai\v{s}vili}
         {On a property of Urysohn's universal metric space}
         {\jRN{DANSSSR}}{101}{1955}{607--610 (Russian)}{#1}}
   \ITEE{#3}{JEHumphreys1990}{
      \BIb{#2}{J.E. Humphreys}
         {Reflection Groups and Coxeter Groups}
         {Cambridge University Press}{1990}{#1}}
   \ITEE{#3}{JRIsbell1964}{
      \BIB{#2}{J.R. Isbell}
         {Six theorems about injective metric spaces}
         {\jRN{CMHelv}}{39}{1964}{65--76}{#1}}
   \ITEE{#3}{SIzumino,YKato1985}{
      \BIB{#2}{S. Izumino and Y. Kato}
         {The closure of invertible operators on Hilbert space}
         {\jRN{ActaSM}}{49}{1985}{321--327}{#1}}
   \ITEE{#3}{CJiang2004}{
      \BIB{#2}{C. Jiang}
         {Similarity classification of Cowen\hyp{}Douglas operators}
         {\jRN{CanadJM}}{56}{2004}{742--775}{#1}}
   \ITEE{#3}{WBJohnson,JLindenstrauss2001}{
      \BiB{#2}{W.B. Johnson and J. Lindenstrauss}{Basic Concepts in the Geometry of Banach Spaces}
         {Chapter 1 in:}{Handbook of the Geometry of Banach Spaces, Vol. 1}{W.B. Johnson 
         and J. Lindenstrauss (editors), Elsevier Science B.V., Amsterdam}{2001}{1--84}{#1}}
   \ITEE{#3}{IBJung,JStochel2008}{
      \BIB{#2}{I.B. Jung and J. Stochel}
         {Subnormal operators whose adjoints have rich point spectrum}
         {\jRN{JFA}}{255}{2008}{1797--1816}{#1}}
   \ITEE{#3}{RVKadison,JRRingrose1983}{
      \BIb{#2}{R.V. Kadison and J.R. Ringrose}
         {Fundamentals of the Theory of Operator Algebras. Volume I: Elementary Theory}
         {Academic Press, Inc., New York\hyp{}London}{1983}{#1}}
   \ITEE{#3}{RVKadison,JRRingrose1986}{
      \BIb{#2}{R.V. Kadison and J.R. Ringrose}
         {Fundamentals of the Theory of Operator Algebras. Volume II: Advanced Theory}
         {Academic Press, Inc., Orlando\hyp{}London}{1986}{#1}}
   \ITEE{#3}{SKakutani1936}{
      \BIB{#2}{S. Kakutani}
         {\"{U}ber die Metrisation der topologischen Gruppen}
         {\jRN{ProcImpAcadTokyo}}{12}{1936}{82--84}{#1}}
   \ITEE{#3}{SKakutani1938}{
      \BIB{#2}{S. Kakutani}
         {Two fixed\hyp{}point theorems concerning bicompact convex sets}
         {\jRN{ProcImpAcadTokyo}}{14}{1938}{242--245}{#1}}
   \ITEE{#3}{SKakutani1941}{
      \BIB{#2}{S. Kakutani}
         {Concrete representation of abstract L\hyp{}spaces}
         {\jRN{AnnM}}{42}{1941}{523--537}{#1}}
   \ITEE{#3}{SKakutani1941a}{
      \BIB{#2}{S. Kakutani}
         {Concrete representation of abstract M\hyp{}spaces}
         {\jRN{AnnM}}{42}{1941}{994--1024}{#1}}
   \ITEE{#3}{NKalton2007}{
      \BIB{#2}{N. Kalton}
         {Extending Lipschitz maps into $\CCc(K)$\hyp{}spaces}
         {\jRN{IsraelJM}}{162}{2007}{275--315}{#1}}
   \ITEE{#3}{RKane2001}{
      \BIb{#2}{R. Kane}
         {Reflection Groups and Invariant Theory}
         {Canadian Mathematical Society, Springer}{2001}{#1}}
   \ITEE{#3}{VKannan,SRRaju1980}{
      \BIB{#2}{V. Kannan and S.R. Raju}
         {The nonexistence of invariant universal measures on semigroups}
         {\jRN{PAMS}}{78}{1980}{482--484}{#1}}
   \ITEE{#3}{IKaplansky1951}{
      \BIB{#2}{I. Kaplansky}
         {A theorem on rings of operators}
         {\jRN{PacJM}}{1}{1951}{227--232}{#1}}
   \ITEE{#3}{MKatetov1988}{
      \BiB{#2}{M. Kat\v{e}tov}{On universal metric spaces}{in: Frolik (ed.),}{General Topology 
         and its Relations to Modern Analysis and Algebra VI. Proceedings of the Sixth Prague 
         Topological Symposium 1986}{Heldermann Verlag Berlin}{1988}{323--330}{#1}}
   \ITEE{#3}{YKatznelson1960}{
      \BIB{#2}{Y. Katznelson}{Sur les alg\'{e}bres 
         dont les \'{e}l\'{e}ments non n\'{e}gatifs admettent des racines carr\'{e}es}
         {\jRN{AnnSciEcNormSupT}}{77}{1960}{167--174}{#1}}
   \ITEE{#3}{RKaufman1981}{
      \BIB{#2}{R. Kaufman}
         {Lipschitz spaces and Suslin sets}
         {\jRN{JFA}}{42}{1981}{271--273}{#1}}
   \ITEE{#3}{RKaufman1984}{
      \BIB{#2}{R. Kaufman}
         {Representation of Suslin sets by operators}
         {\jRN{IEOT}}{7}{1984}{808--814}{#1}}
   \ITEE{#3}{OHKeller1931}{
      \BIB{#2}{O.H. Keller}
         {Die Homoiomorphie der kompakten konvexen Mengen in Hilbertschen Raum}
         {\jRN{MAnn}}{105}{1931}{748--758}{#1}}
   \ITEE{#3}{MAKhamsi,WAKirk,CMartinez2000}{
      \BIB{#2}{M.A. Khamsi, W.A. Kirk, C. Martinez}
         {Fixed point and selection theorems in hyperconvex spaces}
         {\jRN{PAMS}}{128}{2000}{3275--3283}{#1}}
   \ITEE{#3}{ABKhararazishvili1998}{
      \BIb{#2}{A.B. Khararazishvili}
         {Transformation groups and invariant measures. Set\hyp{}theoretic aspects}
         {World Scientific Publishing Co., Inc., River Edge, NJ}{1998}{#1}}
   \ITEE{#3}{YKijima1987}{
      \BIB{#2}{Y. Kijima}
         {Fixed points of nonexpansive self\hyp{}maps of a compact metric space}
         {\jRN{JMAnApp}}{123}{1987}{114--116}{#1}}
  \ITEE{#3}{JSKim,ChRKim,SGLee1980}{
      \BIB{#2}{J.S. Kim, Ch.R. Kim, S.G. Lee}
         {Reducing operator valued spectra of a Hilbert space operator}
         {\jRN{JKoreanMS}}{17}{1980}{123--129}{#1}}
   \ITEE{#3}{JKindler1995}{
      \BIB{#2}{J. Kindler}
         {Minimax theorems with applications to convex metric spaces}
         {\jRN{CollM}}{68}{1995}{179--186}{#1}}
   \ITEE{#3}{WAKirk1998}{
      \BIB{#2}{W.A. Kirk}
         {Hyperconvexity of $\RRR$\hyp{}trees}
         {\jRN{FM}}{156}{1998}{67--72}{#1}}
   \ITEE{#3}{VLKleeJr1952}{
      \BIB{#2}{V.L. Klee Jr.}
         {Invariant metrics in groups (solution of a problem of Banach)}
         {\jRN{PAMS}}{3}{1952}{484--487}{#1}}
   \ITEE{#3}{JLKoszul1965}{
      \BIb{#2}{J.L. Koszul}
         {Lectures on groups of transformations}
         {Tata Institute of Fundamental Research, Bombay}{1965}{#1}}
   \ITEE{#3}{HJKowalsky1957}{
      \BIB{#2}{H.J. Kowalsky}
         {Einbettung metrischer R\"{a}ume}
         {\jRN{ArchM}}{8}{1957}{336--339}{#1}}
   \ITEE{#3}{WKubis,MRubin2010}{
      \BIB{#2}{W. Kubi\'{s} and M. Rubin}
         {Extension and reconstruction theorems for the Urysohn universal metric space}
         {\jRN{CzMJ}}{60}{2010}{1--29}{#1}}
   \ITEE{#3}{KKuratowski1966}{
      \BIb{#2}{K. Kuratowski}
         {Topology. \textup{Vol. I}}
         {\jRN{PWN}}{1966}{#1}}
   \ITEE{#3}{KKuratowski,BKnaster1927}{
      \BIB{#2}{K. Kuratowski and B. Knaster}
         {A connected and connected im kleinen point set which contains no perfect subset}
         {\jRN{BAMS}}{33}{1927}{106--109}{#1}}
   \ITEE{#3}{KKuratowski,AMostowski1976}{
      \BIb{#2}{K. Kuratowski and A. Mostowski}
         {Set Theory with an Introduction to Descriptive Set Theory}
         {\jRN{PWN}}{1976}{#1}}
   \ITEE{#3}{GLewicki1992}{
      \BIB{#2}{G. Lewicki}
         {Bernstein's ``lethargy'' theorem in metrizable topological linear spaces}
         {\jRN{MonatM}}{113}{1992}{213--226}{#1}}
   \ITEE{#3}{ASLewis1996}{
      \BIB{#2}{A.S. Lewis}
         {Group invariance and convex matrix analysis}
         {\jRN{SIAMJMAA}}{17}{1996}{927--949}{#1}}
   \ITEE{#3}{C-KLi,N-KTsing1991}{
      \BIB{#2}{C.-K. Li and N.-K. Tsing}
         {$G$\hyp{}invariant norms and $G(c)$\hyp{}radii}
         {\jRN{LAA}}{150}{1991}{179--194}{#1}}
   \ITEE{#3}{AJLazar,JLindenstrauss1971}{
      \BIB{#2}{A.J. Lazar and J. Lindenstrauss}
         {Banach spaces whose duals are $L_1$ spaces and their representing matrices}
         {\jRN{ActaM}}{126}{1971}{165--193}{#1}}
   \ITEE{#3}{SGLee1980}{
      \BIB{#2}{S.G. Lee}
         {Remarks on reducing operator valued spectrum}
         {\jRN{JKoreanMS}}{16}{1980}{131--136}{#1}}
   \ITEE{#3}{EHLieb,MLoss1997}{
      \BIb{#2}{E.H. Lieb and M. Loss}
         {Analysis \textup{(Graduate Studies in Mathematics, vol. 14)}}
         {Amer. Math. Soc., Providence, RI}{1997}{#1}}
   \ITEE{#3}{HLin2001}{
      \BIB{#2}{H. Lin}
         {Residually finite\hyp{}dimensional and AF\hyp{}embeddable $\CCc^*$\hyp{}algebras}
         {\jRN{PAMS}}{129}{2001}{1689--1696}{#1}}
   \ITEE{#3}{ALindenbaum1926}{
      \BIB{#2}{A. Lindenbaum}
         {Contributions \`{a} l'\'{e}tude de l'espace m\'{e}trique I}
         {\jRN{FM}}{8}{1926}{209--222}{#1}}
   \ITEE{#3}{DLindenstrauss,LTzafriri1971}{
      \BIB{#2}{D. Lindenstrauss and L. Tzafriri}
         {On the complemented subspaces problem}
         {\jRN{IsraelJM}}{9}{1971}{263--269}{#1}}
   \ITEE{#3}{RILoebl1986}{
      \BIB{#2}{R.I. Loebl}
         {A note on containment of operators}
         {\jRN{BAustrMS}}{33}{1986}{279--291}{#1}}
   \ITEE{#3}{LHLoomis1945}{
      \BIB{#2}{L.H. Loomis}
         {Abstract congruence and the uniqueness of Haar measure}
         {\jRN{AnnM}}{46}{1945}{348--355}{#1}}
   \ITEE{#3}{LHLoomis1949}{
      \BIB{#2}{L.H. Loomis}
         {Haar measure in uniform structures}
         {\jRN{DMJ}}{16}{1949}{193--208}{#1}}
   \ITEE{#3}{ERLorch1939}{
      \BIB{#2}{E.R. Lorch}
         {Bicontinuous linear transformation in certain vector spaces}
         {\jRN{BAMS}}{45}{1939}{564--569}{#1}}
   \ITEE{#3}{KLowner1934}{
      \BIB{#2}{K. L\"{o}wner}
         {\"{U}ber monotone Matrixfunctionen}
         {\jRN{MZ}}{38}{1934}{177--216}{#1}}
   \ITEE{#3}{FLuft1968}{
      \BIB{#2}{F. Luft}{The two-sided 
         closed ideals of the algebra of bounded linear operators of a Hilbert space}
         {\jRN{CzMJ}}{18}{1968}{595--605}{#1}}
   \ITEE{#3}{ATLundell,SWeingram1969}{
      \BIb{#2}{A.T. Lundell and S. Weingram}
         {The topology of CW\hyp{}complexes}
         {Litton Educ. Publ.}{1969}{#1}}
   \ITEE{#3}{WLusky1976}{
      \BIB{#2}{W. Lusky}
         {The Gurarij spaces are unique}
         {\jRN{ArchM}}{27}{1976}{627--635}{#1}}
   \ITEE{#3}{WLusky1977}{
      \BIB{#2}{W. Lusky}
         {On separable Lindenstrauss spaces}
         {\jRN{JFA}}{26}{1977}{103--120}{#1}}
   \ITEE{#3}{DMaharam1942}{
      \BIB{#2}{D. Maharam}
         {On homogeneous measure algebras}
         {\jRN{PNAS}}{28}{1942}{108--111}{#1}}
   \ITEE{#3}{MMalicki,SSolecki2009}{
      \BIB{#2}{M. Malicki and S. Solecki}
         {Isometry groups of separable metric spaces}
         {\jRN{MProcCambPhS}}{146}{2009}{67--81}{#1}}
   \ITEE{#3}{PMankiewicz1972}{
      \BIB{#2}{P. Mankiewicz}
         {On extension of isometries in normed linear spaces}
         {\jRN{BAPolSSSM}}{20}{1972}{367--371}{#1}}
   \ITEE{#3}{AManoussos,PStrantzalos2003}{
      \BIB{#2}{A. Manoussos and P. Strantzalos}
         {On the group of isometries on a locally compact metric space}
         {\jRN{JLieTh}}{13}{2003}{7--12}{#1}}
   \ITEE{#3}{JMartinezMaurica,MTPellon1987}{
      \BIB{#2}{J. Martinez\hyp{}Maurica and M.T. Pell\'{o}n}
         {Non\hyp{}archimedean Chebyshev centers}
         {\jRN{IndagMP}}{90}{1987}{417--421}{#1}}
   \ITEE{#3}{KMaurin1980}{
      \BIb{#2}{K. Maurin}
         {Analysis, Part II}
         {D. Reidel, Dordrecht\hyp{}Boston\hyp{}London}{1980}{#1}}
   \ITEE{#3}{EMayer-Wolf1981}{
      \BIB{#2}{E. Mayer-Wolf}
         {Isometries between Banach spaces of Lipschitz functions}
         {\jRN{IsraelJM}}{38}{1981}{58--74}{#1}}
   \ITEE{#3}{SMazur,SUlam1932}{
      \BIB{#2}{S. Mazur and S. Ulam}
         {Sur les transformationes isom\'{e}triques d'espaces vectoriels norm\'{e}s}
         {\jRN{CRASParis}}{194}{1932}{946--948}{#1}}
   \ITEE{#3}{SMazurkiewicz1920}{
      \BIB{#2}{S. Mazurkiewicz}
         {Sur les lignes de Jordan}
         {\jRN{FM}}{1}{1920}{166--209}{#1}}
   \ITEE{#3}{SMazurkiewicz,WSierpinski1920}{
      \BIB{#2}{S. Mazurkiewicz and W. Sierpi\'{n}ski}
         {Contributions a la topologie des ensembles denombrables}
         {\jRN{FM}}{1}{1920}{17--27}{#1}}
   \ITEE{#3}{MMbekhta1992}{
      \BIB{#2}{M. Mbekhta}
         {Sur la structure des composantes connexes semi\hyp{}Fredholm de $B(H)$}
         {\jRN{PAMS}}{116}{1992}{521--524}{#1}}
   \ITEE{#3}{JEMcCarthy1996}{
      \BIB{#2}{J.E. McCarthy}
         {Boundary values and Cowen\hyp{}Douglas curvature}
         {\jRN{JFA}}{137}{1996}{1--18}{#1}}
   \ITEE{#3}{JMelleray2007}{
      \BIB{#2}{J. Melleray}
         {Computing the complexity of the relation of isometry between separable Banach spaces}
         {\jRN{MLQ}}{53}{2007}{128--131}{#1}}
   \ITEE{#3}{JMelleray2007a}{
      \BIB{#2}{J. Melleray}
         {On the geometry of Urysohn's universal metric space}
         {\jRN{TopA}}{154}{2007}{384--403}{#1}}
   \ITEE{#3}{JMelleray2008}{
      \BIB{#2}{J. Melleray}
         {Some geometric and dynamical properties of the Urysohn space}
         {\jRN{TopA}}{155}{2008}{1531--1560}{#1}}
   \ITEE{#3}{JMelleray2008a}{
      \BIB{#2}{J. Melleray}
         {Compact metrizable groups are isometry groups of compact metric spaces}
         {\jRN{PAMS}}{136}{2008}{1451--1455}{#1}}
   \ITEE{#3}{JMelleray,FVPetrov,AMVershik2008}{
      \BIB{#2}{J. Melleray, F.V. Petrov, A.M. Vershik}
         {Linearly rigid metric spaces and the embedding problem}
         {\jRN{FM}}{199}{2008}{177--194}{#1}}
   \ITEE{#3}{EMichael1953}{
      \BIB{#2}{E. Michael}
         {Some extension theorems for continuous functions}
         {\jRN{PacJM}}{3}{1953}{789--806}{#1}}
   \ITEE{#3}{EMichael1954}{
      \BIB{#2}{E. Michael}
         {Local properties of topological spaces}
         {\jRN{DMJ}}{21}{1954}{163--171}{#1}}
   \ITEE{#3}{EMichael1956}{
      \BIB{#2}{E. Michael}
         {Selected selection theorems}
         {\jRN{AmMMon}}{58}{1956}{233--238}{#1}}
   \ITEE{#3}{EMichael1956a}{
      \BIB{#2}{E. Michael}
         {Continuous selections. I}
         {\jRN{AnnM}}{63}{1956}{361--382}{#1}}
   \ITEE{#3}{EMichael1956b}{
      \BIB{#2}{E. Michael}
         {Continuous selections. II}
         {\jRN{AnnM}}{64}{1956}{562--580}{#1}}
   \ITEE{#3}{EMichael1959}{
      \BIB{#2}{E. Michael}
         {A theorem on semi\hyp{}continuous set\hyp{}valued functions}
         {\jRN{DMJ}}{26}{1959}{647--652}{#1}}
   \ITEE{#3}{EMichael1964}{
      \BIB{#2}{E. Michael}
         {A short proof of the Arens-Eells embedding theorem}
         {\jRN{PAMS}}{15}{1964}{415--416}{#1}}
   \ITEE{#3}{JVanMill1986}{
      \BIB{#2}{J. van Mill}
         {Another counterexample in ANR theory}
         {\jRN{PAMS}}{97}{1986}{136--138}{#1}}
   \ITEE{#3}{JVanMill2001}{
      \BIb{#2}{J. van Mill}
         {The Infinite\hyp{}Dimensional Topology of Function Spaces 
         \textup{(North\hyp{}Holland Mathematical Library, vol. 64)}}
         {Elsevier, Amsterdam}{2001}{#1}}
   \ITEE{#3}{KMine2006}{
      \BIB{#2}{K. Mine}
         {Universal spaces of non\hyp{}separable absolute Borel classes}
         {\jRN{TsukubaJM}}{30}{2006}{137--148}{#1}}
   \ITEE{#3}{WMlak1991}{
      \BIb{#2}{W. Mlak}{Hilbert Spaces and Operator Theory}
         {PWN --- Polish Scientific Publishers and Kluwer Academic Publishers, 
         Warszawa\hyp{}Dordrecht}{1991}{#1}}
   \ITEE{#3}{JMogilski1979}{
      \BIB{#2}{J. Mogilski}
         {$CE$\hyp{}decomposition of $\ell_2$\hyp{}manifolds}
         {\jRN{BAPolSSSM}}{27}{1979}{309--314}{#1}}
   \ITEE{#3}{RLMoore1916}{
      \BIB{#2}{R.L. Moore}
         {On the foundations of plane analysis situs}
         {\jRN{TAMS}}{17}{1916}{131--164}{#1}}
   \ITEE{#3}{KMorita1955}{
      \BIB{#2}{K. Morita}
         {A condition for the metrizability of topological spaces and for $n$\hyp{}dimensionality}
         {\jRN{SciRepTokyoA}}{5}{1955}{33--36}{#1}}
   \ITEE{#3}{AMukherjea,NATserpes1976}{
      \BIb{#2}{A. Mukherjea and N.A. Tserpes}
         {Measures on topological semigroups}
         {Springer Lecture Notes in Math. Vol. 547, Berlin}{1976}{#1}}
   \ITEE{#3}{JMycielski1974}{
      \BIB{#2}{J. Mycielski}
         {Remarks on invariant measures in metric spaces}
         {\jRN{CollM}}{32}{1974}{105--112}{#1}}
   \ITEE{#3}{SNNaboko1984}{
      \BIB{#2}{S.N. Naboko}
         {Conditions for similarity to unitary and selfadjoint operators}
         {\jRN{FunkAnalPril}}{18}{1984}{16--27}{#1}}
   \ITEE{#3}{LNachbin1965}{
      \BIb{#2}{L. Nachbin}{The Haar Integral}
         {D. Van Nostrand Company, Inc., 
         Princeton\hyp{}New Jersey\hyp{}Toronto\hyp{}New York\hyp{}London}{1965}{#1}}
   \ITEE{#3}{TDNarang,SKGarg1991}{
      \BIB{#2}{T.D. Narang and S.K. Garg}
         {On the uniqueness of best approximation in non\hyp{}archimedian spaces}
         {\jRN{PeriodMHung}}{22}{1991}{121--124}{#1}}
   \ITEE{#3}{JVonNeumann1930}{
      \BIB{#2}{J. von Neumann}
         {Zur Algebra der Funktionaloperationen und Theorie der normalen Operatoren}
         {\jRN{MAnn}}{102}{1930}{370--427}{#1}}
   \ITEE{#3}{JVonNeumann1934}{
      \BIB{#2}{J. von Neumann}
         {Zum Haarschen Mass in topologischen Gruppen}
         {\jRN{ComposM}}{1}{1934}{106--114}{#1}}
   \ITEE{#3}{JVonNeumann1937}{
      \BiB{#2}{J. von Neumann}{Some matrix\hyp{}inequalities 
         and metrization of matrix\hyp{}space}{\jRN{TomskUnivRev}{} \textbf{1} (1937), 
         286--300; in }{Collected Works}{Pergamon, New York}{1962}{Vol. 4, 205--219}{#1}}
   \ITEE{#3}{JVonNeumann1949}{
      \BIB{#2}{J. von Neumann}
         {On Rings of Operators. Reduction Theory}
         {\jRN{AnnM}}{50}{1949}{401--485}{#1}}
   \ITEE{#3}{ONielson1973}{
      \BIB{#2}{O. Nielson}
         {Borel sets of von Neumann algebras}
         {\jRN{AmJM}}{95}{1973}{145--164}{#1}}
   \ITEE{#3}{pn2002}{\bibITEM{#2}{#1} \mypaplist{pn1}{}}
   \ITEE{#3}{pn2006a}{\bibITEM{#2}{#1} \mypaplist{pn2}{}}
   \ITEE{#3}{pn2006b}{\bibITEM{#2}{#1} \mypaplist{pn3}{}}
   \ITEE{#3}{pn2007}{\bibITEM{#2}{#1} \mypaplist{pn4}{}}
   \ITEE{#3}{pn2008a}{\bibITEM{#2}{#1} \mypaplist{pn5}{}}
   \ITEE{#3}{pn2008b}{\bibITEM{#2}{#1} \mypaplist{pn6}{}}
   \ITEE{#3}{pn2009a}{\bibITEM{#2}{#1} \mypaplist{pn7}{}}
   \ITEE{#3}{pn2009b}{\bibITEM{#2}{#1} \mypaplist{pn8}{}}
   \ITEE{#3}{pn2009c}{\bibITEM{#2}{#1} \mypaplist{pn9}{}}
   \ITEE{#3}{pn2010a}{\bibITEM{#2}{#1} \mypaplist{pn10}{}}
   \ITEE{#3}{pn2010b}{\bibITEM{#2}{#1} \mypaplist{pn11}{}}
   \ITEE{#3}{pn2011a}{\bibITEM{#2}{#1} \mypaplist{pn12}{}}
   \ITEE{#3}{pn2011b}{\bibITEM{#2}{#1} \mypaplist{pn13}{}}
   \ITEE{#3}{pn2011c}{\bibITEM{#2}{#1} \mypaplist{pn14}{}}
   \ITEE{#3}{pn2011d}{\bibITEM{#2}{#1} \mypaplist{pn15}{}}
   \ITEE{#3}{pn2011e}{\bibITEM{#2}{#1} \mypaplist{pn16}{}}
   \ITEE{#3}{pn2011f}{\bibITEM{#2}{#1} \mypaplist{pn17}{}}
   \ITEE{#3}{pn2012a-}{\bibITEM{#2}{#1} \mypaplist[*]{pn18}{}}
   \ITEE{#3}{pn2012a}{\bibITEM{#2}{#1} \mypaplist{pn18}{}}
   \ITEE{#3}{pn2012b}{\bibITEM{#2}{#1} \mypaplist{pn19}{}}
   \ITEE{#3}{pn2012c}{\bibITEM{#2}{#1} \mypaplist{pn20}{}}
   \ITEE{#3}{pn2012d}{\bibITEM{#2}{#1} \mypaplist{pn21}{}}
   \ITEE{#3}{pn2012e}{\bibITEM{#2}{#1} \mypaplist{pn22}{}}
   \ITEE{#3}{pnXXXXb}{
      \bibITEM{#2}{#1} \mypaplist{pnX2}{}}
   \ITEE{#3}{pnXXXXc}{
      \bibITEM{#2}{#1} \mypaplist{pnX3}{}}
   \ITEE{#3}{pnXXXXd}{
      \bibITEM{#2}{#1} \mypaplist{pnX17}{}}
   \ITEE{#3}{MNiezgoda1998}{
      \BIB{#2}{M. Niezgoda}
         {Group majorization and Schur type inequalities}
         {\jRN{LAA}}{268}{1998}{9--30}{#1}}
   \ITEE{#3}{MNiezgoda1998a}{
      \BIB{#2}{M. Niezgoda}{An analytical 
         characterization of effective and of irreducible groups inducing cone orderings}
         {\jRN{LAA}}{269}{1998}{105--114}{#1}}
   \ITEE{#3}{MNiezgoda,TYTam2001}{
      \BIB{#2}{M. Niezgoda and T.Y. Tam}
         {On norm property of $G(c)$\hyp{}radii and Eaton triples}
         {\jRN{LAA}}{336}{2001}{119--130}{#1}}
   \ITEE{#3}{LNNikolskaya1974}{
      \BIB{#2}{L.N. Nikol'skaya}{Structure of the point spectrum 
         of a linear operator \textup{(Russian)}}{\jRN{MZamet}}{15}{1974}{149--158; 
         English translation in \jRN{MNotes} \textbf{15} (1974), 83--87}{#1}}
   \ITEE{#3}{APazy1983}{
      \BIb{#2}{A. Pazy}{Semigroups of Linear Operators and Applications 
         to Partial Differential Equations \textup{(Applied Mathematical Sciences, vol. 44)}}
         {Springer\hyp{}Verlag, New York}{1983}{#1}}
   \ITEE{#3}{CPearcy,NSalinas1974}{
      \BIB{#2}{C. Pearcy and N. Salinas}
         {Finite-dimensional representations of separable $\CCc^*$-algebras}
         {\jRN{NAMS}}{21}{1974}{A-376}{#1}}
   \ITEE{#3}{APelc1982}{
      \BIB{#2}{A. Pelc}
         {Semiregular invariant measures on abelian groups}
         {\jRN{PAMS}}{86}{1982}{423--426}{#1}}
   \ITEE{#3}{RPenrose1955}{
      \BIB{#2}{R. Penrose}
         {A generalized inverse for matrices}
         {\jRN{ProcCambPhS}}{51}{1955}{406--413}{#1}}
   \ITEE{#3}{VPestov2006}{
      \BIb{#2}{V. Pestov}{Dynamics of infinite\hyp{}dimensional 
         groups. The Ramsey\hyp{}Dvoretzky\hyp{}Milman phenomenon}
         {University Lecture Series \textbf{40}, AMS, Providence, RI}{2006}{#1}}
   \ITEE{#3}{VPestov2007}{
      \BiB{#2}{V. Pestov}{Forty\hyp{}plus annotated 
         questions about large topological groups}{in:}{Open Problems in Topology II}
         {Elliot Pearl (editor), Elsevier B.V., Amsterdam}{2007}{439--450}{#1}}
   \ITEE{#3}{PVPetersen1993}{
      \BiB{#2}{P.V. Petersen}{Gromov\hyp{}Hausdorff convergence 
         of metric spaces}{in book:}{Differential Geometry: Riemannian Geometry 
         (Los Angeles, CA, 1990)}{Amer. Math. Soc., Providence, RI}{1993}{489--504}{#1}}
   \ITEE{#3}{HPfister1985}{
      \BIB{#2}{H. Pfister}
         {Continuity of the inverse}
         {\jRN{PAMS}}{95}{1985}{312--314}{#1}}
   \ITEE{#3}{LPontrjagin1946}{
      \BIb{#2}{L. Pontrjagin}
         {Topological Groups}
         {Princeton University Press, Princeton}{1946}{#1}}
   \ITEE{#3}{DRamachandran,MMisiurewicz1982}{
      \BIB{#2}{D. Ramachandran and M. Misiurewicz}
         {Hopf's theorem on invariant measures for a group of transformations}
         {\jRN{SM}}{74}{1982}{183--189}{#1}}
   \ITEE{#3}{WRoelcke,SDierolf1981}{
      \BIb{#2}{W. Roelcke and S. Dierolf}
         {Uniform Structures on Topological Groups and Their Quotients}
         {McGraw Hill, New York}{1981}{#1}}
   \ITEE{#3}{JMRosenblatt1974}{
      \BIB{#2}{J.M. Rosenblatt}
         {Equivalent invariant measures}
         {\jRN{IsraelJM}}{17}{1974}{261--270}{#1}}
   \ITEE{#3}{SRosset1976}{
      \BIB{#2}{S. Rosset}
         {A new proof of the Amitsur-Levitski identity}
         {\jRN{IsraelJM}}{23}{1976}{187--188}{#1}}
   \ITEE{#3}{HLRoyden1963}{
      \BIb{#2}{H.L. Royden}
         {Real Analysis}
         {The Macmillan Co., New York}{1963}{#1}}
   \ITEE{#3}{WRudin1962}{
      \BIb{#2}{W. Rudin}{Fourier Analysis on Groups 
         \textup{(Interscience Tracts in Pure and Applied Mathematics, Number 12)}}
         {Interscience Publishers, New York}{1962}{#1}}
   \ITEE{#3}{WRudin1991}{
      \BIb{#2}{W. Rudin}
         {Functional Analysis}
         {McGraw\hyp{}Hill Science}{1991}{#1}}
   \ITEE{#3}{TSaito1972}{
      \BiB{#2}{T. Sait\^{o}}{Generations of von Neumann algebras}
         {Lecture Notes in Math. vol. 247}{\textup{(}Lecture on Operator Algebras\textup{)}}
         {Springer, Berlin\hyp{}Heidelberg\hyp{}New York}{1972}{435--531}{#1}}
   \ITEE{#3}{KSakai,MYaguchi2003}{
      \BIB{#2}{K. Sakai and M. Yaguchi}{Characterizing 
         manifolds modeled on certain dense subspaces of non\hyp{}separable Hilbert spaces}
         {\jRN{TsukubaJM}}{27}{2003}{143--159}{#1}}
   \ITEE{#3}{SSakai1971}{
      \BIb{#2}{S. Sakai}
         {$\CCc^*$\hyp{}Algebras and $\WWw^*$\hyp{}Algebras}
         {Springer\hyp{}Verlag, Berlin\hyp{}Heidelberg\hyp{}New York}{1971}{#1}}
   \ITEE{#3}{RSchori1971}{
      \BIB{#2}{R. Schori}
         {Topological stability for infinite\hyp{}dimensional manifolds}
         {\jRN{ComposM}}{23}{1971}{87--100}{#1}}
   \ITEE{#3}{JTSchwartz1967}{
      \BIb{#2}{J.T. Schwartz}
         {$\WWw^*$\hyp{}algebras}
         {Gordon and Breach, Science Publishers Inc., New York\hyp{}London\hyp{}Paris}{1967}{#1}}
   \ITEE{#3}{ZSemadeni1971}{
      \BIb{#2}{Z. Semadeni}
         {Banach Spaces of Continuous Functions (Vol. I)}
         {\jRN{PWN}}{1971}{#1}}
   \ITEE{#3}{JPSerre1951}{
      \BIB{#2}{J.-P. Serre}
         {Homologie singuli\`{e}re des espaces fibr\'{e}s}
         {\jRN{AnnM}}{54}{1951}{425--505}{#1}}
   \ITEE{#3}{DSherman2007}{
      \BIB{#2}{D. Sherman}
         {On the dimension theory of von Neumann algebras}
         {\jRN{MScand}}{101}{2007}{123--147}{#1}}
   \ITEE{#3}{DSherman2007a}{
      \BIB{#2}{D. Sherman}
         {Unitary orbits of normal operators in von Neumann algebras}
         {\jRN{JReinAngM}}{605}{2007}{95--132}{#1}}
   \ITEE{#3}{SAShkarin1999}{
      \BIB{#2}{S.A. Shkarin}
         {Universal Abelian topological groups}
         {\jRN{SbM}}{190}{1999}{1059--1076}{#1}}
   \ITEE{#3}{WSierpinski1920}{
      \BIB{#2}{W. Sierpi\'{n}ski}
         {Sur une propri\'{e}t\'{e} topologique des ensembles d\'{e}nombrables denses en soi}
         {\jRN{FM}}{1}{1920}{11--16}{#1}}
   \ITEE{#3}{WSierpinski1928}{
      \BIB{#2}{W. Sierpi\'{n}ski}
         {Sur les projections des ensembles compl\'{e}mentaires aux ensembles \textup{(A)}}
         {\jRN{FM}}{11}{1928}{117--122}{#1}}
   \ITEE{#3}{MSlocinski1980}{
      \BIB{#2}{M. S\l{}oci\'{n}ski}
         {On the Wold\hyp{}type decomposition of a pair of commuting isometries}
         {\jRN{APM}}{37}{1980}{255--262}{#1}}
   \ITEE{#3}{OGSmolyanov,SAShkarin1999}{
      \BiB{#2}{O.G. Smolyanov and S.A. Shkarin}{On the structure of spectra of linear operators 
         in Hilbert space}{in:}{Selected questions in mathematics, mechanics, and their 
         applications}{Moscow State University, Moscow}{1999}{289-–302}{#1}}
   \ITEE{#3}{OGSmolyanov,SAShkarin2001}{
      \BIB{#2}{O.G. Smolyanov and S.A. Shkarin}{On the structure of the spectra 
         of linear operators in Banach spaces \textup{(Russian)}}{\jRN{MSb}}{192}
         {2001}{99--114}{#1} English translation: \jRN{SbM} \textbf{192} (2001), 577--591.}
   \ITEE{#3}{RCSteinlage1975}{
      \BIB{#2}{R.C. Steinlage}
         {On Haar measure in locally compact $T_2$ spaces}
         {\jRN{AmJM}}{97}{1975}{291--307}{#1}}
   \ITEE{#3}{JStochel,FHSzafraniec1989}{
      \BIB{#2}{J. Stochel and F.H. Szafraniec}
         {On normal extensions of unbounded operators. III. Spectral properties}
         {\jRN{PublRIMSKyoto}}{25}{1989}{105--139}{#1}}
   \ITEE{#3}{JStochel,FHSzafraniec1989a}{
      \BIB{#2}{J. Stochel and F.H. Szafraniec}
         {The normal part of an unbounded operator}
         {\jRN{ProcKonink}}{92}{1989}{495--503}{#1}}
   \ITEE{#3}{AHStone1962}{
      \BIB{#2}{A.H. Stone}
         {Absolute $\FFf_{\sigma}$\hyp{}spaces}
         {\jRN{PAMS}}{13}{1962}{495--499}{#1}}
   \ITEE{#3}{AHStone1962a}{
      \BIB{#2}{A.H. Stone}
         {Non\hyp{}separable Borel sets}
         {\jRN{DissM}}{28}{1962}{41 pages}{#1}}
   \ITEE{#3}{AHStone1972}{
      \BIB{#2}{A.H. Stone}
         {Non\hyp{}separable Borel sets II}
         {\jRN{GTopA}}{2}{1972}{249--270}{#1}}
   \ITEE{#3}{MHStone1937}{
      \BIB{#2}{M.H. Stone}
         {Application of the theory of Boolean rings to general topology}
         {\jRN{TAMS}}{41}{1937}{375--481}{#1}}
   \ITEE{#3}{MHStone1948}{
      \BIB{#2}{M.H. Stone}
         {The generalized Weierstrass approximation theorem}
         {\jRN{MMag}}{21}{1948}{167--184}{#1}}
   \ITEE{#3}{RAStruble1974}{
      \BIB{#2}{R.A. Struble}
         {Metrics in locally compact groups}
         {\jRN{ComposM}}{28}{1974}{217--222}{#1}}
   \ITEE{#3}{RGSwan1963}{
      \BIB{#2}{R.G. Swan}
         {An application of graph theory to algebra}
         {\jRN{PAMS}}{14}{1963}{367--373}{#1}}
   \ITEE{#3}{RGSwan1969}{
      \BIB{#2}{R.G. Swan}
         {Correction to ``An application of graph theory to algebra''}
         {\jRN{PAMS}}{21}{1969}{379--380}{#1}}
   \ITEE{#3}{BSz-Nagy1947}{
      \BIB{#2}{B. Sz.\hyp{}Nagy}
         {On uniformly bounded linear transformations in Hilbert space}
         {\jRN{ActaSM}}{11}{1947}{152--157}{#1}}
   \ITEE{#3}{WSzymanski1974}{
      \BIB{#2}{W. Szyma\'{n}ski}
         {Decompositions of operator-valued functions in Hilbert spaces}
         {\jRN{SM}}{50}{1974}{265--280}{#1}}
   \ITEE{#3}{WTakahashi1970}{
      \BIB{#2}{W. Takahashi}
         {A convexity in metric space and nonexpansive mappings, I}
         {\jRN{KodaiMSemRep}}{22}{1970}{142--149}{#1}}
   \ITEE{#3}{MTakesaki2002}{
      \BIb{#2}{M. Takesaki}{Theory 
         of Operator Algebras I \textup{(Encyclopaedia of Mathematical Sciences, Volume 124)}}
         {Springer\hyp{}Verlag, Berlin\hyp{}Heidelberg\hyp{}New York}{2002}{#1}}
   \ITEE{#3}{MTakesaki2003}{
      \BIb{#2}{M. Takesaki}{Theory 
         of Operator Algebras II \textup{(Encyclopaedia of Mathematical Sciences, Volume 125)}}
         {Springer\hyp{}Verlag, Berlin\hyp{}Heidelberg\hyp{}New York}{2003}{#1}}
   \ITEE{#3}{MTakesaki2003a}{
      \BIb{#2}{M. Takesaki}{Theory 
         of Operator Algebras III \textup{(Encyclopaedia of Mathematical Sciences, Volume 127)}}
         {Springer\hyp{}Verlag, Berlin\hyp{}Heidelberg\hyp{}New York}{2003}{#1}}
   \ITEE{#3}{TYTam1999}{
      \BIB{#2}{T.Y. Tam}
         {An extension of a result of Lewis}
         {\jRN{ELA}}{5}{1999}{1--10}{#1}}
   \ITEE{#3}{TYTam2000}{
      \BIB{#2}{T.Y. Tam}
         {Group majorization, Eaton triples and numerical range}
         {\jRN{LMLA}}{47}{2000}{11--28}{#1}}
   \ITEE{#3}{TYTam2002}{
      \BIB{#2}{T.Y. Tam}
         {Generalized Schur\hyp{}concave functions and Eaton triples}
         {\jRN{LMLA}}{50}{2002}{113--120}{#1}}
   \ITEE{#3}{TYTam,WCHill2001}{
      \BIB{#2}{T.Y. Tam and W.C. Hill}
         {On $G$\hyp{}invariant norms}
         {\jRN{LAA}}{331}{2001}{101--112}{#1}}
   \ITEE{#3}{AFTiman,IAVestfrid1983}{
      \BIB{#2}{A.F. Timan and I.A. Vestfrid}
         {Any separable ultrametric space can be isometrically imbedded in $\ell_2$}
         {\jRN{FAA}}{17}{1983}{70--71}{#1}}
   \ITEE{#3}{VTimofte2005}{
      \BIB{#2}{V. Timofte}
         {Stone\hyp{}Weierstrass theorems revisited}
         {\jRN{JAT}}{136}{2005}{45--59}{#1}}
   \ITEE{#3}{JTomiyama1958}{
      \BIB{#2}{J. Tomiyama}
         {Generalized dimension function for $\WWw^*$\hyp{}algebras of infinite type}
         {\jRN{TohokuMJ} (2)}{10}{1958}{121--129}{#1}}
   \ITEE{#3}{HTorunczyk1970}{
      \BIB{#2}{H. Toru\'{n}czyk}
         {Remarks on Anderson's paper ``On topological infinite deficiency''}
         {\jRN{FM}}{66}{1970}{393--401}{#1}}
   \ITEE{#3}{HTorunczyk1970a}{
      \BIb{#2}{H. Toru\'{n}czyk}
         {$G$\hyp{}$K$\hyp{}absorbing and skeletonized sets in metric spaces}
         {Ph.D. thesis, Inst. Math. Polish Acad. Sci., Warszawa}{1970}{#1}}
   \ITEE{#3}{HTorunczyk1972}{
      \BIB{#2}{H. Toru\'{n}czyk}
         {A short proof of Hausdorff's theorem on extending metrics}
         {\jRN{FM}}{77}{1972}{191--193}{#1}}
   \ITEE{#3}{HTorunczyk1974}{
      \BIB{#2}{H. Toru\'{n}czyk}
         {Absolute retracts as factors of normed linear spaces}
         {\jRN{FM}}{86}{1974}{53--67}{#1}}
   \ITEE{#3}{HTorunczyk1975}{
      \BIB{#2}{H. Toru\'{n}czyk}
         {On Cartesian factors and the topological classification of linear metric spaces}
         {\jRN{FM}}{88}{1975}{71--86}{#1}}
   \ITEE{#3}{HTorunczyk1978}{
      \BIB{#2}{H. Toru\'{n}czyk}{Concerning 
         locally homotopy negligible sets and characterization of $\ell_2$\hyp{}manifolds}
         {\jRN{FM}}{101}{1978}{93--110}{#1}}
   \ITEE{#3}{HTorunczyk1980}{
      \BiB{#2}{H. Toru\'{n}czyk}{Characterization of infinite\hyp{}dimensional manifolds}{in:}
         {Proceedings of the International Conference on Geometric Topology (Warsaw, 1978)}
         {\jRN{PWN}}{1980}{431--437}{#1}}
   \ITEE{#3}{HTorunczyk1981}{
      \BIB{#2}{H. Toru\'{n}czyk}
         {Characterizing Hilbert space topology}
         {\jRN{FM}}{111}{1981}{247--262}{#1}}
   \ITEE{#3}{HTorunczyk1985}{
      \BIB{#2}{H. Toru\'{n}czyk}
         {A correction of two papers concerning Hilbert manifolds}
         {\jRN{FM}}{125}{1985}{89--93}{#1}}
   \ITEE{#3}{KTsuda1985}{
      \BIB{#2}{K. Tsuda}
         {A note on closed embeddings of finite dimensional metric spaces}
         {\jRN{BLondMS}}{17}{1985}{273--278}{#1}}
   \ITEE{#3}{PSUrysohn1925}{
      \BIB{#2}{P.S. Urysohn}
         {Sur un espace m\'{e}trique universel}
         {\jRN{CRASParis}}{180}{1925}{803--806}{#1}}
   \ITEE{#3}{PSUrysohn1927}{
      \BIB{#2}{P.S. Urysohn}
         {Sur un espace m\'{e}trique universel}
         {\jRN{BullSM}}{51}{1927}{43--64, 74--96}{#1}}
   \ITEE{#3}{VVUspenskij1986}{
      \BIB{#2}{V.V. Uspenskij}
         {A universal topological group with a countable basis}
         {\jRN{FAA}}{20}{1986}{86--87}{#1}}
   \ITEE{#3}{VVUspenskij1990}{
      \BIB{#2}{V.V. Uspenskij}
         {On the group of isometries of the Urysohn universal metric space}
         {\jRN{CMUC}}{31}{1990}{181--182}{#1}}
   \ITEE{#3}{VVUspenskij2004}{
      \BIB{#2}{V.V. Uspenskij}
         {The Urysohn universal metric space is homeomorphic to a Hilbert space}
         {\jRN{TopA}}{139}{2004}{145--149}{#1}}
   \ITEE{#3}{VVUspenskij2008}{
      \BIB{#2}{V.V. Uspenskij}
         {On subgroups of minimal topological groups}
         {\jRN{TopA}}{155}{2008}{1580--1606}{#1}}
   \ITEE{#3}{VSVaradarajan1963}{
      \BIB{#2}{V.S. Varadarajan}
         {Groups of automorphisms of Borel spaces}
         {\jRN{TAMS}}{109}{1963}{191--220}{#1}}
   \ITEE{#3}{AMVershik1998}{
      \BIB{#2}{A.M. Vershik}{The universal Urysohn space, 
         Gromov's metric triples, and random metrics on the series of natural numbers}
         {\jRN{UspekhiMN}}{53}{1998}{57--64}{#1} English translation: \jRN{RussMS}{} \textbf{53} 
         (1998), 921--928. Correction: \jRN{UspekhiMN}{} \textbf{56} (2001), p. 207. English 
         translation: \jRN{RussMS}{} \textbf{56} (2001), p. 1015.}
   \ITEE{#3}{AMVershik2002}{
      \BIb{#2}{A.M. Vershik}{Random metric spaces and the universal Urysohn space}
         {Fundamental Mathematics Today. 10th anniversary of the Independent Moscow University. 
         MCCME Publ.}{2002}{#1}}
   \ITEE{#3}{NWeaver1999}{
      \BIb{#2}{N. Weaver}
         {Lipschitz Algebras}
         {World Scientific}{1999}{#1}}
   \ITEE{#3}{JWeidmann1980}{
      \BIb{#2}{J. Weidmann}
         {Linear Operators in Hilbert Spaces}
         {(Graduate Texts in Mathematics, vol. 68) Springer\hyp{}Verlag New York Inc.}{1980}{#1}}
   \ITEE{#3}{JEWest1969}{
      \BIB{#2}{J.E. West}
         {Approximating homotopies by isotopies in Fr\'{e}chet manifolds}
         {\jRN{BAMS}}{75}{1969}{1254--1257}{#1}}
   \ITEE{#3}{JEWest1969a}{
      \BIB{#2}{J.E. West}
         {Fixed\hyp{}point sets of transformation groups on infinite\hyp{}product spaces}
         {\jRN{PAMS}}{21}{1969}{575--582}{#1}}
   \ITEE{#3}{JEWest1970}{
      \BIB{#2}{J.E. West}
         {The ambient homeomorphy of infinite\hyp{}dimensional Hilbert spaces}
         {\jRN{PacJM}}{34}{1970}{257--267}{#1}}
   \ITEE{#3}{JHCWhitehead1949}{
      \BIB{#2}{J.H.C. Whitehead}
         {Combinatorial homotopy I}
         {\jRN{BAMS}}{55}{1949}{213--245}{#1}}
   \ITEE{#3}{GTWhyburn1942}{
      \BIb{#2}{G. T. Whyburn}
         {Analytic Topology}
         {Amer. Math. Soc. Colloquium Publications (vol. XXVIII), New York}{1942}{#1}}
   \ITEE{#3}{RWilliamson,LJanos1987}{
      \BIB{#2}{R. Williamson and L. Janos}
         {Constructing metrics with the Heine\hyp{}Borel property}
         {\jRN{PAMS}}{100}{1987}{567--573}{#1}}
   \ITEE{#3}{WWogen1969}{
      \BIB{#2}{W. Wogen}
         {On generators for von Neumann algebras}
         {\jRN{BAMS}}{75}{1969}{95--99}{#1}}
   \ITEE{#3}{RYTWong1967}{
      \BIB{#2}{R.Y.T. Wong}
         {On homeomorphisms of certain infinite dimensional spaces}
         {\jRN{TAMS}}{128}{1967}{148--154}{#1}}
   \ITEE{#3}{LYang,JZhang1987}{
      \BIB{#2}{L. Yang and J. Zhang}
         {Average distance constants of some compact convex space}
         {\jRN{JChinUST}}{17}{1987}{17--23}{#1}}
   \ITEE{#3}{PZakrzewski1993}{
      \BIB{#2}{P. Zakrzewski}
         {The existence of invariant $\sigma$\hyp{}finite measures for a group of transformations}
         {\jRN{IsraelJM}}{83}{1993}{275--287}{#1}}
   \ITEE{#3}{PZakrzewski2002}{
      \BIb{#2}{P. Zakrzewski}
         {Measures on Algebraic\hyp{}Topological Structures, Handbook of Measure Thoery}
         {E. Pap, ed., Elsevier, Amsterdam}{2002, 1091--1130}{#1}}
   \ITEE{#3}{WZelazko1960}{
      \BIB{#2}{W. \.{Z}elazko}
         {A theorem on $B_0$ division algebras}
         {\jRN{BPAS}}{8}{1960}{373--375}{#1}}
   \ITEE{#3}{KZhu2000}{
      \BIB{#2}{K. Zhu}
         {Operators in Cowen\hyp{}Douglas classes}
         {\jRN{IllinoisJM}}{44}{2000}{767--783}{#1}}
   }

\newcommand{\mypaplist}[3][]{
   \ITEE{#2}{pn1}{
      \myBIB{Separate and joint similarity to families of normal operators}
         {\jRN{SM}}{149}{2002}{39--62}{#3}}
   \ITEE{#2}{pn2}{
      \myBIB{Locally arcwise connected metrizable spaces with the fixed point property are 
         complete\hyp{}metrizable}{\jRN{TopA}}{153}{2006}{1639--1642}{#3}}
   \ITEE{#2}{pn3}{
      \myBIB{Invariant measures for equicontinuous semigroups of continuous transformations 
         of a compact Hausdorff space}{\jRN{TopA}}{153}{2006}{3373--3382}{#3}}
   \ITEE{#2}{pn4}{
      \myBIB{Approximation of the Hausdorff distance by the distance of continuous surjections}
         {\jRN{TopA}}{154}{2007}{655--664}{#3}}
   \ITEE{#2}{pn5}{
      \myBIB{Generalized Haar integral}
         {\jRN{TopA}}{155}{2008}{1323--1328}{#3}}
   \ITEE{#2}{pn6}{
      \myBIB{Integration and Lipschitz functions}
         {\jRN{RCMP}}{57}{2008}{391--399}{#3}}
   \ITEE{#2}{pn7}{
      \myBIB{Canonical Banach function spaces generated by Urysohn universal spaces. Measures 
         as Lipschitz maps}{\jRN{SM}}{192}{2009}{97--110}{#3}}
   \ITEE{#2}{pn8}{
      \myBIB{Urysohn universal spaces as metric groups of exponent $2$}
         {\jRN{FM}}{204}{2009}{1--6}{#3}}
   \ITEE{#2}{pn9}{
      \myBIB{Central subsets of Urysohn universal spaces}
         {\jRN{CMUC}}{50}{2009}{445--461}{#3}}
   \ITEE{#2}{pn10}{
      \myBIB{Ultra\hyp{}$\mM$\hyp{}separability}
         {\jRN{TopA}}{157}{2010}{669--673}{#3}}
   \ITEE{#2}{pn11}{
      \myBIB{Functor of extension of $\Lambda$\hyp{}isometric maps between central subsets 
         of the unbounded Urysohn universal space}{\jRN{CMUC}}{51}{2010}{541--549}{#3}}
   \ITEE{#2}{pn12}{
      \myBIB[P. Niemiec and T.Y. Tam]{A representation of $G$\hyp{}invariant norms for Eaton 
         triple}{\jRN{JCA}}{18}{2011}{59--65}{#3}}
   \ITEE{#2}{pn13}{
      \myBIB{Topological structure of Urysohn universal spaces}
         {\jRN{TopA}}{158}{2011}{352--359}{#3}}
   \ITEE{#2}{pn14}{
      \myBIB{A note on invariant measures}
         {\jRN{OpusM}}{31}{2011}{425--431}{#3}}
   \ITEE{#2}{pn15}{
      \myBIB{Strengthened Stone\hyp{}Weierstrass type theorem}
         {\jRN{OpusM}}{31}{2011}{645--650}{#3}}
   \ITEE{#2}{pn16}{
      \myBIB{Generalized absolute values and polar decompositions of a bounded operator}
         {\jRN{IEOT}}{71}{2011}{151--160}{#3}}
   \ITEE{#2}{pn17}{
      \myBIB{Functor of extension of contractions on Urysohn universal spaces}
         {\jRN{ACS}}{19}{2011}{959--967}{#3}}
   \ITEE{#2}{pn18}{
      \myBIB{A note on ANR's}{\jRN{TopA}}{159}{2012}{\ITEE{#1}{}{315--321; erratum: p.~2232}
         \ITEE{#1}{*}{315--321}\ITEE{#1}{**}{p.~2232}}{#3}}
   \ITEE{#2}{pn19}{
      \myBIB{Unitary equivalence and decompositions of finite systems of closed densely defined 
         operators in Hilbert spaces}{\jRN{DissM}}{482}{2012}{106~pp}{#3}}
   \ITEE{#2}{pn20}{
      \myBIB{Problem with almost everywhere equality}
         {\jRN{APM}}{104}{2012}{105--108}{#3}}
   \ITEE{#2}{pn21}{
      \myBIB{Borel parts of the spectrum of an operator and of the operator algebra 
         of a separable Hilbert space}{\jRN{SM}}{208}{2012}{77--85}{#3}}
   \ITEE{#2}{pn22}{
      \myBIB{Normed topological pseudovector groups}
         {\jRN{ACS}}{20}{2012}{303--322}{#3}}
   \ITEE{#2}{pn1001}{
      \myBIB{Normal systems over ANR's, rigid embeddings and nonseparable absorbing sets}
         {\jRN{ActaMSinES}}{}{2012}{\ITE{\equal{#1}{*}}{on-line first}{doi: \texttt{10.1007/s10114-012-0709-8}}}{#3}}
   \ITEE{#2}{pnX2}{
      \myBAPP{Functor of continuation in Hilbert cube and Hilbert space}
         {submitted to \jRN{FM}\oNlINE{#1}{(\texttt{http://arxiv.org/abs/1107.1386})}}{#3}}
   \ITEE{#2}{pnX3}{
      \myBAPP{Norm closures of orbits of bounded operators}{accepted for publication 
         in \jRN{JOT}\oNlINE{#1}{(\texttt{http://arxiv.org/abs/1107.1505})}}{#3}}
   \ITEE{#2}{pnX6}{
      \myBAPP{Extending maps in Hilbert manifolds}
         {accepted for publication in \jRN{BPAS}}{#3}}
   \ITEE{#2}{pnX7}{
      \myBAPP{Spaces of measurable functions}
         {submitted to \jRN{CEurJM}\oNlINE{#1}{(\texttt{http://arxiv.org/abs/1107.1495})}}{#3}}
   \ITEE{#2}{pnX10}{
      \myBAPP{Central points and measures and dense subsets of compact metric spaces}
         {accepted for publication in \jRN{TMNA}}{#3}}
   \ITEE{#2}{pnX12}{
      \myBAPP{Ultrametrics, extending of Lipschitz maps and nonexpansive selections}
         {accepted for publication in \jRN{HJM}}{#3}}
   \ITEE{#2}{pnX15}{
      \myBAPP{Universal valued Abelian groups}
         {submitted to \jRN{AdvM}\oNlINE{#1}{(\texttt{http://arxiv.org/abs/1103.1623})}}{#3}}
   \ITEE{#2}{pnX17}{
      \myBAPP{Isometry groups of proper metric spaces}
         {submitted to \jRN{TAMS}\oNlINE{#1}{(\texttt{http://arxiv.org/abs/1201.5675})}}{#3}}
   \ITEE{#2}{pnX18}{
      \myBAPP{Isometry groups among topological groups}
         {submitted to \jRN{ComposM}\oNlINE{#1}{(\texttt{http://arxiv.org/abs/1202.3368})}}{#3}}
   \ITEE{#2}{pnX19}{
      \myBAPP{$\CCc^*$\hyp{}algebras of pure type $I_n$}
         {submitted to \jRN{IsraelJM}\oNlINE{#1}{(\texttt{http://arxiv.org/abs/1203.0857})}}{#3}}
   }

\begin{document}

\title{Isometry groups of proper metric spaces}
\myData
\begin{abstract}
Given a locally compact Polish space $X$, a necessary and sufficient condition for a group $G$
of homeomorphisms of $X$ to be the full isometry group of $(X,d)$ for some proper metric $d$ on $X$
is given. It is shown that every locally compact Polish group $G$ acts freely on $G \times X$
as the full isometry group of $G \times X$ with respect to a certain proper metric on $G \times X$,
where $X$ is an arbitrary locally compact Polish space having more than one point such that
$(\card(G),\card(X)) \neq (1,2)$. Locally compact Polish groups which act effectively and almost
transitively on complete metric spaces as full isometry groups are characterized. Locally compact
Polish non-Abelian groups on which every left invariant metric is automatically right invariant are
characterized and fully classified. It is demonstrated that for every locally compact Polish space
$X$ having more than two points the set of all proper metrics $d$ such that $\Iso(X,d) = \{\id_X\}$
is dense in the space of all proper metrics on $X$.
\end{abstract}
\maketitle

\SECT{Introduction}

Everywhere in this paper, a metric on a metrizable space (or a metric space) is \textit{proper} iff
all closed balls (with respect to this metric) are compact. A topological group or space is
\textit{Polish} if it is completely metrizable and separable. The set of all proper metrics
on a locally compact Polish space $X$ which induce the topology of $X$ is denoted by $\Metr_c(X)$
(for such $X$, $\Metr_c(X)$ is nonempty, see e.g. \cite{w-j}). The neutral element of a group $G$ is
denoted by $e_G$. The identity map on a set $X$ is denoted by $\id_X$. For every metric space
$(Y,\varrho)$, $\Iso(Y,\varrho)$ stands for the group of all (bijective) isometries
of $(Y,\varrho)$.\par
Isometry groups (equipped with the topology of pointwise convergence) of separable complete metric
spaces are useful `models' for studying Polish groups. On the one hand, they are defined and appear
in topology quite naturally. On the other hand, thanks to the result of Gao and Kechris \cite{g-k},
every Polish group may be `represented' as (that is, is isomorphic to) the (full) isometry group
$\Iso(X,d)$ of some separable complete metric space $(X,d)$. It may be of great importance to know
how to build the space $(X,d)$ (or how `nice' the topological space $X$ can be) such that
$\Iso(X,d)$ is isomorphic to a given Polish group $G$. Natural questions which arise when dealing
with this issue, are the following:
\begin{enumerate}[(Q1)]
\item Does there exist compact $X$ for compact $G$~?
\item Does there exist locally compact $X$ for locally compact $G$~?
\item Is every Lie group isomorphic to the isometry group of a manifold (with respect to some
   compatible metric)?
\item (Melleray \cite{mel}) Is every compact Lie group the isometry group of some compact Riemannian
   manifold?
\item Can $X$ be (metrically) homogeneous? (That is, can $G$ act effectively and transitively on any
   $X$ as the full isometry group?)
\end{enumerate}
Melleray \cite{mel} improved the original proof of Gao and Kechris and answered in the affirmative
question (Q1). Later Malicki and Solecki \cite{ms!} solved problem (Q2) by showing that each locally
compact Polish group is (isomorphic to) the isometry group of some proper metric space. In all these
papers the construction of the crucial metric space $X$ is complicated and based on the techniques
of the so-called Kat\v{e}tov maps, and it may turn out that for a `nice' group $G$ (e.g. a connected
Lie group) the space $X$ contains a totally disconnected open (nonempty) subset. In the present
paper we deal with locally compact Polish groups and propose a new approach to the above topic.
Using new ideas, we solve in the affirmative problem (Q3) (which is closely related to (Q4)) and
characterize locally compact Polish groups which are isomorphic to (full) isometry groups acting
transitively on proper metric spaces (see \THM{trans}, especially points (i) and (ii))---this
answers question (Q5). Our main results in this direction are:

\begin{thm}{free}
Let $G$ be a locally compact Polish group and $X$ be a locally compact Polish space with
$(\card(G),\card(X)) \neq (1,2)$. Let $G \times X \ni (g,x) \mapsto g.x \in X$ be a (continuous)
proper \textbf{non-transitive} action of $G$ on $X$ such that for some point $\omega \in X$:
\begin{enumerate}[\upshape(F1)]
\item $G$ acts freely at $\omega$,
\item $G$ acts effectively on $X \setminus G.\omega$.
\end{enumerate}
Then there exists $d \in \Metr_c(X)$ such that $\Iso(X,d)$ consists precisely of all maps
of the form $x \mapsto a.x\ (a \in G)$.
\end{thm}

\begin{cor}{gx}
Let $(G,\cdot)$ be a locally compact Polish group and $X$ be a locally compact Polish space having
more than one point for which $(\card(G),\card(X)) \neq (1,2)$. There exists $d \in \Metr_c(G \times
X)$ such that $\Iso(G \times X,d)$ consists precisely of all maps of the form $(g,x) \mapsto (ag,x)\
(a \in G)$. In particular, $\Iso(G \times X,d)$ acts freely on $G \times X$ and is isomorphic
to $G$.
\end{cor}

\begin{thm}{trans}
For a locally compact Polish group $(G,\cdot)$ \tfcae
\begin{enumerate}[\upshape(i)]
\item there exists a complete metric space $(X,d)$ and an effective action $G \times X \ni (g,x)
   \mapsto g.x \in X$ such that $\Iso(X,d)$ consists precisely of all maps of the form $x \mapsto
   a.x\ (a \in G)$ and $G.b$ is dense in $X$ for some $b \in X$,
\item there exists a left invariant metric $\varrho \in \Metr_c(G)$ such that $\Iso(G,\varrho)$
   consists precisely of all natural left translations of $G$ on itself; in particular,
   $\Iso(G,\varrho)$ is isomorphic to $G$ and acts freely, transitively and properly on $G$,
\item one of the following two conditions is fulfilled:
   \begin{enumerate}[\upshape(a)]
   \item $G$ is Boolean; that is, $x^2 = e_G$ for each $x \in G$,
   \item $G$ is non-Abelian and there is \textbf{no} open normal Abelian subgroup $H$ of $G$
      of index $2$ such that $x^2 = p$ for any $x \in G \setminus H$, where $p \in H$ is
      (independent of $x$ and) such that $p^2 = e_G \neq p$.
   \end{enumerate}
\end{enumerate}
Moreover, in each of the following cases condition \textup{(ii)} is fulfilled:
\begin{itemize}
\item $G$ is non-solvable,
\item $G$ is non-Abelian and its center is either trivial or non-Boolean,
\item $G$ is non-Abelian and connected.
\end{itemize}
\end{thm}

\begin{pro}{Abel}
Every locally compact Polish Abelian group $(G,\cdot)$ admits an invariant metric $\varrho \in
\Metr_c(G)$ such that $\Iso(G,\varrho)$ consists precisely of all maps of the forms $x \mapsto a x$
and $x \mapsto a x^{-1}\ (a \in G)$.
\end{pro}

\COR{gx} answers in the affirmative question (Q3): every (separable) Lie group $G$ is isomorphic
to the isometry group of $G \times \{-1,1\}$ as well as of $G \times (\RRR/\ZZZ)$ and $G \times
\RRR$ (for certain proper metrics). All these spaces are manifolds (even Lie groups) and the first
two of them are compact provided so is the group $G$. What is more, if $G$ is a connected
non-Abelian Lie group, it is isomorphic to its own isometry group with respect to a certain left
invariant proper metric, by \THM{trans}.\par
It is worth mentioning that studying problems discussed above we have managed to find (and classify)
\textit{all} locally compact Polish non-Abelian groups on which every left invariant metric is
automatically right invariant (see \COR{left-right}). It is fascinating and unexpected that
up to isomorphism there are only countable number of such groups, each of them is of exponent $4$
and totally disconnected and among them only three are infinite: one compact, one discrete and one
noncompact non-discrete. Explicit descriptions are given in \REM{left-right}.\par
To formulate the main result of the paper (which characterizes isometry groups on proper metric
spaces), let us introduce a few necessary notions. Some of them are well-known.

\begin{dfn}{Ascoli}
Let $X$ and $Y$ be topological spaces and let $\FFf$ be a collection of transformations of $X$ into
$Y$.
\begin{itemize}
\item (cf. \cite[p.~162]{eng}) $\FFf$ is said to be \textit{evenly continuous} iff for any $x \in
   X$, $y \in Y$ and a neighbourhood $W$ of $y$ there exist neighbourhoods $U$ and $V$ of $x$ and
   $y$ (respectively) such that conditions $f \in \FFf$ and $f(x) \in V$ imply $f(U) \subset W$.
\item $\FFf$ is \textit{pointwise precompact} iff for each $x \in X$ the closure (in $Y$)
   of the set $\FFf.x := \{f(x)\dd\ f \in \FFf\}$ is compact.
\end{itemize}
\end{dfn}

\begin{dfn}{hull}
Let $X$ and $Y$ be arbitrary sets and let $\FFf$ be a collection of functions of $X$ into $Y$.
\textit{Symmetrized 2-hull} of $\FFf$ is the family $\hULL{\FFf}$ of all functions $g\dd X \to Y$
such that for any two points $x$ and $y$ of $X$ there is $f \in \FFf$ with $\{g(x),g(y)\} =
\{f(x),f(y)\}$. Notice that $\hULL{\FFf} \supset \FFf$.
\end{dfn}

Let $X$ be a locally compact Polish space. Since then $X$ is also $\sigma$-compact, the space
$\CCc(X,X)$ of all continuous functions of $X$ into itself is Polish when equipped with
the compact-open topology (that is, the topology of uniform convergence on compact sets). We shall
always consider $\CCc(X,X)$ with this topology. According to the Ascoli-type theorem (see e.g.
\cite[Theorem~3.4.20]{eng}), a set $\FFf \subset \CCc(X,X)$ is compact iff $\FFf$ is closed, evenly
continuous and pointwise precompact.\par
We are now ready to formulate the main result of the paper. (Notice that below it is not assumed
that the group $G$ is a topological group.)

\begin{thm}{main}
Let $X$ be a locally compact Polish space and $G$ be a group of homeomorphisms of $X$ such that
$(\card(G),\card(X)) \neq (1,2)$. \TFCAE
\begin{enumerate}[\upshape(i)]
\item there exists $d \in \Metr_c(X)$ such that $\Iso(X,d) = G$,
\item there exists a $G$-invariant metric in $\Metr_c(X)$ and for any $G$-invariant metric $d \in
   \Metr_c(X)$ and each $\epsi > 0$ there is $\varrho \in \Metr_c(X)$ such that $d \leqsl \varrho
   \leqsl (1 + \epsi) d$ and $\Iso(X,\varrho) = G$,
\item each of the following three conditions is fulfilled:
   \begin{enumerate}[\upshape({I}so1)]
   \item $G$ is closed in the space $\CCc(X,X)$,
   \item for every compact set $K$ in $X$ the family $\DDd_K = \{h \in G\dd\ h(K) \cap K \neq
      \varempty\}$ is evenly continuous and pointwise precompact,
   \item $\hULL{G} = G$.
   \end{enumerate}
\end{enumerate}
\end{thm}

Point (ii) in the above result asserts much more than just the existence of a proper metric
$\varrho$ with $\Iso(X,\varrho) = G$. It says that `almost' preserving the geometry of the space,
we may approximate any proper $G$-invariant metric by such metrics. The only difficult part
of \THM{main} is the implication `(iii)$\implies$(ii)'. We shall prove it in the next section
involving Baire's theorem and a very recent result by Abels, Manoussos and Noskov \cite{amn}.\par
Further consequences of \THM{main} are stated below.

\begin{cor}{filtr}
Let $X$ be a locally compact Polish space and let $\GGg = \{\Iso(X,d)\dd\ d \in \Metr_c(X)\}$. Then
for any nonempty family $\FFf \subset \GGg$, $\bigcap \FFf \in \GGg$.
\end{cor}

\begin{cor}{id}
For every locally compact Polish space $X$ having more than two points the set of all metrics $d \in
\Metr_c(X)$ for which $\Iso(X,d) = \{\id_X\}$ is dense (in the topology of uniform convergence
on compact subsets of $X \times X$) in $\Metr_c(X)$.
\end{cor}

The paper is organized as follows. The next section is devoted to the proof of \THM{main}. It also
contains short proofs of \THM{free} and \COR{gx}. The proofs of Corollaries~\ref{cor:filtr} and
\ref{cor:id} are left to the reader as simple exercises (they are immediate consequences
of \THM{main}). In the last section we study locally compact Polish groups $G$ which satisfy
condition (ii) of \THM{trans} and prove this result as well as \PRO{Abel}.\par
In general, isometry groups of locally compact separable complete metric spaces may not be locally
compact nor may not act properly on the underlying spaces. However, if $(X,d)$ is a connected
locally compact metric space, then the isometry group of $(X,d)$ is locally compact and acts
properly on $X$ (see \cite{d-w}). This result remains true when we replace the connectedness of $X$
by the properness of the metric $d$ (see \cite{g-k}). Other results in this topic may be found
in \cite{m-s}.\par
It is already known that a proper action of a locally compact Polish group $G$ on a locally compact
Polish space $X$ admits a $G$-invariant proper metric on $X$ (see \cite{amn}). It was proved much
earlier that every locally compact Polish group admits a left invariant proper metric (see
\cite{str}). These are the two main tools of our work. For other results on constructing proper
or $G$-invariant metrics the reader is referred to \cite{kos} and \cite{w-j}.

\begin{rem}{1-2}
A careful reader noticed that in some of results stated above a strange condition that
$(\card(G),\card(X)) \neq (1,2)$ appears. It is an interesting phenomenon that this trivial case ---
when $(\card(G),\card(X)) = (1,2)$ --- is the only exception for these theorems to hold. (We leave
it as a very simple exercise that the above condition is necessary whenever it appears
in the statement.)
\end{rem}

\subsection*{Notation and terminology.} In this paper all considered topological spaces (unless
otherwise stated) are Polish. By a \textit{map} we mean a continuous function. All isomorphisms
between topological groups as well as actions of topological groups on topological spaces are
assumed to be continuous (unless otherwise stated). We use the multiplicative notation for all
groups. Let $G$ be a topological group, $X$ be a topological space and let $G \times X \ni (g,x)
\mapsto g.x \in X$ be an action. We call $G$ a \textit{Boolean} group iff $g^2 = e_G$ for each
$g \in G$. (Every Boolean group is Abelian.) For $g \in G$, $x \in X$, $A \subset G$ and $B \subset
X$ we write $A.B := \{a.b\dd\ a \in A,\ b \in B\}$, $g.B := \{g\}.B$ and $A.x := A.\{x\}$.
The action is \textit{effective} (on a set $Y \subset X$) iff $g.x = x$ for each $x$ (belonging
to $Y$) implies $g = e_G$. It is \textit{free} (at a point $\omega \in X$) iff $g.x = x$ for some
$x \in X$ (for $x = \omega$) implies $g = e_G$. The action is \textit{transitive} (resp.
\textit{almost transitive}) iff $G.x = X$ (resp. iff $G.x$ is dense in $X$) for some $x \in X$.
Finally, it is \textit{proper}, if the map $\Phi\dd G \times X \ni (g,x) \mapsto (x,g.x) \in
X \times X$ is proper, that is, if $\Phi^{-1}(K)$ is compact for any compact $K \subset X \times X$.
This is equivalent to (see \cite{amn}):
\begin{itemize}
\item[($\star$)] for any compact set $K \subset X$, $D_K := \{g \in G\dd\ g.K \cap K \neq
   \varempty\}$ is compact.
\end{itemize}
A metric $d$ on $X$ is \textit{$G$-invariant} if $d(g.x,g.y) = d(x,y)$ for all $g \in G$ and $x, y
\in X$. When $G$ and $X$ are locally compact and Polish, and the action is proper, the set of all
$G$-invariant metrics $d \in \Metr_c(X)$ is nonempty (\cite{amn}) and we denote it
by $\Metr_c(X|G)$.\par
Whenever $X$ is a locally compact Polish space and $G$ is a group of homeomorphisms of $X$, $G$ acts
naturally on $X$ by $\varphi.x = \varphi(x)$. What is more, the function $\CCc(X,X) \times \CCc(X,X)
\ni (f,g) \mapsto g \circ f \in \CCc(X,X)$ is continuous (\cite[Theorem~3.4.2]{eng}) and thus
$(G,\circ)$ is a topological group iff the inverse is continuous on $G$. But this is always true for
locally compact groups with continuous multiplication, by a theorem of Ellis \cite{ell} (for more
general results in this fashion consult \cite{zel}, \cite{bra}, \cite{pfi}, \cite{bou}). However,
in the context of isometry groups of proper metric spaces the continuity of the inverse is
an elementary excercise.\par
Let $(X,d)$ be a metric space. The closed $d$-ball in $X$ with center at $a \in X$ and of radius
$r > 0$ is denoted by $\bar{B}_d(a,r)$; and $d \oplus d$ stands for the `sum' metric on $X \times
X$, that is, $(d \oplus d)((x,y),(x',y')) = d(x,x') + d(y,y')$. For a function $f\dd X \to \RRR$
we put
$$\Lip_d(f) = \sup \left\{\frac{|f(x) - f(y)|}{d(x,y)}\dd\ x,y \in X,\ x \neq y\right\} \in
[0,\infty]$$
provided $\card(X) > 1$ and $\Lip_d(f) = 0$ otherwise. The function $f$ is \textit{nonexpansive} iff
$\Lip_d(f) \leqsl 1$. For every $b \in X$ let $e_b\dd X \to \RRR$ be the so-called
\textit{Kuratowski} map corresponding to $b$; that is, $e_b(x) = d(b,x)$. For a nonempty set
$A \subset X$ we denote by $\dist_d(x,A)$ the $d$-distance of a point $x$ from $A$, i.e.
$$\dist_d(x,A) = \inf_{a \in A} d(x,a).$$
It is well-known (and easy to prove) that both $e_b$ and $\dist_d(\cdot,A)$ are nonexpansive maps.
Also the maximum and the minimum of finitely many nonexpansive (real-valued) maps is nonexpansive.
These facts will be applied later.\par
For any set $X$, $\Delta_X$ is the diagonal of $X$; that is, $\Delta_X = \{(x,x)\dd\ x \in X\}$.
A subset $K$ of $X \times X$ is said to be \textit{symmetric} if $(y,x) \in K$ for every $(x,y)
\in K$.\par
Beside all aspects discussed above, all notions and notations which appeared earlier are obligatory.

\SECT{Characterization of isometry groups\\with respect to proper metrics}

Our first aim is to show \THM{main}. Its proof will be preceded by several auxiliary results.\par
From now on, $X$ is a fixed locally compact Polish (nonempty) space and $G$ is a group
of homeomorphisms of $X$ such that
\begin{equation}\label{eqn:1,2}
(\card(G),\card(X)) \neq (1,2).
\end{equation}
We equip $G$ with the topology inherited from $\CCc(X,X)$. Further, we put
\begin{multline}\label{eqn:RG}
\RRr_G = \{(x,y;f(x),f(y))\dd\ f \in G,\ x,y \in X\} \cup\\\cup \{(x,y;f(y),f(x))\dd\ f \in G,\ x, y
\in X\}
\end{multline}
and for $x,y \in X$,
$$G^s(x,y) = \{(f(x),f(y))\dd\ f \in G\} \cup \{(f(y),f(x))\dd\ f \in G\}.$$
We begin with the following already known result whose proof we omit (see the notes
in the introductory section; use the main result of \cite{amn} to conclude the nonemptiness
of the set $\Metr_c(X|G)$).

\begin{pro}{1-2}
If $G$ satisfies conditions \textup{(Iso1)--(Iso2)} of \textup{\THM{main}}, then $G$ is a locally
compact topological group, the natural action of $G$ on $X$ is proper and the set $\Metr_c(X|G)$ is
nonempty.
\end{pro}

\begin{lem}{equiv}
If $G$ satisfies conditions \textup{(Iso1)--(Iso2)} of \textup{\THM{main}}, then $\RRr_G$ is
a closed equivalence relation on $X \times X$ and for each $(x,y) \in X \times X$ the equivalence
class of $(x,y)$ with respect to $\RRr_G$ coincides with $G^s(x,y)$.
\end{lem}
\begin{proof}
We leave this as an exercise that $\RRr_G$ is an equivalence relation and that $G^s(x,y)$ is
the equivalence class of $(x,y)$. Here we shall focus only on the closedness of $\RRr_G$
in $(X \times X)^2$. By \PRO{1-2}, the action of $G$ on $X$ is proper and thus the function
$X \times G \ni (x,f) \mapsto (x,f(x)) \in X \times X$ is a closed map (cf.
\cite[Theorem~3.7.18]{eng}). Consequently, the set $W = \{(x,f(x))\dd\ x \in X,\ f \in G\}$ is
closed in $X \times X$. So, the notice that
$$\RRr_G = \{(x,y;z,w) \in (X \times X)^2\dd\ (x,z;y,w) \in W \times W\ \vee\ (x,w;y,z)
\in W \times W\}$$
completes the proof.
\end{proof}

The next lemma is the only result in the proof of which the strange condition \eqref{eqn:1,2} is
used. This lemma will find an application later.

\begin{lem}{hull}
Let $u\dd X \to X$ be a function such that for any two distinct points $x$ and $y$ in $X$ there
exists $f \in G$ such that $\{u(x),u(y)\} = \{f(x),f(y)\}$. Then $u \in \hULL{G}$.
\end{lem}
\begin{proof}
According to \DEF{hull}, we only need to check that $u(x) \in G.x$ for every $x \in X$. If $G$ acts
transitively on $X$, the assertion immediately follows since then $G(x) = X$ for any $x \in X$. Thus
we assume that $G$ acts non-transitively. This means that $G(x) \neq X$ for any $x \in X$.\par
First assume that $\card(G) > 1$. Then there is $g \in G$ and $c \in X$ with $g(c) \neq c$. So,
by assumption, there is $f \in G$ with $\{u(c), u(g(c))\} = \{f(c),f(g(c))\} \subset G(c)$ and hence
$u(c) \in G(c)$. Now let $x \in X \setminus G(c)$. Then necessarily $x \neq c$ and thus there is
$f \in G$ with $\{u(x),u(c)\} = \{f(x),f(c)\}$. Since $c \notin G(x)$ and $u(c) \in G(c)$,
it follows that $u(c) = f(c)$ and therefore $u(x) = f(x) \in G(x)$. Finally, if $x \in G(c)$, take
$a \in X \setminus G(c)$ and a function $f \in G$ such that $\{u(x),u(a)\} = \{f(x),f(a)\}$. Now,
as before, since $u(a) \in G(a)$ and $a \notin G(x)$, we obtain that $u(a) = f(a)$ and consequently
$u(x) = f(x) \in G(x)$ as well.\par
Now assume that $G = \{\id_X\}$. By \eqref{eqn:1,2}, $\card(X) > 2$. Let $x \in X$ be arbitrary.
Take two distinct points $y, z \in X \setminus \{x\}$. By assumption, $\{u(x),u(y)\} = \{x,y\}$ and
$\{u(x),u(z)\} = \{x,z\}$ which yields that $u(x) = x$ and we are done.
\end{proof}

\begin{lem}{iso-hull}
If $G$ satisfies conditions \textup{(Iso1)--(Iso2)} of \textup{\THM{main}}, then $\hULL{G}$ is
a group of homeomorphisms, closed in $\CCc(X,X)$, and every $G$-invariant metric on $X$ is
$\hULL{G}$-invariant as well.
\end{lem}
\begin{proof}
The last part of the lemma immediately follows from the definition of $\hULL{G}$. Further,
it follows from \PRO{1-2} that $G$ acts properly on $X$, and thus there is $d \in \Metr_c(X|G)$.
Then---by the first argument of the proof---each member of $\hULL{G}$ is isometric with respect
to $d$. This yields that $\hULL{G} \subset \CCc(X,X)$. Moreover, it follows from \LEM{equiv} that
$G^s(x,y)$ is closed in $X \times X$ for any $x, y \in X$. We infer from this that $\hULL{G}$ is
closed in $\CCc(X,X)$ (since $\hULL{G}$ consists of all functions $u\dd X \to X$ such that
$(u(x),u(y)) \in G^s(x,y)$ for all $x, y \in X$). Hence it suffices to show that each member $u$
of $\hULL{G}$ is a bijection and $u^{-1} \in \hULL{G}$. To see this, fix $a \in X$ and take $f \in
G$ such that $u(a) = f(a)$. Then the map $v := f^{-1} \circ u\dd (X,d) \to (X,d)$ is isometric and
$v(a) = a$. This implies that $v$ sends each closed $d$-ball around $a$, which is compact, into
itself. Since every isometric map of a compact metric space into itself is onto (see e.g.
\cite{lin}), $v$, and consequently $u$, is a bijection. Finally, for arbitrary points $x$ and $y$
of $X$ take $f \in G$ such that $\{u(u^{-1}(x)),u(u^{-1}(y))\} = \{f(u^{-1}(x)),f(u^{-1}(y))\}$.
Then $f^{-1} \in G$ and $\{u^{-1}(x),u^{-1}(y)\} = \{f^{-1}(x),f^{-1}(y)\}$. This yields that
$u^{-1} \in \hULL{G}$ and we are done.
\end{proof}

The statement of the next lemma is complicated. However, this result is our key tool and it will be
applied in its full form in a part of the proof of \THM{main}.

\begin{lem}{Lip}
Let $(Y,\varrho)$ be a metric space, $a$ and $b$ be two distinct points of $Y$ and let $K$ be
a closed symmetric nonempty set in $Y \times Y$ such that $(a,b) \notin K$. Further, let $\epsi > 0$
and let $D$ be a dense subset of $[0,\infty)$. Then there are $\delta > 0$, $\alpha \in D$ and a map
$u\dd Y \to \RRR$ such that:
\begin{enumerate}[\upshape(L1)]
\item $\Lip_{\varrho}(u) \leqsl 1 + \epsi$ and $|u(x) - u(y)| \leqsl \alpha$ for all $x, y \in Y$,
\item $|u(x) - u(y)| = \alpha > \varrho(x,y)$ for each $(x,y) \in \bar{B}_{\varrho}(a,\delta) \times
   \bar{B}_{\varrho}(b,\delta)$,
\item $\sup_{(x,y) \in K} |u(x) - u(y)| < \alpha$.
\end{enumerate}
\end{lem}
\begin{proof}
Decreasing $\epsi$, if needed, we may and do assume that
\begin{equation}\label{eqn:aux00}
\epsi < \frac14 \min(1,\varrho(a,b)) \quad \textup{and} \quad [\bar{B}_{\varrho}(a,2\epsi) \times
\bar{B}_{\varrho}(b,2\epsi)] \cap K = \varempty
\end{equation}
(here we use the closedness of $K$). Everywhere below in this proof $\delta$ is a positive number
less than $\epsi$. Let $A_{\delta} = \bar{B}_{\varrho}(a,\delta)$, $B_{\delta} =
\bar{B}_{\varrho}(b,\delta)$, $c_{\delta} = \varrho(a,b) - 2 \delta > 0$ and let $u_{\delta}\dd
Y \to \RRR$ be given by the formula
$$u_{\delta}(y) = \min(\dist_{\varrho}(y,A_{\delta}),c_{\delta}) - \delta
\min(\dist_{\varrho}(y,A_{\delta}),\dist_{\varrho}(y,B_{\delta}),c_{\delta}).$$
Note that:
\begin{gather}
\Lip_{\varrho}(u_{\delta}) \leqsl 1 + \delta,\label{eqn:aux01}\\
\varrho(x,y) \in [c_{\delta},c_{\delta} + 4\delta] \quad \textup{for} \quad (x,y) \in A_{\delta}
\times B_{\delta},\label{eqn:aux02}\\
\lim_{\delta\to0} c_{\delta} = \varrho(a,b).\label{eqn:aux03}
\end{gather}
It follows from \eqref{eqn:aux00}, \eqref{eqn:aux02} and the fact that $\delta < \epsi$ that
\begin{equation}\label{eqn:aux04}
u_{\delta}(Y) \subset [0,c_{\delta}], \quad u_{\delta}^{-1}(\{0\}) = A_{\delta} \quad \textup{and}
\quad u_{\delta}^{-1}(\{c_{\delta}\}) = B_{\delta}.
\end{equation}
Now let $(x,y) \in K$. We infer from \eqref{eqn:aux00} and the symmetry of $K$ that
$$\min(\varrho(z,a),\varrho(z,b)) \geqsl 2\epsi \quad \textup{for} \quad z \in \{x,y\}.$$
So, for such $z$ we have $\dist_{\varrho}(z,A_{\delta}) \geqsl \epsi$ and
$\dist_{\varrho}(z,B_{\delta}) \geqsl \epsi$. This combined with the inequality $\epsi < c_{\delta}$
(cf. \eqref{eqn:aux00}) gives $u_{\delta}(z) \in [0,c_{\delta} - \delta \epsi]$ which shows that
\begin{equation}\label{eqn:aux05}
\sup_{(x,y) \in K} |u_{\delta}(x) - u_{\delta}(y)| \leqsl c_{\delta} -\delta \epsi < c_{\delta}.
\end{equation}
The final function $u$ may be taken in the form $u = \lambda u_{\delta}$ for small enough $\delta$
and suitably chosen $\lambda > 1$ (so that $\lambda c_{\delta} \in D$). It follows from
\eqref{eqn:aux01}--\eqref{eqn:aux05} that it is possible to do this. The details are left
to the reader.
\end{proof}

As a consequence of \LEM{Lip} we obtain

\begin{lem}{sep}
Let $(Y,\varrho)$ be a metric space, $K$ and $L$ be two disjoint symmetric closed nonempty subsets
of $Y \times Y$ such that $L \cap \Delta_Y = \varempty$ and $\varrho\bigr|_L \equiv \const$. Then
for every $\epsi > 0$ there exists a metric $\varrho_{\epsi}$ on $Y$ such that $\varrho \leqsl
\varrho_{\epsi} \leqsl (1 + \epsi) \varrho$ and
$$\sup_{(x,y) \in K} \varrho_{\epsi}(x,y) \neq \sup_{(x,y) \in L} \varrho_{\epsi}(x,y).$$
\end{lem}
\begin{proof}
We may assume that $\sup_{(x,y) \in K} \varrho(x,y) = \sup_{(x,y) \in L} \varrho(x,y) =: c$ (because
otherwise we may put $\varrho_{\epsi} = \varrho$). Take $(a,b) \in L$, note that $\varrho(a,b) = c$
and apply \LEM{Lip} to obtain a map $u\dd Y \to \RRR$ with properties (L1)--(L3). Now it suffices
to define $\varrho_{\epsi}$ by $\varrho_{\epsi}(x,y) = \max(\varrho(x,y),|u(x) - u(y)|)$. Then
$$\sup_{(x,y) \in K} \varrho_{\epsi}(x,y) = \max(c,\sup_{(x,y) \in K} |u(x) - u(y)|) < |u(a) - u(b)|
\leqsl \sup_{(x,y) \in L} \varrho_{\epsi}(x,y).$$
\end{proof}

The next lemma is obvious and we omit its proof.

\begin{lem}{dG}
Let $d \in \Metr_c(X|G)$.
\begin{enumerate}[\upshape(a)]
\item For any $x, y \in X$, $d$ is constant on $G^s(x,y)$.
\item Let $\varrho$ be a metric on $X$ such that $d \leqsl \varrho \leqsl M d$ for some $M > 1$. Let
   $\varrho_G\dd X \times X \to [0,\infty)$ be given by
   \begin{equation}\label{eqn:dG}
   \varrho_G(x,y) = \sup_{f \in G} \varrho(f(x),f(y)).
   \end{equation}
   Then $\varrho_G \in \Metr_c(X|G)$, $d \leqsl \varrho_G \leqsl M d$ and $\varrho_G(a,b) =
   \sup_{(x,y) \in G^s(a,b)} \varrho(x,y)$ for any $a, b \in X$.
\end{enumerate}
\end{lem}

By the above result, whenever $d \in \Metr_c(X|G)$ and $K = G^s(x,y)$ for some $x, y \in X$, $d(K)$
consists of a single number. For simplicity, we shall write $d[K]$ to denote this number.
In the next two result we shall use the transformation $\varrho \mapsto \varrho_G$ defined
by the formula \eqref{eqn:dG}.

\begin{pro}{const}
Assume $G$ satisfies conditions \textup{(Iso1)--(Iso2)} of \textup{\THM{main}}. For any $d' \in
\Metr_c(X|G)$ and each $\epsi > 0$ there exists a metric $\varrho \in \Metr_c(X|G)$ such that $d'
\leqsl \varrho \leqsl (1 + \epsi) d'$ and whenever $U$ is an open subset of $(X \times X) \setminus
\Delta_X$ with $\varrho\bigr|_U \equiv \const$, then there are two distinct points $x$ and $y$
in $X$ for which $U \subset G^s(x,y)$.
\end{pro}
\begin{proof}
We may assume that $\card(X) > 2$. Let $\{(\xi_n,\eta_n)\}_{n=1}^{\infty}$ be a dense subset
of $(X \times X) \setminus \Delta_X$. We arrange all members of the collection
$\{G^s(\xi_n,\eta_n)\dd\ n \geqsl 1\}$ in a one-to-one sequence (finite or not) $(K_n)_{n=0}^N$
(where $N \in \{0,1,2,\ldots,\infty\}$). We conclude from \LEM{equiv} that the sets $K_0,K_1,\ldots$
are closed, symmetric, disjoint from $\Delta_X$ as well as pairwise disjoint. We shall now construct
sequences $(d_n)_{n=1}^N$ and $(s_n)_{n=1}^N$ such that
\begin{enumerate}[(1$_n$)]
\item $d_n \in \Metr_c(X|G)$, $d_{n-1} \leqsl d_n \leqsl (1 + s_{n-1}) d_{n-1}$ with $d_0 = d'$ and
   $s_0 = \frac{\epsi}{2}$,
\item $c_n := \min \{|d_n[K_j] - d_n[K_l]|\dd\ j,l \in \{0,\ldots,n\},\ j \neq l\} > 0$,
\item $0 < 8 \max(1,d_n[K_0],\ldots,d_n[K_{n+1}]) s_n \leqsl \min(s_{n-1},c_n)$ and $\prod_{j=0}^n
   (1 + s_j) < 1 + \epsi$.
\end{enumerate}
It follows from \LEM{sep} that there is a metric $\varrho_0$ such that $d_0 \leqsl \varrho_0 \leqsl
(1 + s_0) d_0$ and $\sup_{(x,y) \in K_0} \varrho_0(x,y) \neq \sup_{(x,y) \in K_1} \varrho_0(x,y)$.
Put $d_1 = (\varrho_0)_G$. It follows from \LEM{dG} that conditions (1$_1$)--(2$_1$) are fulfilled.
Now choose $s_1$ so that (3$_1$) is satisfied as well.\par
Suppose that we have defined $d_n$ and $s_n$ for some positive $n < N$. If $d_n[K_{n+1}] \notin
\{d_n[K_0],\ldots,d_n[K_n]\}$, we put $d_{n+1} = d_n$. Otherwise there is a unique $s \in
\{0,\ldots,n\}$ such that $d_n[K_{n+1}] = d_n[K_s]$. Another application of \LEM{sep} gives a metric
$\varrho_n$ on $X$ such that $\sup_{(x,y) \in K_{n+1}} \varrho_n(x,y) \neq \sup_{(x,y) \in K_s}
\varrho_n(x,y)$ and $d_n \leqsl \varrho_n \leqsl (1 + s_n) d_n$. We put $d_{n+1} = (\varrho_n)_G$.
As before, we see that (1$_{n+1}$) is satisfied and that $d_{n+1}[K_s] \neq d_{n+1}[K_{n+1}]$. Let
us check that (2$_{n+1}$) is satisfied too. Since $|d_{n+1} - d_n| \leqsl s_n d_n$, for
$j=0,\ldots,n+1$ we have, by (3$_n$), $|d_{n+1}[K_j] - d_n[K_j]| \leqsl \frac14 c_n$. Therefore, for
$j \in \{0,\ldots,n\} \setminus \{s\}$, we obtain
\begin{multline*}
|d_{n+1}[K_{n+1}] - d_{n+1}[K_j]| \geqsl |d_n[K_{n+1}] - d_n[K_j]| - \frac12 c_n =\\
= |d_n[K_s] - d_n[K_j]| - \frac12 c_n \geqsl \frac12 c_n > 0.
\end{multline*}
Similarly, when $j, l \in \{0,\ldots,n\}$ are different, then $|d_{n+1}[K_j] - d_{n+1}[K_l]| \geqsl
|d_n[K_j] - d_n[K_l]| - \frac12 c_n \geqsl \frac12 c_n > 0$. This shows (2$_{n+1}$). Now, as before,
choose $s_{n+1}$ so that (3$_{n+1}$) is fulfilled.\par
Having the sequences $(d_n)_{n=0}^N$ and $(s_n)_{n=0}^N$, use (1$_n$) and (3$_n$) to show that
\begin{enumerate}[(1$_n$)]\addtocounter{enumi}{3}
\item $d_0 \leqsl d_n \leqsl (1 + \epsi) d_0$ and $s_n \leqsl 2^{m-n} s_m$ for $m \in
   \{0,\ldots,n\}$
\end{enumerate}
for each $n$. We define the final metric $\varrho$ as the pointwise limit of $d_n$'s. Precisely,
when $N$ is finite, put $\varrho = d_N$ and note that, by (2$_N$), the numbers
$\varrho[K_0],\ldots,\varrho[K_N]$ are distinct. If $N = \infty$, let $\varrho(x,y) =
\lim_{n\to\infty} d_n(x,y)$ (the limit exists by (1$_n$) and (4$_n$)). Observe that in both
the cases $\varrho \in \Metr_c(X|G)$ and $d' \leqsl \varrho \leqsl (1 + \epsi) d'$. We claim that
also for $N = \infty$ the numbers $\varrho[K_0],\varrho[K_1],\ldots$ are distinct. To see this, take
two integers $p$ and $q$ such that $0 \leqsl q < p$. It then follows that
\begin{equation}\label{eqn:aux1}
|d_p[K_p] - d_p[K_q]| \geqsl c_p > 0
\end{equation}
(see (2$_p$)). For $j \in \{p,q\}$ we have
\begin{equation}\label{eqn:aux2}
0 \leqsl d_{p+1}[K_j] - d_p[K_j] \leqsl s_p d_p[K_j] \leqsl \frac18 c_p
\end{equation}
and for $n > p$:
\begin{equation}\label{eqn:aux3}
0 \leqsl d_{n+1}[K_j] - d_n[K_j] \leqsl s_n d_n[K_j] \leqsl \frac12 s_{n-1} \leqsl \frac{1}{2^{n-p}}
s_p \leqsl \frac{1}{2^{n-p}} \cdot \frac18 c_p.
\end{equation}
So, \eqref{eqn:aux2} and \eqref{eqn:aux3} give $0 \leqsl \varrho[K_j] - d_p[K_j] \leqsl \frac14 c_p$
($j \in \{p,q\}$). But this, combined with \eqref{eqn:aux1}, implies that
$|\varrho[K_p] - \varrho[K_q]| \geqsl \frac12 c_p > 0$ and we are done.\par
To complete the proof, assume that $U \subset (X \times X) \setminus \Delta_X$ is open and nonempty
and $\varrho\bigr|_U \equiv \const$. Since $\{(\xi_n,\eta_n)\}_{n=1}^{\infty}$ is dense
in $(X \times X) \setminus \Delta_X$, so is the set $\bigcup_{j=0}^N K_j$. Hence there is
a nonnegative integer $j \leqsl N$ such that $U \cap K_j \neq \varempty$. Then, by the assumption
on $U$, $\varrho\bigr|_U \equiv \varrho[K_j]$. Since $U \setminus K_j$ is open and $U \cap K_l =
\varempty$ for $l \neq j$ (because $\varrho[K_j] \neq \varrho[K_l]$), $U \setminus K_j$ is empty and
we are done.
\end{proof}

\begin{lem}{2p}
Assume $G$ satisfies conditions \textup{(Iso1)--(Iso2)} of \textup{\THM{main}}. Let $\varrho \in
\Metr_c(X|G)$ be such that whenever $U$ is an open subset of $(X \times X) \setminus \Delta_X$ with
$\varrho\bigr|_U \equiv \const$, then there are two distinct points $x$ and $y$ in $X$ for which
$U \subset G^s(x,y)$. Let $a$ and $b$ be two distinct points of $X$ and let $\Omega_r = \{(x,y) \in
X \times X\dd\ \dist_{\varrho \oplus \varrho}((x,y),G^s(a,b)) < r\}$ where $r > 0$. For each $\epsi
> 0$ there is a metric $\lambda \in \Metr_c(X|G)$ such that $\varrho \leqsl \lambda \leqsl
(1 + \epsi) \varrho$ and $(g(a),g(b)) \in \Omega_r$ for every $g \in \Iso(X,\lambda)$.
\end{lem}
\begin{proof}
We may assume that $\Omega_r \neq X \times X$. Fix $s \in (0,r)$ and let $F$ be the closure
of $\Omega_s$ in $X \times X$. Observe that:
\begin{equation}\label{eqn:aux7}
\Omega_s \subset F \subset \Omega_r, \quad (f \times f)(\Omega_s) = \Omega_s \textup{ for all }
f \in G \quad \textup{and} \quad \Omega_s \textup{ is symmetric}.
\end{equation}
Let $\LLl$ be the collection of all sets $G^s(x,y)$ whose interior is nonempty. By \LEM{equiv} and
the separability of $X$, the family $\LLl$ is countable (finite or not) and thus the set $D =
(0,\infty) \setminus \{\varrho[L]\dd\ L \in \LLl\}$ is dense in $[0,\infty)$. Finally, put $K =
(X \times X) \setminus \Omega_s$ and notice that $K$ is closed, symmetric (thanks
to \eqref{eqn:aux7}), nonempty and $(a,b) \notin K$. Let $\delta > 0$, $\alpha \in D$ and $u\dd X
\to \RRR$ be as in \LEM{Lip} (applied for $Y = X$). Define a metric $\lambda_0$ on $X$
by $\lambda_0(x,y) = \max(\varrho(x,y),|u(x) - u(y)|)$ and put $\lambda = (\lambda_0)_G$. It follows
from \LEM{dG} that $\lambda \in \Metr_c(X|G)$ and $\varrho \leqsl \lambda \leqsl (1 + \epsi)
\varrho$. We claim that $\lambda$ is the metric we searched for. We argue by a contradiction.
Suppose there is $g \in \Iso(X,\lambda)$ such that $(g(a),g(b)) \notin \Omega_r$. Then $(g(a),g(b))
\notin F$. By the continuity of $g$, there are open neighbourhoods
\begin{equation}\label{eqn:aux4}
U \subset \bar{B}_{\varrho}(a,\delta) \quad \textup{and} \quad V \subset \bar{B}_{\varrho}(b,\delta)
\end{equation}
of $a$ and $b$ (respectively) such that
\begin{equation}\label{eqn:aux5}
[g(U) \times g(V)] \cap F = \varempty.
\end{equation}
Fix arbitrary $(x,y) \in U \times V$. We infer from \eqref{eqn:aux5} that $(g(x),g(y)) \notin F$ and
hence $(g(x),g(y)) \in K$. Consequently (see the second property in \eqref{eqn:aux7}), for any
$f \in G$, $(f(g(x)),f(g(y))) \in K$. So, by (L3) (see \LEM{Lip}):
\begin{equation}\label{eqn:aux6}
\sup_{f \in G} |u(f(g(x))) - u(f(g(y)))| < \alpha.
\end{equation}
At the same time, by (L1)--(L2) and \eqref{eqn:aux4}, $|u(f(x)) - u(f(y))| \leqsl \alpha =
|u(x) - u(y)| > \varrho(x,y)$. We infer from this (and from the $G$-invariance of $\varrho$) that
$\lambda(x,y) = \alpha$. Since $g \in \Iso(X,\lambda)$, $\lambda(g(x),g(y)) = \alpha$. This combined
with \eqref{eqn:aux6} and the $G$-invariance of $\varrho$ yields that $\varrho(g(x),g(y)) = \alpha$.
Since $x$ and $y$ were arbitrary, $\varrho\bigr|_P \equiv \alpha$ where $P = g(U) \times g(V)$. $P$
is an open nonempty subset of $(X \times X) \setminus \Delta_X$, since $g$ is a homeomorphism and
$\alpha \in D$ (thus $\alpha \neq 0$). So, it follows from the property of $\varrho$ that there
exist distinct points $p$ and $q$ in $X$ such that $P \subset G^s(p,q)$. But then $G^s(p,q) \in
\LLl$ and consequently $\alpha = \varrho[G^s(p,q)] \notin D$ which is a contradiction and finishes
the proof.
\end{proof}

For $d \in \Metr_c(X|G)$ let $\Delta(d)$ be the set of all metrics $\varrho$ such that $d \leqsl
\varrho \leqsl M d$ for some constant $M > 1$; and let $\Delta_G(d) = \Delta(d) \cap \Metr_c(X|G)$.
For two metrics $\varrho,\varrho' \in \Delta(d)$ we define their distance
$\Lambda_d(\varrho,\varrho')$ as the least nonnegative constant $C$ such that $|\varrho - \varrho'|
\leqsl C d$. The following result is left to the reader as a simple exercise.

\begin{lem}{Delta}
Let $d \in \Metr_c(X|G)$ and $\varrho_1,\varrho_2,\ldots,\varrho \in \Delta(d)$.
\begin{enumerate}[\upshape(A)]
\item $\Lambda_d$ is a complete metric on $\Delta(d)$ and $\Delta_G(d)$ is a closed set
   in $(\Delta(d),\Lambda_d)$.
\item $\lim_{n\to\infty} \Lambda_d(\varrho_n,\varrho) = 0$ iff there is a sequence
   $(\epsi_n)_{n=1}^{\infty}$ of positive real numbers such that $(1 + \epsi_k)^{-1} \varrho \leqsl
   \varrho_k \leqsl (1 + \epsi_k) \varrho$ for each $k$ and $\lim_{n\to\infty} \epsi_n = 0$.
   If $\lim_{n\to\infty} \Lambda_d(\varrho_n,\varrho) = 0$, then the metrics
   $\varrho_1,\varrho_2,\ldots$ converge uniformly to $\varrho$ on each of the sets $\{(x,y) \in
   X \times X\dd\ d(x,y) \leqsl r\}\ (r > 0)$.
\item For any $C > 1$ and $\varrho \in \Delta_G(d)$ the set $\{\varrho' \in \Delta_G(d)\dd\ \varrho
   \leqsl \varrho' \leqsl t C \varrho \textup{ for some } t \in (0,1)\}$ is dense (in the topology
   of $(\Delta(d),\Lambda_d)$) in the set $\{\varrho' \in \Delta_G(d)\dd\ \varrho \leqsl \varrho'
   \leqsl C \varrho\}$.
\end{enumerate}
\end{lem}

From now on, $\Delta(d)$ and all its subsets are equipped with the topology induced by the metric
$\Lambda_d$.

\begin{pro}{Gdelta}
Suppose $G$ satisfies conditions \textup{(Iso1)--(Iso2)} of \textup{\THM{main}}. Let $d \in
\Metr_c(X|G)$ and $\Omega$ be an open set in $X \times X$. For any two distinct points $a$ and $b$
of $X$ the set $A_d(a,b;\Omega)$ consisting of all metrics $\varrho \in \Delta_G(d)$ such that
$(g(a),g(b)) \in \Omega$ for every $g \in \Iso(X,\varrho)$ is $\GGg_{\delta}$ in $\Delta_G(d)$.
\end{pro}
\begin{proof}
Fix $\omega \in X$. For $n \geqsl 1$ let $U_n$ consists of all metrics $\varrho \in \Delta_G(d)$
such that $(g(a),g(b)) \in \Omega$ whenever $g \in \Iso(X,\varrho)$ is such that
$\max(d(g(\omega),\omega),d(g^{-1}(\omega),\omega)) \leqsl n$. Observe that $A_d(a,b;\Omega) =
\bigcap_{n=1}^{\infty} U_n$ and thus it suffices to show that the set $\Delta_G(d) \setminus U_n$ is
closed for every $n$. Fix $m \geqsl 1$ and let $\varrho_1,\varrho_2,\ldots \in \Delta_G(d) \setminus
U_m$ be a sequence which converges to some $\varrho \in \Delta_G(d)$. Then there are a sequence
$(g_n)_{n=1}^{\infty} \subset \CCc(X,X)$ and a number $M > 0$ such that for every $n \geqsl 1$:
\begin{gather}
d \leqsl \varrho_n \leqsl M d, \qquad g_n \in \Iso(X,\varrho_n),\label{eqn:aux11}\\
\max(d(g_n(\omega),\omega),d(g_n^{-1}(\omega),\omega)) \leqsl m,\label{eqn:aux12}\\
(g_n(a),g_n(b)) \in (X \times X) \setminus \Omega.\label{eqn:aux13}
\end{gather}
We conclude from \eqref{eqn:aux11} that $\max(\Lip_d(g_n),\Lip_d(g_n^{-1})) \leqsl M$ for any $n$.
But then, thanks to \eqref{eqn:aux12}, $g_n(x), g_n^{-1}(x) \in \bar{B}_d(\omega,M d(x,\omega) + m)$
for all $x \in X$ and $n \geqsl 1$. Since the metric $d$ is proper, we see that the family $\FFf :=
\{g_n\dd\ n \geqsl 1\} \cup \{g_n^{-1}\dd\ n \geqsl 1\}$ is evenly continuous as well as pointwise
precompact. So, by the Ascoli-type theorem, the closure of $\FFf$ in $\CCc(X,X)$ is compact. Hence,
passing to a subsequence, we may assume that the sequences $g_1,g_2,\ldots$ and
$g_1^{-1},g_2^{-1},\ldots$ converge uniformly on compact sets to some maps $g$ and $h$,
respectively. But then $\max(d(g(\omega),\omega),d(h(\omega),\omega)) \leqsl m$
(by \eqref{eqn:aux12}), $(g(a),g(b)) \notin \Omega$ (cf. \eqref{eqn:aux13}) and $g \circ h = h \circ
g = \id_X$ which means that $g$ is bijective and $h = g^{-1}$. To end the proof, it suffices to show
that $g \in \Iso(X,\varrho)$ (because then $\varrho \notin U_m$). This simply follows from
\eqref{eqn:aux11}:
\begin{multline*}
|\varrho_n(g_n(x),g_n(y)) - \varrho(g(x),g(y))| \leqsl \Lambda_d(\varrho_n,\varrho) d(g_n(x),g_n(y))
+\\+ |\varrho(g_n(x),g_n(y)) - \varrho(g(x),g(y))| \to 0 \quad (n\to\infty).
\end{multline*}
But $\varrho_n(g_n(x),g_n(y)) = \varrho_n(x,y)$ and, similarly as above,
$$\lim_{n\to\infty} |\varrho_n(x,y) - \varrho(x,y)| = 0,$$
which gives $\varrho(g(x),g(y)) = \varrho(x,y)$.
\end{proof}

\THM{main} will easily be concluded from the following

\begin{thm}{hull}
Suppose $G$ satisfies conditions \textup{(Iso1)--(Iso2)} of \textup{\THM{main}}. Then for every
$d \in \Metr_c(X|G)$ and each $M > 1$ the set $\{\varrho \in \Delta_G(d)\dd\ \varrho \leqsl M d,\
\Iso(X,\varrho) = \hULL{G}\}$ is dense and $\GGg_{\delta}$ in $\Delta_G^M(d) := \{\varrho \in
\Delta_G(d)\dd\ \varrho \leqsl M d\}$.
\end{thm}
\begin{proof}
We may and do assume that $\card(X) > 1$. Below we continue the notation of \PRO{Gdelta}. Let
$\{(a_n,b_n)\}_{n=1}^{\infty}$ be a dense set in $(X \times X) \setminus \Delta_X$. It is clear that
$\Delta_G^M(d)$ is closed in $\Delta_G(d)$ and thus it is a complete metric space (with respect
to $\Lambda_d$; cf. \LEM{Delta}). For any $n$ and $m$ put
$$\Omega_{n,m} = \{(x,y) \in X \times X\dd\ \dist_{d \oplus d}((x,y),G^s(a_n,b_n)) < 2^{-m}\}.$$
It is clear that $\Omega_{n,m}$ is open in $X \times X$ and thus it follows from \PRO{Gdelta} that
$A_{n,m} := A_d(a_n,b_n;\Omega_{n,m}) \cap \Delta_G^M(d)$ is a $\GGg_{\delta}$ set
in $\Delta_G^M(d)$. Hence, thanks to Baire's theorem, it suffices to show that:
\begin{enumerate}[(C1)]
\item $A_{n,m}$ is dense in $\Delta_G^M(d)$ for all $n,m \geqsl 1$,
\item $(\Delta :=) \{\varrho \in \Delta_G(d)\dd\ \varrho \leqsl M d,\ \Iso(X,\varrho) = \hULL{G}\} =
   \bigcap_{n,m=1}^{\infty} A_{n,m}$.
\end{enumerate}
We begin with (C1). Fix positive integers $n$ and $m$. By \LEM{Delta}, the set
$$W = \{\varrho \in \Delta_G(d)\dd\ \varrho \leqsl tM d \textup{ for some } t \in (0,1)\}$$
is dense in $\Delta_G^M(d)$. So, take a metric $d' \in W$ and $t \in (0,1)$ such that $d' \leqsl
tM d$. Fix arbitrary $\epsi > 0$ such that $t(1+\epsi)^2 \leqsl 1$. Now apply \PRO{const} to obtain
a suitable metric $\varrho$ (see the statement of \PRO{const}). Since then $d \leqsl \varrho$,
we have $\Omega_r \subset \Omega_{n,m}$ where $r = 2^{-m}$ and
$$\Omega_r = \{(x,y) \in X \times X\dd\ \dist_{\varrho \oplus \varrho}((x,y),G^s(a_n,b_n)) < r\}.$$
Finally, apply \LEM{2p} to obtain a metric $\lambda \in \Metr_c(X|G)$ such that $\varrho \leqsl
\lambda \leqsl (1 + \epsi) \varrho$ and
\begin{equation}\label{eqn:aux20}
(g(a_n),g(b_n)) \in \Omega_r (\subset \Omega_{n,m}) \quad \textup{for every } g \in \Iso(X,\lambda).
\end{equation}
Then necessarily
$$d \leqsl d' \leqsl \varrho \leqsl \lambda \leqsl (1 + \epsi) \varrho
\leqsl (1 + \epsi)^2 d' \leqsl (1 + \epsi)^2 tM d
\leqsl Md$$
which yields that $\lambda \in \Delta_G^M(d)$ and consequently $\lambda \in A_{n,m}$
(by \eqref{eqn:aux20}). What is more, the above estimations imply that $|d' - \lambda| \leqsl
\epsi(2 + \epsi) d' \leqsl \epsi (2 + \epsi) M d$ and hence $\Lambda_d(d',\lambda) \leqsl
\epsi(2 + \epsi) M$ which may be arbitrarily small. This shows (C1).\par
Now we pass to (C2). Since
$$\hULL{G} = \{f\dd X \to X|\quad \forall x,y \in X\dd\ (f(x),f(y)) \in G^s(x,y)\},$$
we clearly have $\Delta \subset \bigcap_{n,m=1}^{\infty} A_{n,m}$. To prove the converse inclusion,
take $\varrho \in \Delta_G^M(d)$ which belongs to each of $A_{n,m}$'s, fix $g \in \Iso(X,\varrho)$
and notice that then, since $G^s(a_n,b_n) = \bigcap_{m=1}^{\infty} \Omega_{n,m}$ (by the closedness
of $G^s(a_n,b_n)$, see \LEM{equiv}), for all $n \geqsl 1$ we have $(g(a_n),g(b_n)) \in
G^s(a_n,b_n)$. Equivalently, $(a_n,b_n;g(a_n),g(b_n)) \in \RRr_G$ (cf. \eqref{eqn:RG}) for any
$n \geqsl 1$. Therefore (since $g$ is continuous, $\RRr_G$ is closed in $(X \times X)^2 $(see
\LEM{equiv}) and the set $\{(a_n,b_n)\dd\ n \geqsl 1\}$ is dense in $(X \times X) \setminus
\Delta_X$) $(x,y;g(x),g(y)) \in \RRr_G$ for any two distinct points $x$ and $y$ of $X$. This means
that we may apply \LEM{hull} which gives $g \in \hULL{G}$. Since $\hULL{G} \subset \Iso(X,\varrho)$
by \LEM{iso-hull}, the proof is complete.
\end{proof}

\begin{proof}[Proof of \THM{main}]
The most difficult part --- implication `(iii)$\implies$(ii)' --- immediately follows from \PRO{1-2}
and \THM{hull}, while implication `(ii)$\implies$(i)' is trivial. Let us briefly explain that (iii)
is implied by (i). If $G = \Iso(X,d)$ for a proper metric $d$ on $X$, then the natural action of $G$
on $X$ is proper, by a theorem of Gao and Kechris \cite{g-k}, which gives (Iso2) (see ($\star$)
in subsection `Notation and terminology'). Point (Iso3) follows from \LEM{iso-hull}. Finally, (Iso1)
is well-known and may be shown in the following way. If functions $g_1,g_2,\ldots \in \Iso(X,d)$
converge uniformly on compact sets to a function $g \in \CCc(X,X)$, then necessarily $g$ is
isometric with respect to $d$ and it suffices to check that $g$ is a surjection. Fixing a point
$\omega \in X$, we see that there is $M > 0$ such that $d(g_n(\omega),\omega) \leqsl M$ and
consequently $d(g_n^{-1}(\omega),\omega) \leqsl M$ for every $n$. Then $g_n^{-1}(x) \in
\bar{B}_d(\omega,d(x,\omega) + M)$ for each $x \in X$ and hence the sequence
$g_1^{-1},g_2^{-1},\ldots$ contains a subsequence which converges uniformly on compact sets to some
$h\dd X \to X$. Then $g \circ h = \id_X$ and we are done.
\end{proof}

\begin{cor}{d-w}
For any connected locally compact Polish space $X$,
\begin{equation}\label{eqn:00}
\{\Iso(X,d)\dd\ d \textup{ is a compatible metric}\}
= \{\Iso(X,\varrho)\dd\ \varrho \in \Metr_c(X)\}.
\end{equation}
\end{cor}
\begin{proof}
It follows from the connectedness of $X$ and the theorem of van Dantzig and van der Waerden
\cite{d-w} that for every compatible metric $d$ on $X$, $G := \Iso(X,d)$ is locally compact and acts
properly on $X$. One simply concludes that therefore conditions (Iso1)--(Iso2) are satisfied. Hence,
according to \LEM{iso-hull}, $\hULL{G}$ is a group of homeomorphisms and (consequently) $\hULL{G}
\subset \Iso(X,d) = G$. This shows that (Iso3) is fulfilled as well. So, there is $\varrho \in
\Metr_c(X)$ for which $\Iso(X,\varrho) = \Iso(X,d)$ (by \THM{main}). This proves the inclusion
`$\subset$' in \eqref{eqn:00}. Since the converse one is obvious, the proof is complete.
\end{proof}

For simplicity, let us call a group $G$ of homeomorphisms of a locally compact Polish space $X$
an \textit{iso-group of transformations} iff $G$ satisfies conditions (Iso1)--(Iso3), or,
equivalently (thanks to \THM{main}), if there is a metric $d \in \Metr_c(X)$ such that $G =
\Iso(X,d)$.

\begin{proof}[Proof of \THM{free}]
Denote by $\tilde{G} \subset \CCc(X,X)$ the group of all maps of the form $x \mapsto a.x\
(a \in G)$. It suffices to show that $\tilde{G}$ satisfies conditions (Iso1)--(Iso3), thanks
to \THM{main}. Since the action of $G$ is proper, (Iso2) is fulfilled (cf. ($\star$)) and the set
$G.\omega$ is closed. Indeed, the properness of the action of $G$ implies that the function $\Phi\dd
G \times X \ni (g,x) \mapsto (g.x,x) \in X \times X$ is a closed map and therefore the set $\Phi(G
\times \{\omega\})$ is closed in $X \times X$, but $\Phi(G \times \{\omega\}) = (G.\omega) \times
\{\omega\}$.\par
Let us briefly check that $\tilde{G}$ is closed in $\CCc(X,X)$. If $\lim_{n\to\infty} g_n.x = u(x)$
for any $x \in X$, then there is $a \in G$ such that $u(\omega) = a.\omega$. It follows from (F1)
and the properness of the action that the function $G \ni g \mapsto g.\omega \in G.\omega$ is
a homeomorphism and hence $\lim_{n\to\infty} g_n = a$, which yields that $u(x) = a.x$ for each
$x \in X$.\par
Now we pass to (Iso3). Suppose $f \in \hULL{\tilde{G}}$. Then
\begin{equation}\label{eqn:aux30}
f(x) \in G.x
\end{equation}
for any $x \in X$. By (F1) and \eqref{eqn:aux30}, for each $\alpha \in G$ there is a \textbf{unique}
$\beta_{\alpha} \in G$ such that
\begin{equation}\label{eqn:aux31}
f(\alpha.\omega) = (\beta_{\alpha} \alpha).\omega.
\end{equation}
Let $a = \beta_{e_G}$. We shall show that $f(x) = a.x$ for every $x \in X$. First assume that
$x \notin G.\omega$. Then $\{f(x),f(\omega)\} = \{g.x,g.\omega\}$ for some $g \in G$. This implies,
thanks to \eqref{eqn:aux30}--\eqref{eqn:aux31}, that $g = a$ and $f(x) = a.x$ (since $x \notin
G.\omega$). Now let $x = \alpha.\omega$ for some $\alpha \in G$. Taking into account
\eqref{eqn:aux31}, we have to show that $\beta_{\alpha} = a$. Since the action of $G$ is
non-transitive, the set $D = X \setminus G.\omega$ is nonempty. For each $y \in D$ we have
$\{f(y),f(x)\} = \{g.y,g.x\}$ for some $g \in G$. We conclude from this that $g = \beta_{\alpha}$
and $f(y) = \beta_{\alpha}.y$ (because $f(x) \notin D$ and $\beta_{\alpha}$ is unique). At the same
time, $f(y) = a.y$, by the previous part of the proof. Hence $(a^{-1} \beta_{\alpha}).y = y$ for
each $y \in D$. Now an application of (F2) yields that $a^{-1} \beta_{\alpha} = e_G$ and
consequently $\beta_{\alpha} = a$.
\end{proof}

\begin{proof}[Proof of \COR{gx}]
It suffices to observe that the function $G \times (G \times X) \ni (a,(g,x)) \mapsto (ag,x) \in
G \times X$ is a (continuous) free and non-transitive (since $\card(X) > 1$) proper action and
to apply \THM{free}.
\end{proof}

Other consequences of \THM{free} are stated below.

\begin{cor}{bd}
Let $G$ be a locally compact Polish group, $X$ be a locally compact Polish space and let $G \times X
\ni (g,x) \mapsto g.x \in X$ be a proper action of $G$ on $X$ such that for some point $\omega \in
X$, $G$ acts freely at $\omega$ and the $G$-orbit $G.\omega$ of $\omega$ is non-open in $X$. Then
the maps $x \mapsto g.x\ (g \in G)$ form an iso-group of transformations.
\end{cor}
\begin{proof}
It follows from the assumptions that $G.\omega \neq X$ and $\card(X) \geqsl \aleph_0$. Thus,
according to \THM{free}, it suffices to show that $G$ acts effectively on $X \setminus G.\omega$.
But if $g \in G$ is such that $g.x = x$ for every $x \in X \setminus G.\omega$, then $g.z = z$ for
some $z \in G.\omega$ (since $X \setminus G.\omega$ is non-closed) and hence $g = e_G$.
\end{proof}

\begin{cor}{effect}
Let $G$ be a locally compact Polish group, $X$ be a locally compact Polish space having more than
one point and let $G \times X \ni (g,x) \mapsto g.x \in X$ be a proper effective action of $G$
on $X$. Let $X \sqcup G$ denote the topological disjoint union of $X$ and $G$ and $\tilde{G} \subset
\CCc(X \sqcup G,X \sqcup G)$ be the set of all maps $\psi_a\dd X \sqcup G \to X \sqcup G$ with
$a \in G$ of the form: $\psi_a(x) = a.x$ for $x \in X$ and $\psi_a(g) = ag$ for $g \in G$. Then
$\tilde{G}$ is an iso-group of transformations.
\end{cor}

The above result is a special case of \THM{free} and its proof is left to the reader.

\SECT{Isometry groups of homogeneous proper metric spaces}

Throughout this section, $G$ is a fixed topological group. For basic information on topological
groups the reader is referred to any classical book in this topic, e.g. \cite{pon}
or \cite{hr1,hr2}.\par
By $\kappa_G\dd G \to G$ we denote the map $x \mapsto x^{-1}$. For $a \in G$ let $L_a\dd G \ni x
\mapsto ax \in G$ and let $\LLl(G) =\{L_a\dd\ a \in G\} \subset \CCc(G,G)$. It is clear that:
\begin{itemize}
\item $\LLl(G)$ is a group of homeomorphisms,
\item $\LLl(G)$ satisfies conditions (Iso1)--(Iso2) (cf. the proof of \THM{free}) provided $G$ is
   locally compact and Polish,
\item $\LLl(G)$ acts freely and transitively on $G$.
\end{itemize}
The main goal of this section is to determine $\hULL{\LLl(G)}$. As a consequence, we shall
characterize all topological groups $G$ satisfying $\hULL{\LLl(G)} = \LLl(G)$ (see \THM{L(G)}). For
this purpose we introduce the following

\begin{dfn}{Hxp}
Let $H$ be a topological Abelian group and $p \in H$ be such that $p^2 = e_H \neq p$. We define
a multiplication on $H \times \{-1,1\}$ as follows:
$$(x,j) \cdot (y,k) = \begin{cases}
                      (xy,jk)       & \textup{if } j = 1\\
                      (xy^{-1},jk)  & \textup{if } j = -1,\ k = 1\\
                      (xy^{-1}p,jk) & \textup{if } j = k = -1
                      \end{cases} \qquad (x,y \in H,\ j,k \in \{-1,1\}).$$
Straightforward calculations show that $(H \times \{-1,1\},\cdot)$ is a topological group when it is
equipped with the product topology. We shall denote it by $\dOUBLE{H}$. Observe that $(e_H,1)$
is the neutral element of $\dOUBLE{H}$ and for any $x \in H$, $(x,-1)^{-1} = (xp,-1)$.\par
To avoid repetitions, every pair $(H,p)$ where $H$ is a topological Abelian group and $p \in H$ is
such that $p^2 = e_H \neq p$ will be called a \textit{base pair}. Two base pairs $(H,p)$ and $(K,q)$
are \textit{isomorphic} if there exists an isomorphism of $H$ onto $K$ which sends $p$ to $q$.
\end{dfn}

The following is left to the reader as an exercise.

\begin{lem}{Hxp}
Let $(H,p)$ be a base pair and let $K = \dOUBLE{H}$.
\begin{enumerate}[\upshape(a)]
\item The set $\tilde{H} := H \times \{1\}$ is an open normal subgroup of $K$ of index $2$ and
   the function $H \ni x \mapsto (x,1) \in \tilde{H}$ is an isomorphism of topological groups.
\item For each $a \in K \setminus \tilde{H}$, $a^2 = \tilde{p} := (p,1)$ and $\tilde{p}^2 = e_K$.
\item For any $x, y \in H$, $(x,-1) \cdot (y,1) \cdot (x,-1)^{-1} = (y^{-1},1)$, $(x,1) \cdot (y,-1)
   \cdot (x,1)^{-1} = (x^2y,-1)$ and $(x,-1) \cdot (y,-1) \cdot (x,-1)^{-1} = (x^2y^{-1},-1)$.
   In particular, if $H$ is non-Boolean, the center $Z(K)$ of $K$ coincides with $\{(x,1)\dd\
   x \in H,\ x^2 = e_H\} = \{z \in K\dd\ z^2 = e_K\}$.
\item $K$ is non-Boolean. $K$ is Abelian iff $H$ is Boolean.
\end{enumerate}
\end{lem}

For further studies, let us introduce one more useful notation: let $\hULL{\LLl(G),e_G}$ stand for
the set of all $f \in \hULL{\LLl(G)}$ which map $e_G$ onto itself.\par
In the following result the group $K$ is equipped with no topology.

\begin{lem}{basic}
Let $K = \hULL{\LLl(G)}$ and $K_0 = \hULL{\LLl(G),e_G}$.
\begin{enumerate}[\upshape(a)]
\item $(K,\circ)$ is a group.
\item $K_0$ is a Boolean subgroup of $K$.
\item For every $f \in K$ there is a unique pair $(a,f_0) \in G \times K_0$ such that $f = L_a \circ
   f_0$.
\item Let $f \in K_0$. Then:
   \begin{equation}\label{eqn:x-x-1}
   \{f(x),f(x^{-1})\} = \{x,x^{-1}\} \qquad (x \in G)
   \end{equation}
   and for every (possibly non-closed) subgroup $D$ of $G$, $f(D) \subset D$ and $f\bigr|_D \in
   \hULL{\LLl(D),e_D}$.
\end{enumerate}
\end{lem}
\begin{proof}
Point (c) as well as the fact that $(K,\circ)$ is a semigroup are left as simple exercises.
Consequently, $K_0$ is a semigroup as well.\par
We start with (d). For any $x \in G$ there is $a \in G$ such that $\{f(x),f(e_G)\} = \{ax,a\}$.
We conclude from this that $f(x) \in \{x,x^{-1}\}$ and consequently $f(x^{-1}) \in \{x,x^{-1}\}$
as well. Further, there is $b \in G$ such that $\{f(x),f(x^{-1})\} = \{bx,bx^{-1}\}$ which yields
that either $f(x) \neq f(x^{-1})$ or $x = x^{-1}$. Both these cases give \eqref{eqn:x-x-1} and hence
$f(f(x)) = x$. This shows (b). Now (b) and (c) imply (a). Finally, if $D$ is a subgroup of $G$,
\eqref{eqn:x-x-1} shows that $f(D) \subset D$. So, if $x$ and $y$ belong to $D$ and $a \in G$ is
such that $\{f(x),f(y)\} = \{ax,ay\}$, then $a \in Dx^{-1} = D$ and therefore $f\bigr|_D \in
\hULL{\LLl(D),e_D}$.
\end{proof}

\begin{lem}{Abel}
For every topological Abelian group $H$,
$$\hULL{\LLl(H)} = \LLl(H) \cup \{L_a \circ \kappa_H\dd\ a \in H\}.$$
\end{lem}
\begin{proof}
Thanks to \LEM{basic}, it suffices to show that $\hULL{\LLl(H),e_H} = \{\id_H,\kappa_H\}$. Observe
that $\{x^{-1},y^{-1}\} = \{(x^{-1}y^{-1}) \cdot x,(x^{-1}y^{-1}) \cdot y\}$ for any $x, y \in H$
and thus $\kappa_H \in \hULL{\LLl(H)}$. So, it remains to check that if $f \in \hULL{\LLl(H),e_H}$
is different from the identity map, then $f = \kappa_H$. By assumption, there is $b \in H$ such that
$f(b) \neq b$. By \eqref{eqn:x-x-1}, it is enough to check that if $f(x) = x$, then $x = x^{-1}$,
or---equivalently---that $x^2 = e_H$.\par
Assume $f(x) = x$. We infer from \eqref{eqn:x-x-1} that $f(b) = b^{-1} \neq b$. Since $f \in
\hULL{\LLl(H)}$, there is $a \in H$ for which $\{ax,ab\} = \{f(x),f(b)\} (= \{x,b^{-1}\})$. Notice
that $ax \neq x$, because otherwise $a = e_H$ and $b^{-1} = ab = b$ which is false. Hence $ax =
b^{-1}$ and $ab = x$ from which we easily deduce that $x^2 = e_H$.
\end{proof}

\begin{proof}[Proof of \PRO{Abel}]
We know that $\LLl(H)$ satisfies (Iso1)--(Iso2) and hence, by \LEM{iso-hull}, so does
$\hULL{\LLl(H)}$. What is more, it is obvious that $\hULL{\hULL{\FFf}} = \hULL{\FFf}$ for
an arbitrary family $\FFf$ of functions. Now it suffices to apply \LEM{Abel} and \THM{main} to get
the assertion.
\end{proof}

The main result of this section is the following

\begin{thm}{L(G)}
Let $G$ be a topological group, $K = \LLl(G)$ and $L = \hULL{K}$. Then $L$ is a group
of homeomorphisms of $G$ and exactly one of the following three conditions holds:
\begin{enumerate}[\upshape(a)]
\item $G$ is either Boolean, or non-Abelian and isomorphic to \textbf{no} group of the form
   $\dOUBLE{H}$ where $(H,p)$ is a base pair. In that case $L = K$.
\item $G$ is either Abelian non-Boolean or isomorphic to a group of the form $\dOUBLE{H}$ where
   $(H,p)$ is a base pair and $\{x^2\dd\ x \in H\} \not\subset \{e_H,p\}$. In that case $K$ is
   a normal subgroup of $L$; $\hULL{K,e_G}$ is isomorphic to $\{-1,1\}$ and consists of two
   automorphisms of $G$; and for any $a \in G$ and $f \in \hULL{K,e_G}$, $f \circ L_a \circ f^{-1} =
   L_{f(a)}$.
\item $G$ is isomorphic to a group of the form $\dOUBLE[\tilde{p}]{(\dOUBLE{H})}$ where $(H,p)$ is
   a base pair, $H$ is Boolean and $\tilde{p} = (p,1)$. In that case $K$ is non-normal in $L$,
   $\hULL{K,e_G}$ is isomorphic to $\{-1,1\}^3$ and contains $\kappa_G$ (which is not
   an automorphism).
\end{enumerate}
\end{thm}

The above result implies, among other things, that there is a unique topology $\tau$ (namely,
the topology of pointwise convergence) on $L = \hULL{\LLl(G)}$ finer than the topology of pointwise
convergence such that $(L,\tau)$ is a topological group and the function $G \ni x \mapsto L_x \in L$
is an embedding. What is more, in this topology $\LLl(G)$ is open in $L$.\par
For simplicity, we shall call every group $G$ satisfying condition (a) of \THM{L(G)}
an \textit{iso-group}; and if $G$ satifies condition (c), it will be called
\textit{iso-singular}.\par
The proof of \THM{L(G)} is quite elementary. However, it is not so short. We shall precede it
by a few auxiliary results. Under the notation of \THM{L(G)}, it follows from point (c)
of \LEM{basic} that connection $\hULL{\LLl(G),e_G} = \{\id_G\}$ implies that $K = L$. This is true
if $G$ is Boolean; on the other hand, for non-Boolean Abelian groups we have that
$\hULL{\LLl(G),e_G} = \{\id_G,\kappa_G\}$, by \LEM{Abel}. In the sequel we shall show, among other
things, that if $G$ is non-Abelian and $\hULL{\LLl(G),e_G} \neq \{\id_G\}$, then $G$ has a very
special form: it is isomorphic to $\dOUBLE{H}$ for some base pair $(H,p)$ with non-Boolean
$H$.

\begin{lem}{Hp}
For a topological group $G$ \tfcae
\begin{enumerate}[\upshape(i)]
\item there exists a base pair $(H,p)$ such that $G$ is isomorphic to $\dOUBLE{H}$,
\item there is an Abelian (possibly non-closed) subgroup $K$ of $G$ of index $2$ such that $x^2 = q$
   for any $x \in G \setminus K$ and some $q \in K$ such that $q^2 = e_G \neq q$.
\end{enumerate}
Moreover, if $G$, $K$ and $q$ are as in \textup{(ii)}, then $K$ is open and normal (in $G$) and $G$
is isomorphic to $\dOUBLE[q]{K}$.
\end{lem}
\begin{proof}
Implication `(i)$\implies$(ii)' follows from \LEM{Hxp}. Here we shall focus on the converse
implication. Under the assumptions of (ii), $K$ is normal, there is $b \in G \setminus K$ and
$G \setminus K = K b$. Let us check that the formulas $(x,1) \mapsto x$ and $(x,-1) \mapsto xb$
define an isomorphism $\Phi$ of $\dOUBLE[q]{K}$ onto $G$. It is clear that $\Phi$ is a bijection.
Observe that for each $z \in Kb$, $z^4 = q^2 = e_G$ and therefore $z^{-1} = z^3 = z q = q z$. This
means that for each $x \in K$, $xqb = (xb)q = (xb)^{-1} = b^{-1}x^{-1} = qbx^{-1}$ and consequently
$xb = bx^{-1}$. So, for any $x, y \in K$ we have $\Phi(x,1) \cdot \Phi(y,k) = \Phi((x,1)\cdot(y,k))$
for each $k \in \{-1,1\}$ and:
\begin{gather*}
\Phi(x,-1) \cdot \Phi(y,1) = x b y = x y^{-1} b = \Phi(xy^{-1},-1) = \Phi((x,-1)\cdot(y,1)),\\
\Phi(x,-1) \cdot \Phi(y,-1) = xbyb = xy^{-1} b^2 = xy^{-1}q = \Phi(xy^{-1}q,1)
= \Phi((x,-1)\cdot(y,-1))
\end{gather*}
which shows that $\Phi$ is a (possibly discontinuous) homomorphism. Notice that $\Phi$ is
a homeomorphism iff $K$ is open in $G$, iff $K$ is closed (being of index $2$ in $G$). So,
to complete the proof, we only need to verify the closedness of $K$. If $G$ is non-Abelian, then
the closure of $K$ differs from $G$ (since $K$ is Abelian) and thus it has to coincide with $K$,
since $K$ is of index $2$. Finally, if $G$ is Abelian, then $K$ is Boolean and $K = \{x \in G\dd\
x^2 = e_G\}$, which yields the closedness of $K$.
\end{proof}

\begin{lem}{odd}
Let $(H,p)$ be a base pair and $G = \dOUBLE{H}$. Let $\Phi_G\dd G \to G$ be given by $\Phi_G(x,1) =
(x,1)$ and $\Phi_G(x,-1) = (xp,-1) (= (x,-1)^{-1})\ (x \in H)$. Then $\Phi_G \in \hULL{\LLl(G),e_G}
\setminus \{\id_G\}$; $\Phi_G$ is both a homeomorphism and an automorphism of $G$ and
\begin{equation}\label{eqn:split}
\Phi_G \circ L_a \circ \Phi_G^{-1} = L_{\Phi_G(a)}
\end{equation}
for each $a \in G$.
\end{lem}
\begin{proof}
Since $p \neq e_H$, $\Phi_G \neq \id_G$. It is clear that $\Phi_G(e_G) = e_G$ and $\Phi_G$ is
a homeomorphism. For $x, y \in H$ we have $\{\Phi_G(x,1),\Phi_G(y,1)\} =
\{e_G\cdot(x,1),e_G\cdot(y,1)\}$, $\{\Phi_G(x,1),\Phi_G(y,-1)\} =
\{(xyp,-1)\cdot(x,1),(xyp,-1)\cdot(y,-1)\}$ and
$$\{\Phi_G(x,-1),\Phi_G(y,-1)\} = \{(p,1)\cdot(x,-1),(p,1)\cdot(y,-1)\}$$
which gives $\Phi_G \in \hULL{\LLl(G)}$. A verification that $\Phi_G$ is an automorphism and
satisfies condition \eqref{eqn:split} is left to the reader.
\end{proof}

\begin{lem}{-1}
Let $G$ be non-Abelian. \TFCAE
\begin{enumerate}[\upshape(i)]
\item $\kappa_G \in \hULL{\LLl(G)}$,
\item for any $x, y \in G$, $x^2 = y^2$ or $xy = yx$,
\item there is a nontrivial topological Boolean group $H$ and $p \in H \setminus \{e_H\}$ such that
   $G$ is isomorphic to $\dOUBLE[\tilde{p}]{(\dOUBLE{H})}$ where, as usual, $\tilde{p} = (p,1)$.
   In particular, $\card(\{x^2\dd\ x \in G\}) = 2$.
\end{enumerate}
\end{lem}
\begin{proof}
The equivalence of (i) and (ii) is straightforward. To see that (iii) is followed by (ii), first
note that $K = \dOUBLE{H}$ is Abelian (by \LEM{Hxp}) and that $\{y^2\dd\ y \in K\} =
\{\tilde{p},e_K\}$. Thus, $K$ is non-Boolean and therefore the center of $L :=
\dOUBLE[\tilde{p}]{K}$ coincides with $\{z \in L\dd\ z^2 = e_L\}$. Consequently, the assertion
of (ii) now easily follows since $\{z^2\dd\ z \in L\} = \{(\tilde{p},1),e_L\}$. The main point is
to show that (iii) is implied by (ii).\par
Assume $G$ satisfies (ii). Since $G$ is non-Abelian, there are points $a, b, c \in G$ such that
\begin{equation}\label{eqn:nonAbel}
ab \neq ba \qquad \textup{and} \qquad c^2 \neq e_G.
\end{equation}
Everywhere below, $x, y$ and $z$ denote arbitrary elements of $G$ and $Z$ denotes its center. Let
$p = c^2 (\neq e_G)$ and $K = Z \cup cZ$. We divide the proof into a few steps.\par
\begin{enumerate}[\bfseries S1.]
\item If $xy \neq yx$, $zx = xz$ and $yz = zy$, then $z^2 = e_G$.\par
   Proof: It follows from (ii) that $x^2 = y^2$. Similarly, since $(zx) y \neq z(yx) = y(zx)$,
   another application of (ii) gives $y^2 = (zx)^2 = z^2 x^2 = z^2 y^2$ and therefore $z^2 = e_G$.
\item $z \in Z(G) \iff z^2 = e_G$.\par
   Proof: Implication `$\implies$' follows from \eqref{eqn:nonAbel} and S1. To see the converse,
   assume $z^2 = e_G$. If $x^2 \neq e_G$, (ii) gives $zx = xz$. Finally, if $x^2 = e_G$, then
   $c^2 \neq x^2$ and $c^2 \neq z^2$ (cf. \eqref{eqn:nonAbel}) and hence, again thanks to (ii),
   $cx = xc$ and $cz = zc$. Since $c^2 \neq e_G$, S1 yields $zx = xz$.
\item $\{x^2\dd\ x \in G\} = \{p,e_G\}$, $p^2 = e_G$ and $K$ is an Abelian subgroup of $G$.\par
   Proof: It follows from S2 that there is $u \in G$ for which $c u \neq u c$. Then, by (ii), $u^2 =
   p$. Assume $x^2 \neq p$. Another applications of (ii) give: $xc = cx$ (since $x^2 \neq c^2$) and
   $xu = ux$ (since $x^2 \neq u^2$). Now S1 implies that $x^2 = e_G$. This proves the first
   assertion of S3, which is followed by the remainder (because $c^2 = p \in Z \subset K$).
\item $K = \{x \in G\dd\ xc = cx\}$.\par
   Proof: Since $K$ is Abelian, we only need to check that if $xc = cx$, then $x \in K$. Taking into
   account S2, we may assume that $x^2 \neq e_G$ and hence, by S3, $x^2 = p$. Then $(xc)^2 = x^2 c^2
   = p^2 = e_G$ and consequently $xc \in Z$ or, equivalently, $x \in Zc^{-1} = Zc^3 = Zp c =
   Zc \subset K$.
\item $xc \in \{cx,pcx\} \cup Z$.\par
   Proof: We may assume that $xc \neq cx$ (and thus $x^2 = p$, by (ii)) and $(xc)^2 \neq e_G$
   (by S2). Then $(xc)^2 = p$ (cf. S3) and consequently $xcxc = c^2$ from which we deduce that
   $xcx = c$, $xc = cx^3$ (by S3) and finally $xc = cpx = pcx$.
\item $x, y \notin K \implies xy \in K$.\par
   Proof: It follows from S4 that $xc \neq cx$ and $yc \neq cy$. Now since $xc, yc \notin K$,
   S5 gives $cx = pxc$ and $cy = pyc$. So, $c (xy) = pxcy = p^2xyc = (xy) c$ and therefore
   $xy \in K$ (by S4).
\end{enumerate}
We are now ready to complete the proof. It follows from S3 and S6 that $K$ is an Abelian subgroup
of $G$ which has index $2$. What is more, S3 and S2 show that $x^2 = p \in K$ for any $x \notin K$
(since then $x \notin Z$). So, \LEM{Hp} yields that $G$ is isomorphic to $\dOUBLE{K}$. But
it is easily seen (since $K = Z \cup cZ$ and $c^2 = p \in Z$) that $K$ is isomorphic to $\dOUBLE{Z}$
(again apply \LEM{Hp}). Observe that under the identification of $K$ with $\dOUBLE{Z}$ given
in the proof of \LEM{Hp}, the point $p \in K$ corresponds to $\tilde{p} = (p,1) \in \dOUBLE{Z}$ and
thus $G$ is isomorphic to $\dOUBLE[\tilde{p}]{(\dOUBLE{Z})}$.
\end{proof}

\begin{lem}{non-1}
Suppose $\kappa_G \notin \hULL{\LLl(G)}$ and let $f \in \hULL{\LLl(G),e_G}$ be different from
$\id_G$. Then:
\begin{enumerate}[\upshape(P1)]
\item the set $K := \{x \in G\dd\ f(x) = x\}$ is an Abelian subgroup of $G$ of index $2$,
\item there is $p \in K \setminus \{e_G\}$ such that $p^2 = e_G$ and for any $x \in K$ and $y \in
   G \setminus K$, $xyx = y$ and $f(y) = y^{-1}$,
\item $\{x^2\dd\ x \in K\} \not\subset \{p,e_G\}$,
\item $G$ is isomorphic to $\dOUBLE{K}$.
\end{enumerate}
\end{lem}
\begin{proof}
Let $U = \{x \in G\dd\ f(x) \neq x\}$, $V = \{x \in G\dd\ f(x) \neq x^{-1}\}$ and $W = \{x \in G\dd\
x^2 = e_G\}$. It follows from \eqref{eqn:x-x-1} that $U$, $V$ and $W$ are pairwise disjoint, $G =
U \cup V \cup W$ and $x \in X \iff x^{-1} \in X$ whenever $x \in G$ and $X \in \{U,V,W\}$. Further,
the assumptions imply that $U$ and $V$ are nonempty. Observe also that $K = G \setminus U$. For
further usage, fix $u \in U$ and $v \in V$ and put $p := u^2$. As previously, we divide the proof
into steps. Everywhere below, $x$, $y$ and $z$ denote arbitrary elements of $G$.
\begin{enumerate}[\bfseries S1.]
\item Suppose $xy = yx$. If $x \in U$, then $f(y) = y^{-1}$; if $x \in V$, then $f(y) = y$.\par
   Proof: Let $C$ be the group generated by $x$ and $y$. We infer from \LEM{basic} that $f\bigr|_C
   \in \hULL{\LLl(C),e_C}$. Since $C$ is Abelian, \LEM{Abel} yields that $f\bigr|_C = \id_C$
   or $f\bigr|_C = \kappa_C$. Now the assertion easily follows.
\item If $x \in U$ and $y \in K$, then $yxy = x$.\par
   Proof: Let $a \in G$ be such that $(\{x^{-1},y\} =) \{f(x),f(y)\} = \{ax,ay\}$. If $y = ay$, then
   $a = e_G$ and $x^{-1} = x$ which is false. Thus $y = ax$ and $x^{-1} = ay$ from which we may
   deduce that indeed $yxy = x$.
\item If $x, y \notin V$, then $xy = yx$ or $x^2 = y^2$.\par
   Proof: Notice that then $f(x) = x^{-1}$ and $f(y) = y^{-1}$ and use the fact that $f \in
   \hULL{\LLl(G)}$ (cf. points (i) and (ii) in \LEM{-1}).
\item If $x \in U$, then $x^4 = e_G$.\par
   Proof: Suppose, for the contrary, that $x^4 \neq e_G$. It follows from S1 that $f(x^2) = x^{-2}
   (\neq x^2)$ and hence $x^2 \in U$. Now S2 gives $vxv = x$ and $vx^2v = x^2$ as well. So, $vx^2v =
   xvxv$ and consequently $vx = xv$. Another usage of S1 yields that $f(v) = v^{-1}$ which is false.
\item If $x \in U$, then $xy, yx \in U$ or $x^2 = y^2$.\par
   Proof: If $xy \notin U$, then S2 implies that $(xy) x (xy) = x$ and thus $yx^2y = e_G$, $x^2 =
   y^{-2}$, $y^2x^2 = e_G$ and finally, by S4, $y^2 = x^2$. Similarly, if $yx \notin U$, then $y^2 =
   x^2$ which completes the proof of S5.
\item $x, y \in U$, $z \in K \implies (xy) z = z (xy)$.\par
   Proof: It follows from S2 that $zxz = x$ and $zyz = y$. But then $yzy^{-1} = z^{-1}$ and $z =
   xz^{-1}x^{-1}$. So, $z = x(yzy^{-1})x^{-1} = (xy) z (xy)^{-1}$ and we are done.
\item $x,y \in U \implies xy \in K$.\par
   Proof: It suffices to show that if $x, y, xy \in U$ for some $x$ and $y$, then $f = \kappa_G$.
   Fix $z$. It follows from \eqref{eqn:x-x-1} that $f(z) \in \{z,z^{-1}\}$. So, we only need
   to check that if $f(z) = z$, then $f(z) = z^{-1}$ as well. But this is immediate: if $f(z) = z$,
   then S6 gives $(xy)z = z(xy)$ and thus, since $xy \in U$, $f(z) = z^{-1}$, by S1.
\item $\{x^2\dd\ x \in U\} = \{p\}$ and $p^2 = e_G \neq p$.\par
   Proof: Let $x \in K$. By S7, $xu \notin U$ and hence, thanks to S5, $x^2 = u^2 (= p)$. Finally,
   $p^2 = u^4 = e_G$, by S4, and, of course, $p \neq e_G$ since $u \notin W$.
\item $W$ coincides with the center of $G$.\par
   Proof: It may be deduced from S1 that each element of $U$ commutes with no element of $V$ and
   therefore the center of $G$ is contained in $W$. Conversely, if $x^2 = e_G$ and $y$ is arbitrary,
   we have the following four possibilities:
   \begin{enumerate}[(1$^\circ$)]
   \item $y \in U$; then $y^2 \neq x^2$ and hence $yx = xy$, by S3,
   \item $y^2 \neq p$; then $y^2 \neq u^2$ and therefore $yu \in U$ (by S5) and it follows from
      (1$^{\circ}$) that $x$ commutes with both $u$ and $yu$ which easily implies that $xy = yx$,
   \item $y^2 = p$ and $(xy)^2 \neq p$; then, by (2$^{\circ}$), $xy$ commutes with $x$ and
      consequently $xy = yx$,
   \item $y^2 = (xy)^2 = p$; then $xyx = y$ and therefore $xy = yx$, since $x^2 = e_G$.
   \end{enumerate}
\item There is $q \in K$ such that $q^2 \notin \{p,e_G\}$.\par
   Proof: Since $\kappa_G \notin \hULL{\LLl(G)}$, \LEM{-1} implies that there are $x, y \in G$ such
   that $xy \neq yx$ and $x^2 \neq y^2$. Then $x, y \notin W$ (thanks to S9) and there is $q \in
   \{x,y\}$ such that $q^2 \neq p$. Then $q^2 \notin \{p,e_G\}$ and consequently $q \in K$ (by S8).
\item $K$ is a subgroup of $G$.\par
   Proof: Let $q$ be as in S10. It suffices to check that $K$ coincides with the centralizer of $q$,
   that is, $K = \{x \in G\dd\ xq = qx\}$. Since $q \in V$, the inclusion `$\supset$' follows from
   S1. To see the converse, first note that $qu \in U$, since $q^2 \neq u^2$ (cf. S5). As mentioned
   earlier, $u^{-1} \in U$ as well. So, if $x \in K$, an application of S6 gives $(qu \cdot u^{-1})
   x = x (qu \cdot u^{-1})$ and we are done.
\item $K$ is Abelian.\par
   Proof: The map $\varphi\dd K \ni x \mapsto uxu^{-1} \in G$ is a homomorphism. However, since
   $u \in U$, it follows from S2 that $xux = u$ for $x \in K$ and therefore $\varphi(x) = x^{-1}$.
   So, $\kappa_K$ is a homomorphism from which we infer the assertion of S12.
\end{enumerate}
Now we are ready to finish the proof. It follows from S11, S12 and S7 that $p \in K$ and point (P1)
is fulfilled. Further, S8, S2 and the definition of $K$ give point (P2), while (P3) is covered
by S10. So, an application of \LEM{Hp} yields (P4) and we are done.
\end{proof}

\begin{lem}{hull1}
Let $(H,p)$, $G$ and $\Phi_G$ be as in \textup{\LEM{odd}}. If $\{x^2\dd\ x \in H\} \not\subset
\{p,e_H\}$, then $\hULL{\LLl(G),e_G} = \{\Phi_G,\id_G\}$.
\end{lem}
\begin{proof}
It follows from the assumptions that $H$ is non-Boolean and thus $G$ is non-Abelian. What is more,
we conclude from point (iii) of \LEM{-1} that $\kappa_G \notin \hULL{\LLl(G)}$. Thus, if $f \in
\hULL{\LLl(G),e_G} \setminus \{\id_G\}$, \LEM{non-1} implies that $V = \{x \in G\dd\ f(x) = x\}$ is
an Abelian subgroup of $G$ of index $2$, $xyx = y$ for any $x \in V$ and $y \in G \setminus V$, and
$f$ coincides with $\kappa_G$ on $G \setminus V$. So, to convince that $f = \Phi_G$, we only need
to check that $V = H \times \{1\}$. Since both these groups have index $2$, it is enough to verify
that $V \subset H \times \{1\}$. But if $(a,-1) \in V$ for some $a \in H$, then $V$ (being Abelian)
is contained in the centralizer of $(a,-1)$ which coincides with $\{(ax,-1)\dd\ x \in H,\ x^2 =
e_H\} \cup \{(x,1)\dd\ x \in H,\ x^2 = e_H\}$ (cf. point (c) of \LEM{Hxp}). Hence, if $b \in H$ is
such that $b^2 \notin \{p,e_H\}$, then $(b,1) \notin V$ and therefore $(a,-1) \cdot (b,1) \cdot
(a,-1) = (b,1)$. But this is false since $(a,-1) \cdot (b,1) \cdot (a,-1) = (b^{-1}p,1)$ and $b^2
\neq p$. The proof is complete.
\end{proof}

Our last result necessary for giving a proof of \THM{L(G)} is the following

\begin{lem}{hull2}
Let $H$ be a nontrivial topological Boolean group, $p \in H \setminus \{e_H\}$ and let $G =
\dOUBLE[\tilde{p}]{(\dOUBLE{H})}$. Then $\kappa_G \in \hULL{\LLl(G),e_G}$, $\LLl(G)$ is non-normal
in $\hULL{\LLl(G)}$ and $\hULL{\LLl(G),e_G}$ consists of homeomorphisms and is isomorphic
to $\{-1,1\}^3$.
\end{lem}
\begin{proof}
We have already shown (in \LEM{-1}) that $\kappa_G \in \hULL{\LLl(G)}$. We conclude from this (and
the non-commutativity of $G$) that $\LLl(G)$ is non-normal (because $(\kappa_G \circ L_a \circ
\kappa_G^{-1})(x) = xa^{-1}$ for any $a, x \in G$ and thus $\kappa_G \circ L_a \circ \kappa_G^{-1}
\notin \LLl(G)$ whenever $a$ does not belong to the center of $G$).\par
We shall naturally identify $G$, as a set, with $H \times \{-1,1\} \times \{-1,1\}$. Put $Z =
H \times \{1\} \times \{1\}$, $a = (e_H,-1,1)$, $b = (e_H,1,-1)$ and $c = ab$. Then $Z$ is
the center of $G$ (since $\dOUBLE{H}$ in non-Boolean --- see \LEM{Hxp}), the sets $Z$, $aZ$, $bZ$
and $cZ$ are pairwise disjoint and their union is $G$. What is more, since $x^2 = \tilde{\tilde{p}}
(= (p,1,1))$ for any $x \notin Z$, the sets $H_a$, $H_b$ and $H_c$ are subgroups of $G$ where $H_x =
xZ \cup Z$ for $x \in \{a,b,c\}$. It is clear that each of these groups is open and Abelian. Thus,
if $f \in \hULL{\LLl(G),e_G}$, then $f\bigr|_{H_x}$ coincides with $\id_{H_x}$ or $\kappa_{H_x}$ for
$x = a,b,c$ (see Lemmas~\ref{lem:basic} and \ref{lem:Abel}). This implies that $f$ is
a homeomorphism and that $\card(\hULL{\LLl(G),e_G}) \leqsl 8$. Since the group $\hULL{\LLl(G),e_G}$
is Boolean (cf. \LEM{basic}), it remains to show that $\card(\hULL{\LLl(G),e_G}) = 8$. We leave it
as an exercise that whenever $f\dd G \to G$ is such that $f\bigr|_Z = \id_Z$ and $f\bigr|_{xZ} \in
\{\id_G\bigr|_{xZ},\kappa_G\bigr|_{xZ}\}$ for $x \in \{a,b,c\}$, then $f \in \hULL{\LLl(G),e_G}$,
which finishes the proof. (For example, use the fact that for each $x \in \{a,b,c\}$, $H_x$ is
of index $2$ and $y^2 = \tilde{\tilde{p}}$ for $y \notin H_x$ to show that $G$ is `naturally'
isomorphic to $\dOUBLE[\tilde{\tilde{p}}]{H_x}$, cf. \LEM{Hp}. Then apply \LEM{odd} to conclude that
$f_x \in \hULL{\LLl(G),e_G}$ where $f_x(y) = y$ for $y \in H_x$ and $f_x(y) = y^{-1}$ otherwise.
Finally, check that composing $f_a,f_b,f_c$ and $\kappa_G$ one may obtain any of functions mentioned
above.)
\end{proof}

\begin{proof}[Proof of \THM{L(G)}]
Let us briefly sum up all facts already established to conclude the whole assertion. The case when
$G$ is Abelian (Boolean or not) directly follows from \LEM{Abel}. Therefore, we may assume $G$ is
non-Abelian. We have three possibilities:
\begin{itemize}
\item $\kappa_G \in \hULL{\LLl(G),e_G}$; in that case use Lemmas~\ref{lem:-1} and \ref{lem:hull2}
   to get that $G$ is isomorphic to $\dOUBLE[\tilde{p}]{(\dOUBLE{H})}$ for some base pair $(H,p)$
   with Boolean $H$ and that $\hULL{\LLl(G),e_G}$ consists of $8$ homeomorphisms;
\item $\kappa_G \notin \hULL{\LLl(G),e_G}$ and $\hULL{\LLl(G),e_G} \neq \{\id_G\}$; in that case use
   Lemmas~\ref{lem:non-1}, \ref{lem:hull1} and \ref{lem:odd} to conclude that $G$ is isomorphic
   to $\dOUBLE{H}$ for some base pair $(H,p)$ with $\{x^2\dd\ x \in H\} \not\subset
   \{p,e_H\}$ and that $\hULL{\LLl(G),e_G}$ consists of two homeomorphic automorphisms of $G$;
\item $\hULL{\LLl(G),e_G} = \{\id_G\}$; in that case use \LEM{odd} to infer that $G$ is isomorphic
   to no group of the form $\dOUBLE{H}$ where $(H,p)$ is a base pair.
\end{itemize}
Notice that all assumptions (about the form of $\hULL{\LLl(G),e_G}$) in the above cases, as well as
all their conclusions, are mutually exclusive, hence points (a), (b) and (c) of the theorem are
satisfied and there is no other possibility. (In all above situations the fact that $\hULL{\LLl(G)}$
consists of homeomorphisms may simply be deduced from point (c) of \LEM{basic}.) Now the proof is
complete.
\end{proof}

\begin{cor}{isom}
Let $(H,p)$ and $(K,q)$ be two base pairs.
\begin{enumerate}[\upshape(A)]
\item The topological groups $\dOUBLE{H}$ and $\dOUBLE[q]{K}$ are isomorphic iff
   the base pairs $(H,p)$ and $(K,q)$ are isomorphic (cf. \textup{\DEF{Hxp}}).
\item If $H$ and $K$ are Boolean, then the topological groups $\dOUBLE[\tilde{p}]{(\dOUBLE{H})}$ and
   $\dOUBLE[\tilde{q}]{(\dOUBLE[q]{K})}$ are isomorphic iff the base pairs $(H,p)$ and $(K,q)$ are
   isomorphic.
\end{enumerate}
\end{cor}
\begin{proof}
The sufficiency of the latter condition in both points (A) and (B) is immediate. Let us show
the necessity in (A). For simplicity, put $G = \dOUBLE{H}$ and $L = \dOUBLE[q]{K}$. Below $\tilde{H}
= H \times \{1\} \subset G$ and $\tilde{K} = K \times \{1\} \subset L$. Let $\Psi\dd G \to L$ be
an isomorphism. Since $(H,p)$ and $(\tilde{H},\tilde{p})$ (respectively $(K,q)$ and
$(\tilde{K},\tilde{q})$) are isomorphic, it suffices to show that $(\tilde{H},\tilde{p})$ and
$(\tilde{K},\tilde{q})$ are isomorphic. We have three possibilities:
\begin{enumerate}[(1$^{\circ}$)]
\item $G$ is Abelian. Then so is $L$ and thus $H$ and $K$ are Boolean. Therefore $\Psi(\tilde{H}) =
   \tilde{K}$ since both $\tilde{H}$ and $\tilde{K}$ consist of all elements of order $2$.
   Consequently, $\Psi(\tilde{p}) = \tilde{q}$, since $x^2 = \tilde{p}$ for $x \in G \setminus
   \tilde{H}$ and $y^2 = \tilde{q}$ for $y \in L \setminus \tilde{K}$, and we are done.
\item $\card(\{x^2\dd\ x \in G\}) > 2$. Then $\card(\{y^2\dd\ y \in L\}) > 2$ as well and
   consequently $\{h^2\dd\ h \in H\} \not\subset \{p,e_H\}$ and $\{k^2\dd\ k \in K\} \not\subset
   \{q,e_K\}$. Then $f := \Psi \circ \Phi_G \circ \Psi^{-1} \in \hULL{\LLl(L),e_L} \setminus
   \{\id_L\}$ and therefore $f = \Phi_L$, by \LEM{hull1}. But then $\Psi(\tilde{H}) = \tilde{K}$,
   because $\tilde{H} = \{x \in G\dd\ \Phi_G(x) = x\}$ and similarly for $\tilde{K}$. Consequently,
   $\Psi(\tilde{p}) = \tilde{q}$ (cf. point (1$^{\circ}$)) and we are done.
\item $G$ is non-Abelian and $\{x^2\dd\ x \in G\} = \{\tilde{p},e_G\}$. Then $L$ is non-Abelian and
   $\{y^2\dd\ y \in L\} = \{\tilde{q},e_L\}$ as well and therefore $\Psi(\tilde{p}) = \tilde{q}$,
   $\{h^2\dd\ h \in H\} = \{p,e_H\}$ and $\{k^2\dd\ k \in K\} = \{q,e_K\}$ (notice that both $H$ and
   $K$ are non-Boolean). Put $Z = \{h \in H\dd\ h^2 = e_H\}$, $W = \{k \in K\dd\ k^2 = e_K\}$,
   $\tilde{Z} = Z \times \{1\} \subset \tilde{H}$ and $\tilde{W} = W \times \{1\} \subset
   \tilde{K}$. Then $\Psi(\tilde{Z}) = \tilde{W}$ since $\tilde{Z}$ and $\tilde{W}$ coincide with
   the centra of $G$ and $L$, respectively (cf. \LEM{basic}). So, $(Z,p)$ and $(W,q)$ are isomorphic
   and therefore $(\dOUBLE{Z},\tilde{p})$ and $(\dOUBLE[q]{W},\tilde{q})$ are isomorphic as well.
   Finally, \LEM{Hp} implies that $(H,p)$ is isomorphic to $(\dOUBLE{Z},\tilde{p})$ and $(K,q)$
   to $(\dOUBLE[q]{Z},\tilde{q})$.
\end{enumerate}
So, the proof of (A) is complete, while point (B) simply follows from (A).
\end{proof}

\begin{cor}{count}
Up to isomorphism, there are only countably many locally compact Polish iso-singular groups. Among
them only three are infinite: one compact, one discrete and one non-compact non-discrete. Each such
a group $G$ is of exponent $4$ (that is, $x^4 = e_G$ for any $x \in G$) and totally disconnected.
And if $G$ is finite, it is uniquely determined (up to isomorphism) by its cardinality.
\end{cor}
\begin{proof}
By a theorem of Braconnier \cite{brc} (see also \cite[(25.29)]{hr1}), every locally compact Boolean
group is isomorphic to the group of the form
\begin{equation}\label{eqn:Boole}
\{-1,1\}^{\alpha} \times \{-1,1\}^{\oplus \beta}
\end{equation}
where $\alpha$ and $\beta$ are arbitrary cardinals (when one of them is equal to $0$, omit suitable
factor), $\{-1,1\}^{\alpha}$ is the (full) Cartesian product of $\alpha$ copies of $\{-1,1\}$
equipped with the product topology and $\{-1,1\}^{\oplus \beta}$ is the \textit{weak direct product}
of $\beta$ copies of $\{-1,1\}$ (cf. \cite[(2.3)]{hr1}) equipped with the discrete topology; that
is, if $\card(Y) = \beta$, $\{-1,1\}^{\oplus \beta}$ may be represented as the subgroup
of $\{-1,1\}^Y$ (taken without inheriting the topology and) consisting of all functions $f\dd Y \to
\{-1,1\}$ for which the set $\{y \in Y\dd\ f(y) = -1\}$ is finite. When the group \eqref{eqn:Boole}
is Polish, both cardinals $\alpha$ and $\beta$ do not exceed $\aleph_0$. So, up to isomorphism,
there are only countable number of nontrivial locally compact Polish Boolean groups: $B(n) :=
\{-1,1\}^n\ (n=1,2,3,\ldots)$, $B(\omega) := \{-1,1\}^{\aleph_0}$, $B(\sigma) :=
\{-1,1\}^{\oplus \aleph_0}$ and $B(\infty) := B(\sigma) \times B(\omega)$. Each of these groups may
natutally be represented as the group of $\{-1,1\}$-valued sequences (finite or not): $B(n)$,
$B(\omega)$, $B(\sigma)$ and $B(\infty)$ consist of sequences numbered by, respectively, $J(n) =
\{1,\ldots,n\}$, $J(\omega) = J(\sigma) = \{1,2,\ldots\}$ and $J(\infty) = \ZZZ$ where
$$B(\infty) = \{(x_k)_{k\in\ZZZ} \subset \{-1,1\}|\quad \exists k_0 \in \ZZZ\quad
\forall k < k_0\dd\ x_k = 1\}$$
(note also that the topology of $B(\infty)$ is finer than the one of pointwise convergence).
It follows from the definition of an iso-singular group and \COR{isom} that it is enough to show
that whenever $k \in \{1,2,\ldots,\omega,\sigma,\infty\}$ and $p, q \in B(k)$ are different from
the neutral element, then there exists an automorphism $\tau\dd B(k) \to B(k)$ which sends $p$
to $q$. This may be provided as follows. There are $n, m \in J(k)$ such that $p_m = q_n = -1$. Put
$H_p = \{x \in B(k)\dd\ x_m = 1\}$ and similarly $H_q = \{x \in B(k)\dd\ x_n = -1\}$. It is clear
that $H_p$ and $H_q$ are isomorphic and open in $G$. So, if $\eta\dd H_p \to H_q$ is an isomorphism,
then the formulas $\tau(\xi) = \eta(\xi)$ for $\xi \in H_p$ and $\tau(\xi) = \eta(\xi p) q$
otherwise well defines the automorphism we searched for. Further details are left to the reader (see
also \REM{left-right} below).
\end{proof}

Another consequences of \THM{L(G)} are stated below.

\begin{cor}{iso}
Let $G$ be a locally compact Polish group.
\begin{itemize}
\item $G$ is an iso-group iff there is a left invariant metric $d \in \Metr_c(G)$ such that
   $\Iso(G,d) = \LLl(G)$.
\item If $G$ is iso-singular, there is a left invariant metric $d \in \Metr_c(G)$ such that the set
   $\Iso(G,d;e_G) := \{f \in \Iso(G,d)\dd\ f(e_G) = e_G\}$ contains precisely $8$ functions.
\item If $G$ is non-iso-singular and not an iso-group, there is a left invariat metric $d \in
   \Metr_c(G)$ such that $\card(\Iso(G,d;e_G)) = 2$.
\end{itemize}
\end{cor}
\begin{proof}
Just apply Theorems~\ref{thm:L(G)} and \ref{thm:hull} (notice that $G$ is an iso-group \iaoi{}
$\hULL{\LLl(G)} = \LLl(G)$).
\end{proof}

\begin{cor}{left-right}
For a locally compact Polish group $G$ \tfcae
\begin{enumerate}[\upshape(i)]
\item every left invariant pseudometric on $G$ (possibly having nothing in common with the topology
   of $G$) is right invariant,
\item every left invariant metric in $\Metr_c(G)$ is right invariant,
\item $G$ is either Abelian or iso-singular.
\end{enumerate}
In particular, up to isomorphism, there are only countably many (locally compact Polish) non-Abelian
groups $G$ which satisfy \textup{(i)}.
\end{cor}
\begin{proof}
It is well-known (and quite an easy exercise) that a left invariant pseudometric on a group $G$ is
right invariant as well iff $\kappa_G$ is isometric with respect to this pseudometric. Therefore
the conclusion simply follows from Theorems~\ref{thm:L(G)} and \ref{thm:hull} and \COR{count}.
\end{proof}

\begin{rem}{left-right}
As we announced in the introductory part, we are now able to give explicit descriptions of all
locally compact Polish non-Abelian groups on which every left invariant metric is automatically
right invariant. In what follows, we preserve the notation introduced in the proof
of \COR{count}.\par
Let $p(1) = -1 \in B(1)$, $p(2) = (-1,1) \in B(2)$ and for $n > 2$ let $p(n) = (-1,1,1,\ldots,1) \in
B(n)$; further, let $p(\omega) = p(\sigma) = (-1,1,1,1,\ldots) \in B(\sigma) \subset B(\omega)$
(the inclusion is purely set-theoretic, i.e. it does not imply that the topology is inherited) and
finally let $p(\infty) = (p_m)_{m\in\ZZZ}$ be such that $p_0 = -1$ and $p_m = 1$ otherwise. Now for
$k \in \{1,2,\ldots,\omega,\sigma,\infty\}$ let
$$IS(k) = \dOUBLE[\widetilde{p(k)}]{(\dOUBLE[p(k)]{B(K)})}$$
(where, as usual, $\widetilde{p(k)} = (p(k),1)$; cf. \DEF{Hxp}). \COR{left-right} combined with
the argument presented in the proof of \COR{count} shows that $IS(k)$'s are the only groups under
the question.
\end{rem}

\COR{left-right} may simply be generalized as follows.

\begin{pro}{left-right}
For a metrizable topological group $G$ \tfcae
\begin{enumerate}[\upshape(i)]
\item every left invariant pseudometric on $G$ (possibly having nothing in common with the topology
   of $G$) is right invariant,
\item every left invariant compatible metric on $G$ is right invariant,
\item $G$ is either Abelian or iso-singular.
\end{enumerate}
\end{pro}
\begin{proof}
It follows from \THM{L(G)} that (iii) is equivalent to the fact that $\kappa_G \in \hULL{\LLl(G)}$.
We have also already mentioned in the proof of \COR{left-right} that a left invariant pseudometric
is right invariant iff it is preserved by $\kappa_G$. We infer from these remarks that (i) follows
from (iii). Since (ii) is obviously implied by (i), to end the proof it remains to show that
if $\kappa_G \notin \hULL{\LLl(G)}$, then there is a compatible metric on $G$ which is not preserved
by $\kappa_G$. To this end, it suffices to apply \LEM{notin} stated below for $f = \kappa_G$.
\end{proof}

\begin{lem}{notin}
Let $G$ be a metrizable topological group and $d$ be a left invariant compatible metric on $G$. For
any $f \notin \hULL{\LLl(G)}$ and $\epsi > 0$ there is a left invariant metric $\varrho$ on $G$ such
that $d \leqsl \varrho \leqsl (1 + \epsi) d$ and $f \notin \Iso(G,\varrho)$.
\end{lem}
\begin{proof}
We may assume $f \in \Iso(X,d)$ (because otherwise it suffices to put $\varrho = d$). Since
$f \notin \hULL{\LLl(G)}$, there are two points $a$ and $b$ in $G$ such that $(f(a),f(b)) \notin
\{(ga,gb)\dd\ g \in G\} \cup \{(gb,ga)\dd\ g \in G\}$. Equivalently, $(f(a),f(b)) \notin K$ where
$K := \{(x,y) \in G \times G\dd\ x^{-1}y = a^{-1}b\ \vee\ y^{-1}x = b^{-1}a\}$. We conclude that
$a \neq b$. Observing that $K$ is a closed symmetric subset of $G \times G$, we may apply \LEM{Lip}
to obtain a function $u\dd G \to \RRR$ such that $\Lip_d(u) \leqsl 1+\epsi$, $|u(f(a)) - u(f(b))| >
d(a,b)$ and $|u(f(a)) - u(f(b))| > \sup_{(x,y) \in K} |u(x) - u(y)|$. It is now easily seen that
the function $\varrho\dd G \times G \to \RRR$ given by
$$\varrho(x,y) = \max\Bigl(d(x,y),\sup_{g \in G} |u(gx) - u(gy)|\Bigr) \qquad (x,y \in G)$$
is a well defined metric having all postulated properties (since $\varrho(f(a),f(b)) >
\varrho(a,b)$).
\end{proof}

\begin{rem}{base}
\PRO{left-right} favours groups of the form $G := \dOUBLE[\tilde{p}]{(\dOUBLE{H})}$ with Boolean $H$
among topological non-Abelian groups. However, the fact that the above $G$ is independent
(up to isomorphism) of the choice of $p \in H \setminus \{e_H\}$ (as it is in case of locally
compact Polish $G$; see the proof of \COR{count}) is no longer true in general, even in the class
of Polish groups. To convince of that, take a nontrivial connected Polish Boolean group $B$ and put
$H = B \times \{-1,1\}$. Further, let $p = (b,1)$ and $q = (b,-1)$ where $b \in B \setminus
\{e_B\}$. Then the topological groups $\dOUBLE[\tilde{p}]{(\dOUBLE{H})}$ and
$\dOUBLE[\tilde{q}]{(\dOUBLE[q]{H})}$ are non-isomorphic. Indeed, \COR{isom} asserts that these
groups are isomorphic iff so are the base pairs $(H,p)$ and $(H,q)$. But $p$ belongs
to the component of $H$ containing $e_H$, while $q$ does not and therefore these base pairs are
non-isomorphic.
\end{rem}

Before we pass to the proof of \THM{trans}, let us show the following

\begin{pro}{trans}
Let $G$ be a topological group, $(X,d)$ be a metric space and $G \times X \ni (g,x) \mapsto g.x \in
G \times X$ be a (possibly discontinuous) effective action of $G$ on $X$ such that
\begin{equation}\label{eqn:iso-trans}
\Iso(X,d) = \{M_g\dd\ g \in G\}
\end{equation}
where for $g \in G$, $M_g\dd X \ni x \mapsto g.x \in X$, and for some $b \in X$ one of the following
two conditions is fulfilled:
\begin{enumerate}[\upshape(T1)]
\item $G.b$ is dense in $X$ and $(X,d)$ is complete; or
\item $G.b = X$.
\end{enumerate}
Then $G$ is an iso-group.
\end{pro}
\begin{proof}
We argue by a contradiction. Suppose $G$ is not an iso-group. Let $K = \{g \in G\dd\ g.b = b\}$.
It is clear that $K$ is a subgroup of $G$. What is more, it follows from (T1) (since $M_g$ is
continuous), (T2) and the effectivity of the action that for each $a \in G$,
\begin{equation}\label{eqn:aux60}
(\forall x \in G\dd\ xax^{-1} \in K) \implies a = e_G.
\end{equation}
We claim that $K = \{e_G\}$. When $G$ is Abelian, this immediately follows from \eqref{eqn:aux60}.
On the other hand, if $G$ is non-Abelian, there is a base pair $(H,p)$ such that $H$ is non-Boolean
and $G$ is isomorphic to $\dOUBLE{H}$. Let us identify these two groups. Let $Z$ be the center
of $G$. Property \eqref{eqn:aux60} implies that $K \cap Z = \{e_G\}$. So, $K \subset \tilde{H} (= H
\times \{1\})$ because for $x \in G \setminus \tilde{H}$, $x^2 = \tilde{p} \in Z$. Since for $a \in
\tilde{H}$, $xax^{-1} = a$ for $x \in \tilde{H}$ and $xax^{-1} = a^{-1}$ otherwise (cf. \LEM{Hxp}),
we see that indeed $K = \{e_G\}$. Consequently, the action is free at $b$ and the function $\Phi\dd
G \ni g \mapsto g.b \in G.b$ is a bijection. For simplicity, let $X_0$ and $d_0$ stand for,
respectively, $G.b$ and the restriction of $d$ to $X_0 \times X_0$.\par
Since $G$ is not an iso-group, there is $f \in \hULL{\LLl(G)}$ which does not belong to $\LLl(G)$.
Put $u_0 := \Phi \circ f \circ \Phi^{-1}\dd X_0 \to X_0$ and observe that $u_0 \in
\hULL{\Iso(X_0,d_0)}$, by \eqref{eqn:iso-trans}. This means that $u_0$ is isometric (with respect
to $d_0$) and consequently $u_0 \in \Iso(X_0,d_0)$, being a bijection. Now when (T1) is fulfilled,
we see that there is $u \in \Iso(X,d)$ which extends $u_0$; otherwise (i.e. if (T2) holds) we put
$u = u_0 \in \Iso(X,d)$. We conclude from \eqref{eqn:iso-trans} that there is $a \in G$ such that
$u(x) = a.x$ for any $x \in X$. But then for every $g \in G$,
$$f(g) = \Phi^{-1}(u_0(\Phi(g))) = \Phi^{-1}(u(g.b)) = \Phi^{-1}(ag.b) = ag$$
and hence $f = L_a \in \LLl(G)$, a contradiction.
\end{proof}

We are now ready to give a short

\begin{proof}[Proof of \THM{trans}]
Implication `(ii)$\implies$(i)' is immediate; \PRO{trans} and \LEM{Hp} show that (iii) is implied
by (i); and applications of \THM{L(G)}, \LEM{Hp} and \COR{iso} show that (ii) follows from
(iii).\par
Finally, the remainder of the theorem is a consequence of \LEM{Hp} and the equivalence between
points (ii) and (iii) (of the theorem): every group $G$ of the form $\dOUBLE{H}$ (where $H$ is
non-Boolean) is solvable (since $\tilde{H}$ is normal Abelian and $G / \tilde{H}$ is Abelian
as well), disconnected ($\tilde{H}$ is open) and has nontrivial Boolean center (by \LEM{Hxp}).
\end{proof}

Taking into account Theorems~\ref{thm:trans}, \ref{thm:L(G)} and \LEM{notin}, the following question
arises:

\begin{prb}{1}
Does every metrizable iso-group $G$ admit a compatible left invariant metric $d$ such that
$\Iso(X,d)$ consists precisely of all left translations of $G$~?
\end{prb}

\end{document}